\documentclass[a4paper,11pt,english,reqno]{amsart}
\usepackage{amsmath,amssymb,amsfonts,dsfont,amsthm}    
\usepackage[utf8]{inputenc}
\usepackage[english]{babel}
\usepackage[T1]{fontenc} 
\usepackage{graphicx}
\usepackage{float}
\usepackage[a4paper]{geometry} 
\usepackage{enumerate}
\usepackage{ifpdf}

\newcommand{\erw}{\mathds{E}}
\newcommand{\e}{\mathrm{e}}

\newcommand{\tr}{\mathrm{tr\,}}
\newcommand{\supp}{\mathrm{supp\,}}

\newcommand{\Ima}{\mathrm{Im\,}}
\newcommand{\Rea}{\mathrm{Re\,}}
\newcommand{\dist}{\mathrm{dist\,}}
\newcommand{\perm}{\mathrm{perm\,}}
\newcommand{\mO}{\mathcal{O}}
\newcommand{\C}{\mathds{C}}
\newcommand{\R}{\mathds{R}}
\newcommand{\LeftEqNo}{\let\veqno\leqno}

\newtheorem{thm}{Theorem}
\newtheorem{cor}[thm]{Corollary}
\newtheorem{prop}[thm]{Proposition}
\newtheorem{lem}[thm]{Lemma}
\newtheorem{defn}[thm]{Definition}
\newtheorem{rem}[thm]{Remark}
\newtheorem{hypo}[thm]{Hypothesis}
\numberwithin{equation}{section}
\usepackage{hyperref}
\hypersetup{pdfborder=0 0 0, 
	    colorlinks=true,
	    citecolor=blue,
	    linkcolor=blue,
	    urlcolor=red,
	    pdfauthor={Martin Vogel}
	   }
\usepackage{pdfpages}  
\title[Eigenvalue correlation of non-selfadjoint random operators]{Two point eigenvalue correlation for a class of non-selfadjoint operators 
under random perturbations}
\author{Martin Vogel} 
\address[Martin Vogel]{D\'epartement de Math\'ematiques - UMR 8628 CNRS, 
Universit\'e Paris-Sud, B\^atiment 425, F-91405 Orsay Cedex.}
\email{martin.vogel@math.u-psud.fr}
\begin{document}
\maketitle
\begin{abstract}
We consider a non-selfadjoint $h$-differential model operator $P_h$ in the 
semiclassical limit ($h\rightarrow 0$) subject to random perturbations with 
a small coupling constant $\delta$. Assume that 
$\e^{-\frac{1}{Ch}} < \delta \ll h^{\kappa}$ for constants $C,\kappa>0$ 
suitably large. Let $\Sigma$ be the closure of the range of the principal 
symbol.
\par
We study the $2$-point intensity measure of the random point process of 
eigenvalues of the randomly perturbed operator $P_h^{\delta}$ and prove 
an $h$-asymptotic formula for the average $2$-point density of eigenvalues.
With this we show that two eigenvalues of $P_h^{\delta}$ in the interior 
of $\Sigma$ exhibit close range repulsion and long range decoupling.
  \vskip.5cm
  \par\noindent \textbf{R{\'e}sum{\'e}}
Nous consid\'erons un op\'erateur diff\'erentiel non-autoadjoint 
$P_h$ dans la limite semiclassique ($h\rightarrow 0$) soumis 
\`a de petites perturbations al\'eatoires. De plus, nous imposons 
que la constant de couplage $\delta$ v\'erifie 
$\e^{-\frac{1}{Ch}} < \delta \ll h^{\kappa}$ pour certaines constantes 
$C,\kappa>0$ choisies assez grandes. Soit $\Sigma$ l'adh\'erence de l'image 
du symbole principal de $P_h$.
\par
Dans cet article, nous donnons une formule $h$-asymptotique pour la 
$2$-points densit\'e des valeurs propres en \'etudiant la mesure de 
comptage al\'eatoire des valeurs propres \`a l'int\'erieur de $\Sigma$. 
En \'etudiant cette densit\'e, nous prouvons que deux valeurs propres 
sont r\'epulsives \`a distance courte et ind\'ependantes \`a long 
distance. 
\end{abstract}
\section{Introduction}
\label{sec:introduction}
It is well known that the norm of the resolvent of non-normal operators 
can be very large even far away from the spectrum. Consequently, 
the spectrum of such operators can be highly unstable even under 
tiny perturbations, cf \cite{Da97,Da99,Da07,Pr08,Tr97}.
A way to quantify this zone of spectral instability is given by the 
$\varepsilon$-pseudospectrum. Following the work of L.N.~Trefethen and M.~Embree \cite{TrefEmbr}, 
the $\varepsilon$-pseudospectrum of a closed linear operator $A$ on 
a Banach space $X$ is defined by
\begin{equation*}
     \sigma_{\varepsilon}(A) 
     : = \left\{
	         z\in\mathds{C}\backslash\sigma(A);
		~\lVert (z-A)^{-1} \rVert > \frac{1}{\varepsilon} 
	 \right\}
	 \cup\sigma(A),
    \end{equation*}
where $\sigma(A)$ denotes the spectrum of $A$. Equivalently,
\begin{equation}
\label{pssp}
     \sigma_{\varepsilon}(A)  
     = \bigcup_{\substack{B\in\mathcal{B}(X) \\ 
     \lVert B\rVert < \varepsilon}}
	\sigma(A+B).
\end{equation}
In view of \eqref{pssp} it is natural to study the spectrum of such operators 
under small random perturbations. One line of recent interest 
has focused on the case of elliptic (pseudo-)differential operators 
subject to small random perturbations:
\par 
A series of papers by W.~Bordeaux-Montrieux, M.~Hager and J.~Sj\"ostrand 
\cite{Ha06,BM,Ha06b,BoSj09,HaSj08,SjAX1002,Sj08,Sj09} established a probabilistic 
Weyl law in the interior of the pseudospectrum  for a large class of elliptic 
(pseudo-)differential operators subject to small random perturbations in the semiclassical or 
high energy limit. 
Furthermore, a similar result 
has been obtained by T. Christiansen and M. Zworski for certain randomly perturbed Toeplitz operators 
in \cite{ZwChrist10}.
\par
In \cite{Vo14}, we considered a class of elliptic semiclassical differential 
operators introduced by M.~Hager \cite{Ha06} and obtained a precise $h$-asymptotic description of the 
average density of eigenvalues in the entire pseudospectrum by studying the first moment 
of linear statistics of the random point process of eigenvalues. In particular, we showed that there is an accumulation of 
eigenvalues in a small neighbourhood of the boundary of the pseudospectrum, leading to 
a break down of the Weyl law. 
\par
However, there have not yet been any results concerning the statistical correlation 
between the eigenvalues. The purpose of this paper is, therefore, to study the $2$-point 
eigenvalue correlation in the case of Hager's model operator (cf. \cite{Ha06}): 
\paragraph{\textbf{Hager's model operator}} 
Let $0 < h \ll 1$, we consider on $S^1=\mathds{R}/2\pi\mathds{Z}$ 
the semiclassical operator $P_h:L^2(S^1)\rightarrow L^2(S^1)$ 
given by
  \begin{equation}
  \label{eqn:defnModelOperator}
    P_h := hD_x + g(x), 
    \quad 
    D_x := \frac{1}{i}\frac{d}{dx}, 
    \quad
    g\in\mathcal{C}^{\infty}(S^1;\mathds{C})
  \end{equation}
where we assume that $g\in\mathcal{C}^{\infty}(S^1;\mathds{C})$ is such that 
$\Ima g$ has exactly two critical points and they are non-degenerate, 
one minimum and one maximum, say at $a,b\in S^1$, with 
$\Ima g(a) < \Ima g(b)$. 
\\
\par
We denote the semiclassical principal symbol of $P_h$ by 
  \begin{equation}
  \label{eq_I3}
   p(x,\xi) = \xi + g(x), \quad (x,\xi) \in T^*S^1.
  \end{equation}
The Poisson bracket of $p$ and $\overline{p}$ is given by 
  \begin{equation*}
   \{p,\overline{p}\} = 
	 {p}'_{\xi}\cdot{\overline{p}}'_x - 
	{p}'_x\cdot{\overline{p}}'_{\xi}.
  \end{equation*}
The spectrum of $P_h$ is discrete with simple eigenvalues, given by 
  \begin{equation}\label{ad_b5}
   \sigma(P_h) = \{z\in\mathds{C}:~z=\langle g\rangle +kh,~k\in\mathds{Z}\},
   \quad \langle g\rangle := (2\pi)^{-1} \int_0^{2\pi} g(y)dy.
  \end{equation} 
\paragraph{\textbf{Zone of spectral instability}} 
For semiclassical pseudo-differential operators there are various ways to 
quantify the zone of spectral instability: Following \cite{Da99,Zw01,NSjZw04}, 
we define for $p$ as in \eqref{eq_I3}
    \begin{equation*}
     \Sigma:= \overline{p(T^{*}S^1)} \subset \mathds{C}.
    \end{equation*}
\noindent In the case of (\ref{eqn:defnModelOperator}) and 
(\ref{eq_I3}) $p(T^{*}S^1)$ is already closed 
due to the ellipticity of $P_h$. Next, 
for $z\in\mathring{\Sigma}$, consider the equation $z=p(x,\xi)$. 
It has precisely two solutions 
$\rho_{\pm}:=(x_{\pm},\xi_{\pm})$ where $x_{\pm}$ are given by
\begin{equation}
 \label{eqn:xpm}
 \Ima g(x_{\pm})  = \Ima z, ~ 
 \text{with }
 \pm \Ima g'(x_{\pm}) < 0 
\end{equation}
and $\xi_{\pm} = \Rea z - \Rea g(x_{\pm})$. Since we have 
that $\{ \Rea p,\Ima p\}(\rho_+(z))<0$ for all 
$z\in\Omega\Subset\mathring{\Sigma}$ it follows from the work 
of N.~Dencker, J.~Sj\"ostrand and M.~Zworski \cite{NSjZw04}
that we can construct $h^{\infty}$-quasimodes $u\in L^2(S^1)$ of $P_h$ 
with semiclassical wave front set $\mathrm{WF}_h(u) = \{\rho_+(z)\}$ 
(in the case of \eqref{eqn:defnModelOperator} we can even construct 
exponentially accurate quasimodes even though it is not analytic cf \cite{Ha06,Vo14}).
We recall that for $v=v(h)$, $\lVert v\rVert_{L^2}(S^1)= \mO(h^{-N})$, for some 
fixed $N$, the semiclassical wave front set of $v$ is defined by 
\begin{equation*}
 \mathrm{WF}_h(v):=\complement
 \left\{(x,\xi)\in T^*S^1 : \exists 
 a\in\mathcal{S}(T^* S^1),~a(x,\xi)=1,~
 \lVert a^wv\rVert_{L^2}(S^1)=\mO(h^{\infty})
 \right\}
\end{equation*}
where $a^w$ denotes the Weyl quantization of $a$. 
\par
Alternatively, it has been shown in 
\cite{Ha06,SjAX1002,Vo14} that for all $\Omega\Subset\mathring{\Sigma}$ and all 
$z\in\Omega$ 
\begin{equation*}
 \lVert (P_h - z)^{-1}\rVert \geq C_1\e^{\frac{1}{C_2 h}},
\end{equation*}
with $C_1,C_2>0$ constants that only depend on $\Omega$.
This implies that such an $\Omega$ is inside the 
$\e^{-1/Ch}$-pseudospectrum of $P_h$.  
\\
\par
Next, by the natural projection 
$\Pi:\mathds{R}\rightarrow S^1=\mathds{R}/2\pi\mathds{Z}$ and a 
slight abuse of notation we identify the points $x_{\pm},a,b\in S^1$ 
with points $x_{\pm},a,b\in\mathds{R}$ such that 
$x_- -2\pi < x_+ < x_-$ and $b-2\pi < a < b$. Furthermore, we will 
identify $S^1$ with the interval $[b-2\pi,b[$. 
\paragraph{\textbf{Adding a random perturbation}}
We are interested in the following random perturbation of 
$P_h$:
  \begin{equation}
  \label{eqn:DefPertP}
    P_h^{\delta} := P_h +\delta Q_{\omega}, \quad 0\leq \delta \ll 1,
  \end{equation}
where $Q_{\omega}$ is an integral operator $L^2(S^1)\to L^2(S^1)$ of 
the form
\begin{equation}\label{defn:PertIntOp}
    Q_{\omega} u(x) := \sum\limits_{|j|,|k|\leq \left
    \lfloor\frac{C_1}{h}\right\rfloor} \alpha_{j,k} (u|e^k)e^j(x).
  \end{equation}
Here, $\lfloor x\rfloor :=\max\{n\in\mathds{N}:~x\geq n\}$ 
for $x\in\mathds{R}$, $C_1>0$ is large enough, 
$e^k(x):=(2\pi)^{-1/2} \e^{ikx}$, $k\in\mathds{Z}$, and 
$\alpha_{j,k}$ are complex valued independent and identically 
distributed random variables with complex Gaussian distribution law 
$\mathcal{N}_{\mathds{C}}(0,1)$. Since, $Q_{\omega}$ is a compact 
operator, the spectrum of $P_h^{\delta}$ is discrete. 
\par
We recall that a random variable $\alpha$ has complex 
Gaussian distribution law $\mathcal{N}_{\mathds{C}}(0,1)$ if 
\begin{equation*}
 \alpha_*(P(d\omega)) = \frac{1}{\pi}\e^{-\alpha \overline{\alpha}} L(d\alpha)
\end{equation*}
where $L(d\alpha)$ denotes the Lebesgue measure on $\C$ and $\omega$ is the random 
parameter living in the sample space $\mathcal{M}$ of a probability space 
$(\mathcal{M},\mathcal{A},P)$ with $\sigma$-algebra $\mathcal{A}$ and 
probability measure $P$. $\alpha\sim \mathcal{N}_{\mathds{C}}(0,1)$ implies that 
$\alpha$ has expectation $0$ and variance $1$, since
\begin{equation*}
 \erw[ \alpha ] = 0, \quad \text{and} \quad  \erw\left[|\alpha|^2\right] = 1.
\end{equation*}
Here, $\erw[\cdot]$ denotes the expectation. The Markov inequality 
yields that for $C>0$ 
large enough
\begin{equation}\label{eq_ad5.1}
   \lVert Q_{\omega}\rVert_{\mathrm{HS}} \leq \frac{C}{h}, 
   \quad
   \text{with probability } 
   \leq 1 - \e^{-\frac{1}{h^2}}.
  \end{equation}
This, has been obtained as well by W.~Bordeaux-Montrieux in \cite{BM}.
Hence, we restrict our probability space to a open ball $B(0,R)\Subset\C^N$, 
with $N:=(2\left\lfloor\frac{C_1}{h}\right\rfloor +1)^2$, of radius $R=C/h$ 
and centered at $0$, to obtain a uniform (in the random variables) bound 
on $Q_{\omega}$.
\\
\par
In this paper we are interested in the eigenvalues of 
$P_h^{\delta}$ in the interior of the pseudospectrum. Therefore, 
we make the following assumptions on $\Omega\Subset\Sigma$:
\begin{hypo}\label{Hyp:H6}
 We assume that there exists a $C>1$ such that 
\begin{equation}\label{eq_i35}
    \Omega\Subset\mathring{\Sigma} ~ 
    \text{is open, convex, relatively compact and simply connected with } 
    \dist(\Omega,\partial\Sigma) > \frac{1}{C}.
 \end{equation}
\end{hypo}
It will be very useful to give bounds on the coupling constant $\delta$ 
in terms of the imaginary part of the action between $\rho_+(z)$ and 
$\rho_-(z)$, $z\in\Omega$ as in \eqref{eq_i35} (cf \eqref{eqn:xpm}), 
defined by:
\begin{equation}\label{def_act}
   S:=\min\left( \Ima \int_{x_+}^{x_-}(z-g(y))dy,\Ima \int_{x_+}^{x_--2\pi}(z-g(y))dy \right).
\end{equation}
\begin{hypo}\label{Hyp:H7}
 The coupling constant $\delta>0$ in (\ref{eqn:DefPertP})satisfies  
 \begin{equation}\label{Hyp:Delta}
    \delta := \delta(h) 
	:= \sqrt{h}\e^{-\frac{\epsilon_0(h)}{h}}
 \end{equation}
 with $ \left(\kappa - \frac{1}{2}\right)h\ln (h^{-1}) +Ch \leq \epsilon_0(h) < 
 \min_{z\in\overline{\Omega}}S(z)/C$ 
 for some $\kappa>52/10$ and $C>0$ large and where the last inequality 
 is uniform in $h>0$. Equivalently, $\delta$ satisfies the inequality
 \begin{equation*}
   \sqrt{h}\exp\left\{-\frac{\min_{z\in\overline{\Omega}}S(z)}{Ch}\right\} 
    <  \delta \ll h^{\kappa}.
  \end{equation*}
\end{hypo}
\begin{rem}
We chose these hypotheses because the aim of this paper is to treat the 
two-point eigenvalue density and correlation in the interior of the 
pseudospectrum. Hypotheses \ref{Hyp:H6} and \ref{Hyp:H7} prevent us from reaching 
the pseudospectral boundary since either we need to allow for 
sufficiently small coupling constants which would bring the boundary 
of the pseudospectrum in the interior of $\Omega$ (with 
$\dist(\Omega,\partial\Sigma)>1/C$ for some $C>0$),	 
or we need to allow sets $\Omega\Subset\Sigma$ with 
$\dist(\Omega,\partial\Sigma)\geq Ch^{2/3}$. 
  The two-point interaction close to the pseudospectral boundary remains 
  an interesting open problem.
\end{rem}
\section{Main Results}
\label{sec:MainResult}
We are interested in the $2$-point 
correlation of eigenvalues of the perturbed operator 
$P_h^{\delta}$. Therefore, we study the 
$2$-point intensity measure $\nu$, given by 
\begin{align}
 \label{eqn:NU}
  \erw \left[
     \sum_{\substack{z,w\in\sigma(P_h^{\delta}) \\
		     z\neq w}}
     \varphi(z,w)\mathds{1}_{B(0,R)}\right]
  = \int_{\C^2}\varphi(z,w) d\nu(z,w), 
  \quad 
  \varphi\in \mathcal{C}_0(\Omega^2).
\end{align}
\begin{rem} 
The above approach is more classical in the study of zeros of 
random polynomials and Gaussian analytic functions; we refer 
the reader to the works of B. Shiffman and S. Zelditch 
\cite{SZ03,ShZe08,ShZe99,Sh08}, M. Sodin \cite{So00} an the book 
\cite{HoKrPeVi09} by J. Hough, M. Krishnapur, Y. Peres and 
B. Vir\'ag.
\end{rem}
We begin by giving an $h$-asymptotic formula for its Lebesgue 
density valid at a distance $\gg h^{3/5}$ from the diagonal. For 
$\Omega$ as in \eqref{eq_i35} and $C_2 >0$, we define the set 
\begin{equation}\label{eq_i27}
 D_h(\Omega,C_2) := \{(z,w)\in\Omega^2;~ |z-w| \leq C_2 h^{3/5} \}.
\end{equation}
Before, we state the main result, let us recall that has been shown in 
\cite{Ha06,Vo14} that the direct image $p_*(d\xi\wedge dx)$ 
of the symplectic volume form $d\xi\wedge dx$ on $T^*S^1$ is absolutely continuous 
with respect to the Lebesgue measure on $\C$ and its Radon-Nikodym derivative is 
\begin{equation}\label{eq_i26}
 \sigma(z):=\frac{ p_*(d\xi\wedge dx)}{L(dz)} 
  = \left(\frac{2i}{\{p,\overline{p}\}(\rho_+(z))}
  +\frac{2i}{\{\overline{p},p\}(\rho_-(z))}
  \right).
\end{equation}
\begin{thm}\label{thm_H2}
 Let $\Omega\Subset\Sigma$ be as in \eqref{eq_i35}. Let $\delta>0$ be 
 as in Hypothesis \ref{Hyp:H7}. Let $\nu$ be the 
 measure defined in \eqref{eqn:NU}
 and let $\sigma(z)$ be as in \eqref{eq_i26}. Then, for 
 $|z-w|\leq 1/C$ with $C>1$ large enough, there exist smooth functions
 \begin{itemize}
  \item  $\sigma_h(z,w) = \sigma\left(\frac{z+w}{2}\right) + \mO(h)$,
  \item  $ K(z,w;h) = \sigma_h(z,w)\frac{|z-w|^2}{4h}(1 + \mO(|z-w|+h^{\infty}))$,
  \item  $D^{\delta}(z,w;h) = \frac{\Lambda(z,w)}{(2\pi h)^{2}\left(1- \e^{-2K}\right)}
    \left(1 
      + \mO\!\left(\delta h^{-\frac{8}{5}}\right)\right)
    + \mO\!\left(\e^{-\frac{D}{h^2}}\right)$, with 
  \begin{equation*}
  \begin{split}
  \Lambda(z,w;h) 
   = &\sigma_h(z,z)\sigma_h(w,w) + \sigma_h(z,w)^2(1 + \mO(|z-w|))
	      \e^{-2K}
	  \\
	  &+ 
	  \frac{\sigma_h(z,w)^2(1 + \mO(|z-w|)) }{\e^{K}\sinh(K)}
	  \left(2K^2 \coth(K)  - 4K  \right) 
	  +\mO\!\left(h^{\infty} + \delta h^{-\frac{32}{10}}\right)
 \end{split}
 \end{equation*} %
 \end{itemize}
 and there exists a constant $c>0$ such that 
 for all  $\varphi\in\mathcal{C}_0^{\infty}(\Omega^2\backslash D_h(\Omega,c))$  
 \begin{equation*}
    \int_{\C^2} \varphi(z,w) d\nu(z,w)
    = 
    \int_{\C^2} \varphi(z,w)D^{\delta}(z,w;h) L(d(z,w)).
 \end{equation*}
\end{thm}
By this result we see that in the interior of the pseudospectrum 
the leading terms of the $2$-point density of eigenvalues depends 
only on the symplectic volume form in phase space. This agrees 
very well with previous results of Hager, Sj\"ostrand and Vogel 
\cite{Ha06,SjAX1002,Vo14} saying that in the 
interior of the pseudospectrum the probabilistic and average density 
of eigenvalues depends only on the symplectic volume form.  
\par
Let us stress once more that due to the assumptions on $\Omega$ 
and $\delta$, the formula for the $2$-point density presented in 
Theorem \ref{thm_H2} is not valid close to the pseudospectral boundary. 
However, in view of the results presented in \cite{Vo14}, we would 
expect the $2$-point density to change drastically close to the pseudospectral 
boundary, but for now this remains an open problem.
\begin{rem}\label{remA}
Avoiding $D_h(\Omega,c)$, cf. \eqref{eq_i27}, with the support of the 
test functions $\varphi$ in Theorem \ref{thm_H2} is due to a technical 
difficulty in the proof, since there is some degeneracy due to the 
error terms when $|z-w|$ is too small, see Proposition 
\ref{prop:GramInvert} below. 
\par
However, having a formula for the $2$-point density of eigenvalues outside 
$D_h(\Omega,c)$ is sufficient to include the study of the close range 
correlation between two eigenvalues up to a certain distance, cf. Theorem \ref{prop:H10} 
and \ref{prop:H11}. 
\end{rem}
\subsection{Asymptotic regimes of the density}
Using the formula obtained in Theorem \ref{thm_H2}, we will prove that 
two eigenvalues of $P_h^{\delta}$ exhibit the following interaction:
\begin{thm}\label{prop:H10}
 Under the hypothesis of Theorem \ref{thm_H2}, we have that 
 \begin{itemize}
  \item for $h^{\frac{4}{7}}\ll |z-w| \ll h^{\frac{1}{2}}$ 
         \begin{equation*}
          D^{\delta}(z,w;h) = \frac{\sigma_h^3(z,w)|z-w|^2}{(4\pi)^2 h^3}\left(1+ 
         \mO\!\left(\frac{|z-w|^2}{h}+\delta h^{-\frac{8}{5}}\right)\right);
         \end{equation*}
  \item for $|z-w| \gg (h\ln h^{-1} )^{\frac{1}{2}}$
         \begin{equation*}
           D^{\delta}(z,w;h) = \frac{\sigma(z)\sigma(w) + \mO(h)}{(2h\pi)^{2}}
          \left(1 + \mO\!\left( \delta h^{-\frac{8}{5}}\right)\right).
         \end{equation*}
 \end{itemize}
\end{thm}
Let us give some comments on this result: The fact that we cannot analyze the eigenvalue 
interaction completely up to the diagonal is due to some technical difficulties. In 
the above theorem, two eigenvalues of the perturbed operator 
$P_h^{\delta}$ show the following types of interaction:
\begin{description}
 \item[\textbf{Short range repulsion}] The two-point density decays quadratically in 
       $|z-w|$ if two eigenvalues are too close, and in view of the numerical simulations 
       presented in Section \ref{S_NumSim} we conjecture that this 
       is the case for all $z,w$ as above satisfying $0<|z-w| \ll h^{\frac{1}{2}}$.
 \item[\textbf{Long range decoupling}] If the distance between two eigenvalues is 
       $\gg (h\ln h^{-1} )^{\frac{1}{2}}$ the two-point density is given by the 
       product of two one-point densities (cf. \eqref{eq_i30}). This means that 
       at this distance two eigenvalues are placed in average in an uncorrelated 
       way.
\end{description}
\subsection{$2$-point correlation function}
M. Hager \cite{Ha06} showed, using subharmonic estimates, that, 
with  probability close to $1$, the 
eigenvalues of the perturbed operator $P_h^{\delta}$ contained in $\Omega$ 
(as in Hypothesis \ref{Hyp:H6}) follow a Weyl law, i.e. 
\begin{equation*}
   \#(\sigma(P_h^{\delta})\cap\Omega) \sim 
   \frac{1}{2\pi h}\mathrm{vol}(\{\rho\in T^*S^1; p(\rho)\in\Omega\}).
\end{equation*}
In \cite{Vo14}, we considered the random 
point process given by eigenvalues of $P_h^{\delta}$ : 
  \begin{equation}
  \label{defn:PP}
   \Xi := \sum_{z\in\sigma(P_h^{\delta})}\delta_z.
  \end{equation}
where the eigenvalues are counted according to their multiplicities 
and $\delta_z$ denotes the Dirac-measure at $z$. 
\par
We studied in \cite{Vo14} the first moment of linear statistics of 
$\Xi$ with the random variables $\alpha$ 
restricted to a ball $B(0,R)\subset\C^N$ with $R=C/h$, 
i.e. the measure $\mu_1$ defined by 
\begin{equation*}
 \erw[\Xi(\varphi)\mathds{1}_{B(0,R)}] 
           = \int_{\C}\varphi(z) d\mu_1(z)
\end{equation*}
for all $\varphi\in\mathcal{C}_0(\Omega)$ 
with $\Omega\Subset\Sigma$ such that   
$\dist(\Omega,\partial\Sigma)\gg h^{2/3}$. 
For $\Omega$ as in Hypothesis \ref{Hyp:H6}, 
Theorem 2.11 in \cite{Vo14} implies that 
       \begin{equation}\label{eq_i30.5}
        \erw[\Xi(\varphi)\mathds{1}_{B(0,R)}] = 
        \int \varphi(z)d(z;h)L(dz), \quad \forall
        \varphi\in\mathcal{C}_0(\Omega),
       \end{equation}
       where 
       \begin{equation}\label{eq_i30}
        d(z;h) = \frac{1}{2\pi h} \sigma(z) + \mO(1), 
        \quad \sigma(z) \text{ is as in } \eqref{eq_i26}.
       \end{equation}
In other words, the average density of eigenvalues in $\Omega$ 
is up to first order determined by symplectic volume form in phase space.
\par
It follows from \eqref{eq_i30}, \eqref{eq_i26} that for $h>0$ small enough 
$d(z;h)>0$ for all $z\in\Omega$ as in \eqref{eq_i35}. Hence, under the assumptions of 
Theorem \ref{thm_H2}, the $2$-point correlation function of the eigenvalues 
of $P_h^{\delta}$, is well defined and given by 
  \begin{equation*}
   \kappa^{\delta}(z,w;h) :=  \frac{D^{\delta}(z,w;h)}{d(w;h)d(z;h)}.
  \end{equation*}
 \begin{thm}\label{prop:H11}
 Under the hypothesis of Theorem \ref{thm_H2}, we have 
 that for $(z,w)\in\Omega^2\backslash D_h(\Omega,c)$ as 
 in Theorem \ref{thm_H2} that
  \begin{equation*}
  \begin{split}
 \kappa^{\delta}(z,w;h) 
   = &
   \frac{1+\mO(h)}{\left(1- \e^{-2K}\right)}\bigg(
   1+(1 + \mO(|z-w|)) \e^{-2K}+\frac{(1 + \mO(|z-w|)) }{\e^{K}\sinh(K)}
	  \left(2K^2 \coth(K)  - 4K  \right) 
	  \\
	  &
	  +\mO\!\left(h^{\infty} + \delta h^{-\frac{32}{10}}\right) \bigg )
    + \mO\!\left(\e^{-\frac{D}{h^2}}\right).
 \end{split}
 \end{equation*}
 Moreover, we have the following asymptotic behaviour of the 
 $2$-point correlation function $\kappa^{\delta}(z,w;h) $:
 \begin{itemize}
  \item for $h^{\frac{4}{7}}\ll |z-w| \ll h^{\frac{1}{2}}$ 
         \begin{equation*}
          \kappa^{\delta}(z,w;h)  = \frac{\sigma_h(z,w)|z-w|^2}{4 h}\left(1+ 
         \mO\!\left(\frac{|z-w|^2}{h}+\delta h^{-\frac{8}{5}}\right)\right) \ll 1;
         \end{equation*}
  \item for $|z-w| \gg (h\ln h^{-1} )^{\frac{1}{2}}$
         \begin{equation*}
           \kappa^{\delta}(z,w;h) = 1 + \mO( h ).
         \end{equation*}
 \end{itemize}
\end{thm}
In the above Theorem we see that two eigenvalues of $P_h^{\delta}$ 
shows the following behaviour:
\begin{description}
 \item[\textbf{Short range repulsion}] The $2$-point correlation function 
 $\kappa^{\delta}(z,w;h) $ decays quadratically in $|z-w|$ if the distance 
 between $z$ and $w$ is smaller than a term of order $h^{\frac{1}{2}}$. 
 It is thus less likely to find two eigenvalues close together. Furthermore,  
 we see by \eqref{eq_i26} that $\sigma(z)$ grows towards the boundary 
 of $\Sigma$, hence the short range repulsion is weaker for $\Omega$ 
 closer to the boundary of $\Sigma$, as we expected from the numerical 
 simulations presented in \cite{Vo14}, see Figure 4 therein.
 \par
The fact that we cannot analyze close range correlation up to the diagonal 
is due to a degeneracy resulting from error terms, cf. Remark \ref{remA} and the 
proofs of Theorem \ref{prop:H11} and \ref{prop:H10}. However, the conclusions of Theorem 
\ref{prop:H11} allow for the study of the scaling limit of the $2$-point 
correlation function, which yields the limiting local $2$-point statistics 
of eigenvalues of $P_h^{\delta}$, after re-scaling distances between 
eigenvalues to be independent of $h$, cf. Section \ref{se_SL} 
and Corollary \ref{cor1}.
 \item[\textbf{Long range decoupling}] If the distance between $z$ and $w$ is larger 
       than a term of order $(h\ln h^{-1})^{\frac{1}{2}}$, the $2$-point correlation function 
       $\kappa^{\delta}(z,w;h) $ is given up to a small error by $1$. 
       Hence, we see that at these distances two eigenvalues of $P_h^{\delta}$ are up to a 
       small error uncorrelated. 
\end{description}
\begin{rem}
Recall from the discussion after \eqref{eqn:DefPertP} that in 
this paper we focus on the case where the random perturbation is 
given by a random matrix whose entries are independent and 
identically distributed complex Gaussian random variables. 
As supported by numerical experiments 
(cf. Section \ref{S_NumSim}) we expect Theorem \ref{prop:H11} to 
hold for a much more general class of random variables as long as 
the perturbation is of the form \eqref{defn:PertIntOp}.
\par
Questions concerning the universality of the result of Theorem \ref{prop:H11} 
in the case of small random perturbations of a more general class of 
(pseudo-)differential operators are currently under investigation by the author. 
We expect the type of perturbation (by random matrix or by random potential) 
rather than its probability distribution 
to be decisive, since although in both cases we can obtain a probabilistic Weyl law for 
the eigenvalues, see \cite{Ha06b,HaSj08,Sj08,Sj09}, numerical experiments 
suggest that the $2$-point correlation functions in both cases differ. 
\end{rem}
\subsection{Scaling limit of the $2$-point correlation function}\label{se_SL}
We can use Theorem \ref{prop:H11} to study the limiting local $2$-point correlation 
function in the interior of the pseudospectrum. Therefore, let $\Omega$ be as in \eqref{eq_i35}, and fix a $z_0\in\Omega$. Let $d(z;h)$ be as in \eqref{eq_i30} and set $d_0:=d(z_0;h)\asymp h^{-1}$. Let $\kappa^{\delta}(z,w;h)$ be as in Theorem \ref{prop:H11}, let $W$ be a compact 
subset of $\{(z,w) \in\C^2; z\neq w\}$ and consider, for $h>0$ small enough, 
\begin{equation}\label{eq_b1}
	\widetilde{\kappa}_h(z,w) : = \kappa^{\delta}(z_0+d_0^{-1/2}z,z_0+d_0^{-1/2}w;h), 
	\quad (z,w)\in W.
\end{equation}
This is well defined since for $h>0$ small enough $(z_0+d_0^{-1/2}z,z_0+d_0^{-1/2}w) \in \Omega^2\backslash D_h(\Omega,c)$ (see Theorem \ref{prop:H11}) for all $(z,w)\in W$. 
Similarly to the discussion before Theorem \ref{prop:H11}, we notice that we can view 
$\widetilde{\kappa}_h(z,w)$ as the $2$-point correlation function of the random point 
process of the re-scaled eigenvalues of $P_h^{\delta}$:
  \begin{equation}\label{eq_b2}
   \widetilde{\Xi} := \sum_{z\in\sigma(P_h^{\delta})}\delta_{(z-z_0)d_0^{1/2}}.
  \end{equation}
When considering the first moment of linear statistics of $\widetilde{\Xi}$ 
(cf. \eqref{eq_i30.5}, \eqref{eq_i30}) we see that we have re-scaled distances 
in such a way that the leading order of the average density of eigenvalues 
(after re-scaling) is independent of $h$. 
\par
From Theorem \ref{prop:H11} we obtain the following result.
\begin{cor}\label{cor1} For any compact $W\Subset \{(z,w) \in\C^2; z\neq w\}$ we have that
\begin{equation*}
   \lim\limits_{h\to0^+}\widetilde{\kappa}_h(z,w) = \kappa\left(\frac{\pi}{2}|z-w|^2\right), 
   \quad (z,w)\in W, 
\end{equation*}
uniformly on $W$, where 
\begin{equation}\label{eq_b3}
	\kappa(t) = \frac{(\sinh^2t+t^2)\cosh t - 2t\sinh t}{\sinh^3t}, \quad t=\frac{\pi}{2}|z-w|^2.
\end{equation}
\end{cor}
Let us remark that the scaling limit $2$-point correlation function is independent of $z_0$ and 
depends only on the distance between points. Similar 
to the asymptotic regimes presented in Theorem \ref{prop:H11}, we obtain short 
range repulsion between two re-scaled eigenvalues of $P_h^{\delta}$ since, by Taylor 
expansion, $\kappa(t)=t(1+\mO(t^2))$, as $t\to 0^+$, which shows that the scaling 
limit $2$-point correlation function decays quadratically for small distances between 
$2$ points (see Figure \ref{fig_4}).
\par
Similarly, we have long range decorrelation between two re-scaled eigenvalues of $P_h^{\delta}$ since, by Taylor expansion, $\kappa(t)=1+\mO(t^2\e^{-2t})$, as $t\to +\infty$.
\par
The same scaling limit  $2$-point correlation function $\kappa$ has been found as well by 
J.H.~Hannay \cite{Han95} in the case of zeros of certain random polynomials and by 
P.~Bleher, B.~Shiffman and S.~Zelditch \cite{BlShZe00} in the case of random 
holomorphic sections of the Nth power of a positive Hermitian line bundle over a 
compact complex manifold. 
\begin{figure}[h]
         \centering 
         \includegraphics[width=0.6\textwidth]{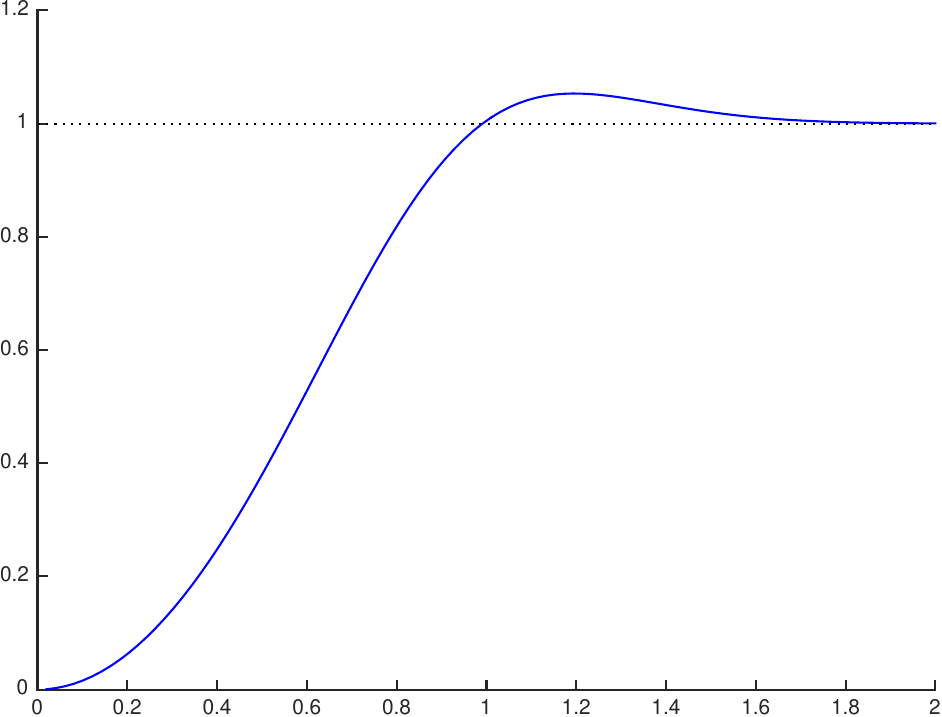}
        \caption{The pair correlation function $\kappa(\frac{\pi}{2}|z-w|^2)$, as a 
        		function of the distance, as given in Corollary \ref{cor1}. A similar figure  
		can be found in \cite{Han95,BlShZe00}}
        \label{fig_4}
\end{figure}
\subsection{Numerical simulation}\label{S_NumSim}
To illustrate Theorem \ref{prop:H11} and Corollary \ref{cor1}, we have numerically determined 
the $2$-point correlation function of the eigenvalues of a discretisation of the operator 
$hD_x + \e^{-ix}$, with $h=2\cdot 10^{-3}$, perturbed with a random complex Gaussian matrix 
with coupling constant $\delta=2\cdot 10^{-12}$. The left hand side of Figure \ref{fig_2} shows one realisation of these 
eigenvalues and the region, where we determine the 2-point correlation function, staying 
inside of the pseudospectrum and away from the effects caused by the finite dimensional 
approximation of the operator. The right hand side shows the eigenvalues in the region of 
interest after re-scaling by $z\mapsto d(0;h)^{1/2}z$, as in \eqref{eq_b2}.
\par
Figure \ref{fig_3} compares the scaling limit pair correlation function 
$\kappa(\frac{\pi}{2} |z-w|^2)$ 
(as a function of the distance) to the histogram data of the numerically obtained 
re-scaled $2$-point correlation function, which corresponds to $\widetilde{\kappa}_h(z,w)$ 
as in \eqref{eq_b2}, obtained from the numerically simulated re-scaled eigenvalues  
depicted on the right hand side of Figure \ref{fig_2} and averaged over 200 
realisations of Gaussian random matrices. 
\par
\begin{figure}[ht]
         \centering
         \begin{minipage}[b]{0.52\linewidth}
         \centering
         \includegraphics[width=\textwidth]{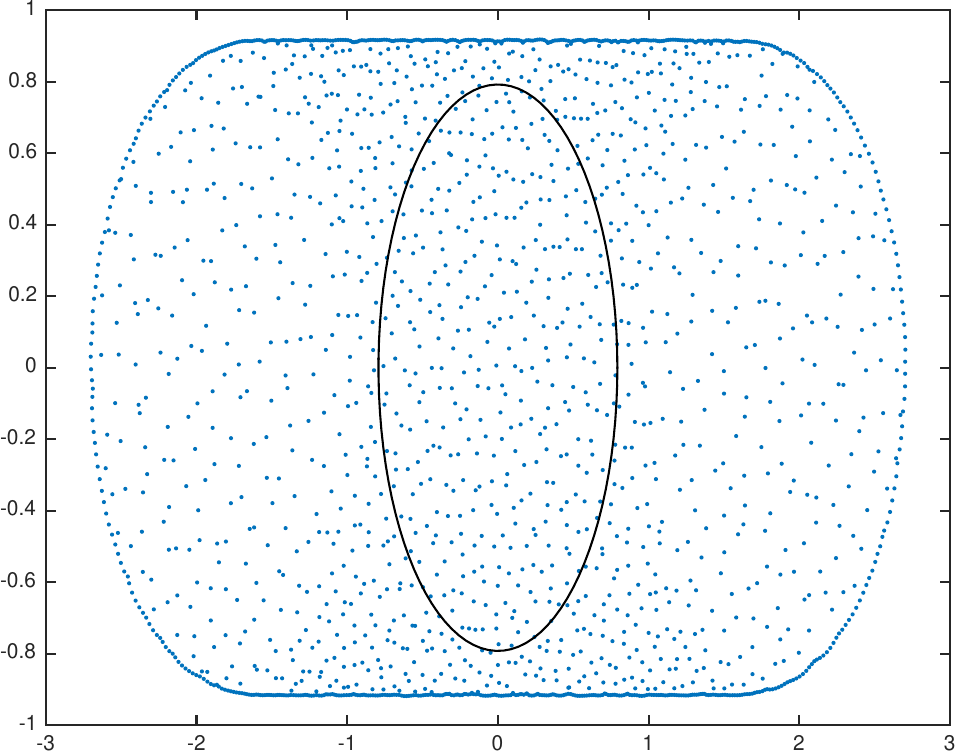}
        \end{minipage}
        \hspace{0.025\linewidth}
        \begin{minipage}[b]{0.415\linewidth}
         \centering 
         \includegraphics[width=\textwidth]{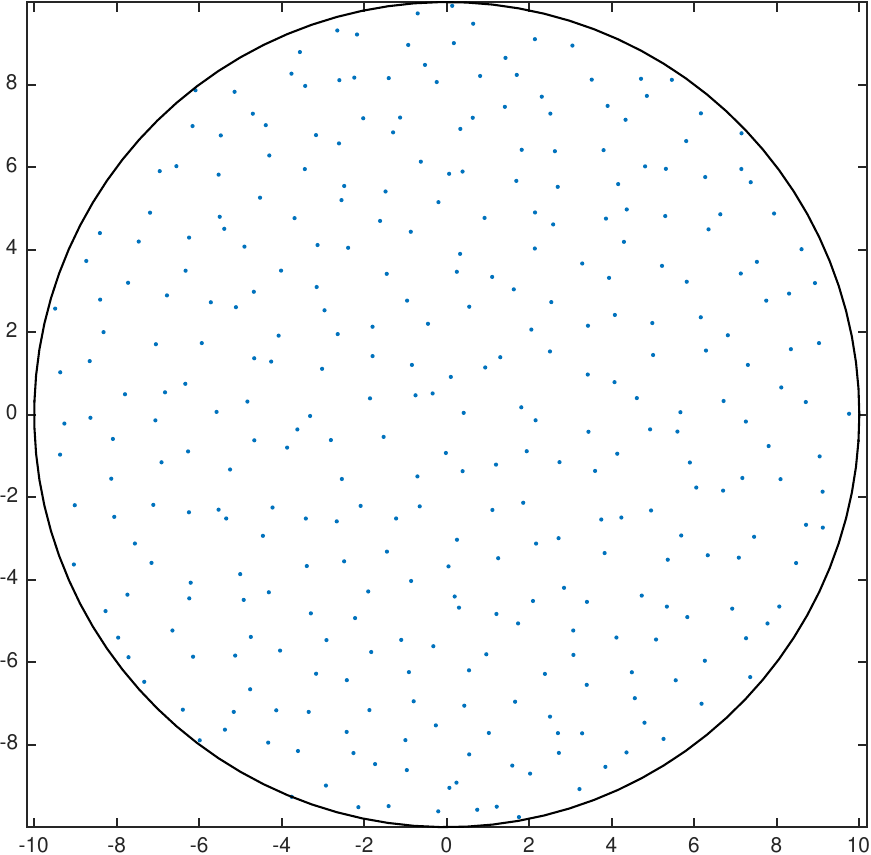}
        \end{minipage}
        \caption{On the left hand side we present the spectrum of the discretisation 
        of $hD + \exp(-ix)$, $h=2\cdot 10^{-3}$, (approximated by a $2001\times 
        2001$-matrix) perturbed with a random complex Gaussian matrix with 
        coupling constant $\delta=2\cdot 10^{-12}$. The black disc indicates the 
        region where we determine the $2$-point correlation function  presented 
        on the right hand side of Figure\ref{fig_3}. The right hand side shows the same 
	disc after re-scaling by $d(0;h)\approx (\pi h)^{-1}$, the average density of eigenvalues 
	at $0$, cf. \eqref{eq_i30}.}
        \label{fig_2}
\end{figure}
\begin{figure}[ht]
         \centering
         \includegraphics[width=0.6\textwidth]{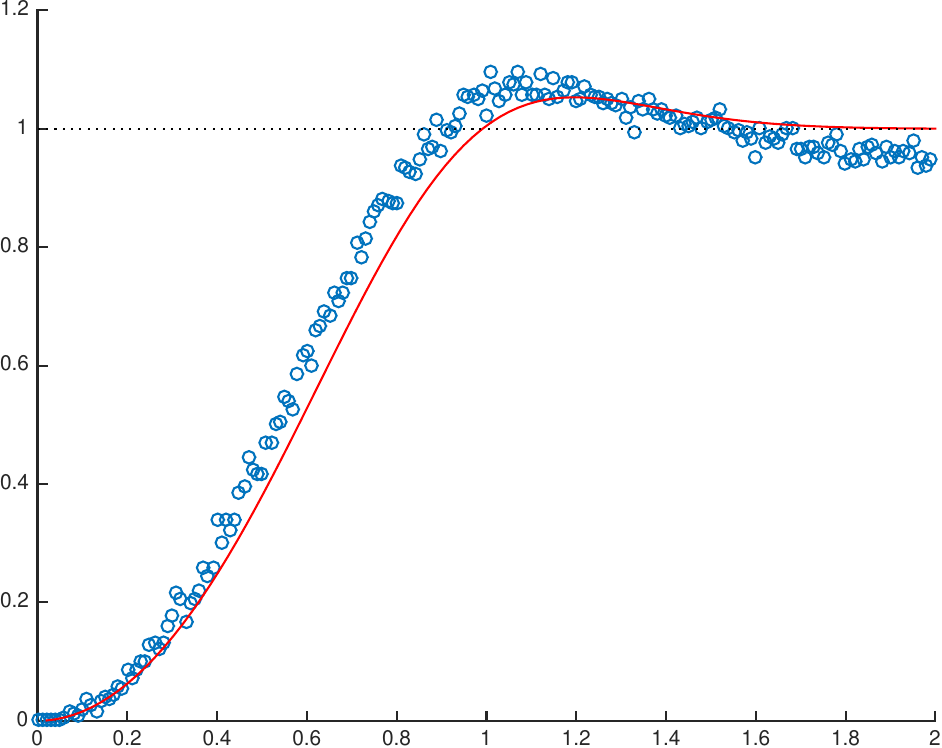}
        \caption{The red line shows the scaling limit pair correlation function 
        		$\kappa(\frac{\pi}{2}|z-w|^2)$, as a function of the distance, and the 
		blue circles sow the histogram data corresponding to the the numerically 
		    determined $2$-point correlation function given by the re-scaled 
		    eigenvalues of the random matrix presented in Figure \ref{fig_2}.}
        \label{fig_3}
\end{figure}
We see that up to a small error the numerically determined re-scaled $2$-point correlation 
function is given by its scaling limit, showing decorrelation for large distances and quadratic 
decay, as the distance between two points goes to zero, confirming the conclusions of 
Theorem \ref{prop:H11} and Corollary \ref{cor1}. 
\\
\par
Finally, let us remark, that when running numerical experiments with a perturbation given by a complex random matrix whose entries follow a uniform or a Poisson distribution instead of a complex Gaussian one, we are able to produce the same results as presented in Figure \ref{fig_3}, suggesting that the 
results of Theorem \ref{prop:H11} and Corollary \ref{cor1} are valid for random perturbations 
of the form \ref{defn:PertIntOp} given by a more general class of random variables.  
\paragraph{\textbf{Organisation of this paper}}
In Section \ref{sec:GrushinProblem} we recall some results from \cite{Vo14} needed for this paper 
and we provide a formula (cf. Proposition \ref{prop:2ptCorrelation}) representing the two-point density of eigenvalues in terms of the permanent and determinant of certain correlation matrices. This formula will be proved in Section 
\ref{sec:CorrelationFormula}. Section \ref{sec:StPh} provides a detailed description of the 
elements of these matrices using the method of stationary phase. Section \ref{susec:Gramian} 
then exploits the main result of Section \ref{sec:StPh} to obtain precise formulas and estimates for the permanent and determinant of the matrices appearing in Proposition \ref{prop:2ptCorrelation}.
Section \ref{sec:PMR} states the proofs of the main results of this paper. 

\paragraph{\textbf{Notation}}
We will use the standard 
scalar products on $L^2(S^1)$ and $\mathds{C}^N$ defined by 
  \begin{equation*}
   (f|g) := \int_{S^1} f(x) \overline{g}(x)dx,
   \quad f,g\in L^2(S^1),
  \end{equation*}
and 
  \begin{equation*}
   (X|Y) := \sum_{i=1}^N X_i \overline{Y}_i,
   \quad X,Y\in\mathds{C}^N.
  \end{equation*}
Throughout this work we shall denote the Lebesgue measure 
on $\mathds{C^d}$ by $L(dz)$; denote $d(z):=\dist(z,\partial\Sigma)$; 
work with the convention that when we write $\mathcal{O}(1)^{-1}$ 
then we mean implicitly an arbitrarily small positive constant; 
denote by $f(x) \asymp g(x)$ that 
there exists a constant $C>0$ such that $C^{-1}g(x) \leq f(x) \leq Cg(x)$. \\
\paragraph{\textbf{Acknowledgments}}
I would like to thank very warmly my thesis advisor Johannes Sj\"ostrand for 
reading the first draft of this work and for his kind and enthusiastic manner 
in supporting me along the way. I would also like to thank sincerely my thesis 
advisor Fr\'ed\'eric Klopp for his kind and generous support. I am also grateful 
to Steve Zelditch for pointing out important references. I would also like to thank the 
referees and the editor for the remarks that have helped to improve the presentation of 
this paper.
\section{A formula for the two-point intensity measure}
\label{sec:GrushinProblem}
In this section we will give a short review of a well-posed 
Grushin problem for the perturbed operator $P_h^{\delta}$ 
which has already been used in \cite{Vo14,SjAX1002}. 
We will then employ the resulting effective Hamiltonians to 
derive a formula for the two-point intensity measure defined 
in \eqref{eqn:NU}. 
\par
We recall that we always suppose that $\Omega\Subset\mathring{\Sigma}$ 
is such that Hypothesis \ref{Hyp:H6} is satisfied, if nothing else 
is specified. 
\subsection{Grushin Problem}
We begin by giving a short refresher on Grushin problems. They 
have become an important tool in microlocal analysis and are 
employed with great success in a vast number of works. As 
reviewed in \cite{SjZw07}, the central idea 
is to set up an auxiliary problem of the form 
\begin{equation*}
 \begin{pmatrix}
  P(z) & R_- \\ 
  R_+ & 0 \\
 \end{pmatrix}
 :
 \mathcal{H}_1\oplus \mathcal{H}_- 
 \longrightarrow \mathcal{H}_2\oplus \mathcal{H}_+,
\end{equation*}
where $P(z)$ is the operator under investigation and $R_{\pm}$ are 
suitably chosen. We say that the Grushin problem is well-posed 
if this matrix of operators is bijective. If 
$\dim\mathcal{H}_-  = \dim\mathcal{H}_+ < \infty$, on typically 
writes 
\begin{equation*}
 \begin{pmatrix}
  P(z) & R_- \\ 
  R_+ & 0 \\
 \end{pmatrix}^{-1}
 =
 \begin{pmatrix}
  E(z) & E_{+}(z) \\ 
  E_{-}(z) & E_{-+}(z) \\
 \end{pmatrix}.
\end{equation*}
The key observation goes back to the Shur complement formula or, 
equivalently, the Lyapunov-Schmidt bifurcation method, 
i.e. the operator $P(z): \mathcal{H}_1 \rightarrow \mathcal{H}_2$ 
is invertible if and only if the finite dimensional matrix 
$E_{-+}(z)$ is invertible and when $E_{-+}(z)$ is invertible, 
we have 
\begin{equation*}
  P^{-1}(z) = E(z) - E_{+}(z) E_{-+}^{-1}(z) E_{-}(z).
\end{equation*}
$E_{-+}(z)$ is sometimes called effective Hamiltonian. 
\\
\par
Next, we give a short reminder of the Grushin Problem used 
to study $P_h^{\delta}$. First, we introduce the following 
auxiliary operators which have already been used by M. Hager 
J. Sj\"ostrand in \cite{HaSj08}. 
\subsection{Two auxiliary operators.}
\label{sec:AuxOpe} 
For $z\in\C$ we consider $Q(z)$ and $\widetilde{Q}(z)$, 
two $z$-dependent elliptic self-adjoint operators from 
$L^2(S^1)$ to $L^2(S^1)$, defined by 
  \begin{align}
   \label{def:AuxOp}
      Q(z):=(P_h-z)^*(P_h-z), \quad
      \tilde{Q}(z):=(P_h-z)(P_h-z)^*
  \end{align}
with natural domains given by 
$\mathcal{D}(Q(z)),\mathcal{D}(\tilde{Q}(z))=H^2_{\mathrm{sc}}(S^1)$. 
Since $S^1$ is compact and these are elliptic, 
non-negative, self-adjoint operators their spectra 
are discrete and contained in the interval 
$[0,\infty[$. Since
  \begin{equation*}
      Q(z)u = 0 \Rightarrow (P_h-z)u=0
  \end{equation*}
it follows that $\mathcal{N}(Q(z)) = \mathcal{N}(P_h-z)$ 
and $\mathcal{N}(\tilde{Q}(z)) = \mathcal{N}((P_h-z)^*)$. 
Furthermore, if $\lambda\neq 0$ is an eigenvalue of $Q(z)$ 
with corresponding eigenvector $e_{\lambda}$ we see that 
$f_{\lambda}:=(P_h-z)e_{\lambda}$ is an eigenvector of 
$\tilde{Q}(z)$ with the eigenvalue $\lambda$. Similarly, 
every non-vanishing eigenvalue of $\tilde{Q}(z)$ is an 
eigenvalue of $Q(z)$ and moreover, since $P_h-z$, $(P_h-z)^*$ 
are Fredholm operators of index $0$ we see that 
$\text{dim}\mathcal{N}(P_h-z)=\text{dim}\mathcal{N}((P_h-z)^*)$. 
Hence the spectra of $Q(z)$ and $\tilde{Q}(z)$ are equal
  \begin{equation}
  \label{eqn:auxillaryOpEigVals}
      \sigma(Q(z))=\sigma(\tilde{Q}(z))
	  =\{t_0^2,t_1^2,\dots\},~0\leq t_j\nearrow\infty.
  \end{equation}
Now consider the orthonormal basis of $L^2(S^1)$ 
  \begin{equation}
  \label{def:e0,e1,...}
    \{e_0,e_1,\dots\}
  \end{equation}
consisting of the eigenfunctions of $Q(z)$.  By the 
previous observations we have 
 \begin{align*}
  (P_h-z)(P_h-z)^*(P_h-z)e_j &= t_j^2(P_h-z)e_j.
 \end{align*}
Thus defining $f_0$ to be the normalized eigenvector of 
$\widetilde{Q}$ corresponding to the eigenvalue $t_0^2$ 
and the vectors $f_j\in L^2(S^1)$, for $j\in\mathds{N}^*$, 
as the normalization of $(P_h-z)e_j$ such that 
  \begin{equation}
  \label{eq_i18}
   (P_h-z)e_j=\alpha_jf_j, \quad (P_h-z)^*f_j
   =\beta_je_j \quad\text{with} ~\alpha_j\beta_j = t_j^2,
  \end{equation}
yields an orthonormal basis of $L^2(S^1)$
  \begin{equation}
    \label{def:f0,f1,...}
      \{f_0,f_1,\dots\}
  \end{equation}
consisting of the eigenfunctions of $\tilde{Q}(z)$. 
Since $\alpha_j = ((P_h-z)e_j|f_j) = (e_j|(P_h-z)^*f_j) = \overline{\beta}_j$ 
we can conclude that $\alpha_j\overline{\alpha}_j=t_j^2$. 
\subsection{A Grushin Problem for the perturbed operator $P_h^{\delta}$}
Following Sj\"ostrand in \cite{SjAX1002}, we us the eigenfunctions of the operators 
$Q$ and $\widetilde{Q}$ (cf \eqref{def:AuxOp}) to create a well-posed Grushin Problem. 
The sequel is taken from \cite{Vo14}, but it originates partly in the works of 
Hager \cite{Ha06}, Bordeaux-Montrieux \cite{BM} and Sj\"ostrand \cite{SjAX1002}.
\begin{prop}\label{prop:GrushUnpertOp}
Let $z\in\Omega \Subset\Sigma$ with $\dist(\Omega,\partial\Sigma)>1/C$ 
and let $\alpha_0,e_0$ and $f_0$ be as in (\ref{eq_i18}). Define
  \begin{align*}
    R_+:~ &H^1(S^1)\longrightarrow\mathds{C}:~ u  \longmapsto (u|e_0),  \notag \\
    R_-:~&\mathds{C}\longrightarrow L^2(S^1):~ u_-  \longmapsto u_-f_0.  
  \end{align*}
Then 
  \begin{equation*}
   \mathcal{P}(z):=\begin{pmatrix}
                    P_h-z & R_- \\ R_+ & 0 \\
                   \end{pmatrix}
   :~ H^1(S^1)\times \mathds{C}\longrightarrow L^2(S^1)\times \mathds{C}
  \end{equation*}
is bijective with the bounded inverse 
  \begin{equation*}
   \mathcal{E}(z) = \begin{pmatrix}
                     E(z) & E_+(z) \\ E_-(z) & E_{-+}(z) \\
                    \end{pmatrix}
  \end{equation*}
where $E_-(z)v = (v|f_0)$, $E_+(z)v_+ = v_+e_0$, 
$E(z)=(P_h-z)^{-1}|_{(f_0)^{\perp}\rightarrow (e_0)^{\perp}}$ 
and $E_{-+}(z)v_+ = -\alpha_0 v_+$. Furthermore, we have 
the estimates for $z\in\Omega$
    \begin{align}
      & \lVert E_-(z)\rVert_{L^2\rightarrow \mathds{C}}, ~
		\lVert E_+(z)\rVert_{\mathds{C}\rightarrow H^1 } 
	  = \mO(1), 
	  \notag\\
      &\lVert E(z)\rVert_{L^2\rightarrow H^1} = \mathcal{O}(h^{-1/2}), 
	  \notag \\
      &|E_{-+}(z)| = \mathcal{O}\!\left(\sqrt{h}\e^{-\frac{S}{h}}\right)
                   =\mathcal{O}\!\left(\e^{-\frac{1}{Ch}}\right);
    \end{align}
\end{prop}
\begin{defn}
\label{def:X} 
For $x\in\mathds{R}$ we denote the integer part of $x$ by  
$\lfloor x\rfloor$. Let 
$C_1>0$ be big enough as above and define 
$N:=(2\lfloor\frac{C_1}{h}\rfloor +1)^2$. 
Let $e_0$ and $f_0$ be as in (\ref{eq_i18}), let
$z\in\Omega\Subset\Sigma$ and let 
$\widehat{e_0}(z;\cdot)$ and $\widehat{f_0}(z;\cdot)$ denote 
the Fourier coefficients of $e_0$ and $f_0$. 
We define the vector 
$X(z)=(X_{j,k}(z))_{|j|,|k|\leq \lfloor\frac{C_1}{h}\rfloor }\in\mathds{C}^{N}$ 
to be given by 
  \begin{equation}\label{eq_a1}
	X_{j,k}(z) =  \widehat{e_0}(z;k) \overline{\widehat{f_0}(z;j)}, 
	\quad\text{for }|j|,|k| \leq \left\lfloor\frac{C_1}{h}\right\rfloor.
  \end{equation}
\end{defn}
\begin{prop}
Let $z\in\Omega \Subset\Sigma$. Let $N$ be as in Definition \ref{def:X} 
and let $B(0,R)\subset\mathds{C}^{N}$ be the ball of radius $R:=C/h$, 
$C>0$ large, centered at $0$. Let $P_{h}^{\delta}$ be as in \eqref{eqn:DefPertP}, 
\eqref{eqn:defnModelOperator}. Let $R_-,R_+$ be as in Proposition 
\ref{prop:GrushUnpertOp}. Then
  \begin{equation*}
   \mathcal{P}_{\delta}(z):=\begin{pmatrix}
                    P_h^{\delta}-z & R_- \\ R_+ & 0 \\
                   \end{pmatrix}
   :~ H^1(S^1)\times \mathds{C}\longrightarrow L^2(S^1)\times \mathds{C}
  \end{equation*}
is bijective with the bounded inverse 
  \begin{equation*}
   \mathcal{E}_{\delta}(z) = \begin{pmatrix}
                     E^{\delta}(z) & E^{\delta}_+(z) \\ 
                     E^{\delta}_-(z) & E^{\delta}_{-+}(z) \\
                    \end{pmatrix}
  \end{equation*}
where 
  \begin{align*}
   &E^{\delta}(z)= E(z) + \mathcal{O}\!\left(\delta h^{-2}\right)
	  =\mathcal{O}(h^{-1/2})  
	  \notag \\
   &E^{\delta}_-(z)=   E_-(z) + \mathcal{O}\!\left(\delta h^{-3/2}\right)
	  =\mathcal{O}(1)  
	  \notag \\
   &E^{\delta}_+(z)= E_+(z) + \mathcal{O}\!\left(\delta h^{-3/2}\right)
          =\mathcal{O}(1)  
  \end{align*}
and 
\begin{equation}
  \label{eqn:PertE-+}
   E_{-+}^{\delta}(z) = 
   E_{-+}(z) 
   - \delta X(z)\cdot\alpha 
   + T(z;\alpha),
  \end{equation}
with $X(z)\cdot\alpha = E_-Q_{\omega}E_+$, $\alpha\in B(0,R)$, and
 \begin{equation}
  \label{eqn:PowerSeriesT_b}
   T(z,\alpha) : = \sum_{n=1}^{\infty}(-\delta)^{n+1} E_-Q_{\omega}(EQ_{\omega})^nE_+ 
    =  \mathcal{O}(\delta^2 h^{-5/2}).
  \end{equation}
Here, the dot-product $X(z)\cdot\alpha$ is the natural bilinear one.   
\end{prop}
\begin{rem}
\label{rem:ET_estim}
 The effective Hamiltonian $E_{-+}^{\delta}(z)$ depends smoothly on 
 $z\in\Omega$ and holomorphically on $\alpha \in B(0,R)\subset\mathds{C}^{N}$. 
 As in \cite[(8.6) and Proposition 4.6]{Vo14} we have the following 
 estimates: for all $z\in\Omega$, all $\alpha\in B(0,R)$ and all 
 $\beta=(\beta_1,\beta_2)\in \mathds{N}^2$ 
 \begin{align*}
  &\partial_{z}^{\beta_1}\partial_{\overline{z}}^{\beta_2} E_{-+}(z) = 
  \mO\left(
  h^{-|\beta|+1/2} \e^{-\frac{S}{h}}
  \right), ~\text{and}
  \notag \\
  &\partial_{z}^{\beta_1}\partial_{\overline{z}}^{\beta_2} T(z,\alpha) 
  = \mO\left( \delta^2
   h^{-(|\beta|+\frac{5}{2})} 
  \right)  
 \end{align*}
 where $S$ is as in \eqref{def_act}. 
 \par
 Moreover, as remarked in \cite{SjAX1002} the effective Hamiltonian 
 $E_{-+}^{\delta}(z)$ satisfies a $\overline{\partial}$-equa\-tion, 
 i.e. there exists a smooth function 
 $f^{\delta}:\Omega \rightarrow \C$ such that 
 \begin{equation*}
  \partial_{\overline{z}}E_{-+}^{\delta}(z)
  + f^{\delta}(z)E_{-+}^{\delta}(z) = 0. 
 \end{equation*}
 This implies that the zeros of $E_{-+}^{\delta}(z)$ 
 are isolated and countable and we may use the same notion of multiplicity 
 as for holomorphic functions. 
\end{rem}
\subsection{Counting zeros}
\label{sec:CountZero2}
By the above well-posed Grushin Problem for the perturbed operator 
$P_h^{\delta}$ we have that $\sigma(P_h^{\delta}) = (E_{-+}^{\delta})^{-1}(0)$. 
Hence, to study the the two-point intensity measure $\nu$ defined in 
(\ref{eqn:NU}), we investigate the integral
\begin{align*}
 \pi^{-N} \int_{B(0,R)} \Bigg(
     \sum_{\substack{z,w\in (E_{-+}^{\delta})^{-1}(0) \\
		     z\neq w}}
     \varphi(z,w) \Bigg)
     \e^{-\alpha^*\cdot\alpha} L(d\alpha)
     =
     \int_{\C^2}\varphi(z_1,z_2) d\nu(z_1,z_2) 
\end{align*}
with $\varphi\in\mathcal{C}_0(\Omega\times\Omega)$. Using 
Remark \ref{rem:ET_estim}, we see that the integral is finite since 
the number of pairs of zeros of $E_{-+}^{\delta}(\cdot,\alpha)$ in 
$\supp \varphi$ is uniformly bounded for $\alpha\in B(0,R)$. 
\par
Recall the definition of the point process $\Xi$ given in \eqref{defn:PP}. 
Using Lemma 7.1 in \cite{Vo14}, we get the following regularization of the 
$2$-fold counting measure $\Xi\otimes\Xi$ 
  \begin{align*}
  \langle \varphi, 
    \Xi\otimes \Xi \rangle
    = 
    \lim\limits_{\varepsilon \rightarrow 0^+}
    \iint \varphi(z_1,z_2)\prod_{j=1}^2
    \varepsilon^{-2}\chi\left(\frac{E_{-+}^{\delta}(z_l)}{\varepsilon}\right)
    |\partial_{z_l}E_{-+}^{\delta}(z_l)|^2 L(dz_1)L(dz_2)
    ,
  \end{align*}
where $\chi\in\mathcal{C}_0^{\infty}(\C)$ such that $\int \chi(w) L(dw)=1$. 
Assuming that $\varphi\in\mathcal{C}_0(\Omega\times\Omega)$ is such that 
$\{(z,z); ~z\in\Omega\}\cap \supp\varphi = \emptyset$, we see 
by the Lebesgue dominated convergence theorem that the 
two-point intensity measure of the point process $\Xi$ is given by 
 \begin{align}
 \label{eqn:kthMoment}
   \int_{\C^2}\varphi(z_1,z_2) d\nu(z_1,z_2)    
     = 
   \lim\limits_{\varepsilon\rightarrow 0^+}
    \iint \varphi(z_1,z_2)
     K^{\delta}_{\varepsilon}(z_1,z_2;h)
    L(dz_1)L(dz_2)
  \end{align}
with 
\begin{equation*}
 K^{\delta}_{\varepsilon}(z_1,z_2;h)
 :=\int_{B(0,R)} \left[\prod_{l=1}^2\varepsilon^{-2}\chi
    \left(\frac{E_{-+}^{\delta}(z_l)}{\varepsilon}\right)
    |\partial_{z_l}E_{-+}^{\delta}(z_l)|^2\right]
     \e^{-\alpha^*\alpha} L(d\alpha).
\end{equation*}
Using \eqref{eqn:PertE-+}, we see that the main object of interest, 
encoding all the information needed for \eqref{eqn:kthMoment},  
is the random vector 
\begin{align}
 \label{eqn:F}
 F^{\delta}(z,w,\alpha;h) &= 
    \begin{pmatrix}
        E_{-+}^{\delta}(z) \\ 
        E_{-+}^{\delta}(w) \\
        (\partial_z E_{-+}^{\delta})(z) \\
         (\partial_z E_{-+}^{\delta})(w) \\
    \end{pmatrix}
    \\
    &=
    \begin{pmatrix}
            E_{-+}(z) \\ 
            E_{-+}(w) \\ 
            (\partial_z E_{-+})(z) \\
            (\partial_z E_{-+})(w) \\
   \end{pmatrix}
   -
   \delta
    \begin{pmatrix}
        X(z)\cdot\alpha \\ 
        X(w)\cdot\alpha  \\ 
        (\partial_z X)(z)\cdot\alpha  \\
        (\partial_z X)(w)\cdot\alpha  \\
   \end{pmatrix}
   +
        \begin{pmatrix}
            T(z,\alpha) \\ 
            T(w,\alpha) \\ 
            (\partial_z T)(z,\alpha) \\ 
            (\partial_z T)(w,\alpha) \\ 
     \end{pmatrix},
     \notag 
  \end{align}
where $X(z)$, $X(w)$ are given in Definition \ref{def:X}. 
It will be very useful in the sequel to define the following $G$.
\begin{align}
  \label{eqn:G_44}
   G : = \begin{pmatrix}
          A & 
          B \\ 
          B^* & 
          C \\ 
          \end{pmatrix}\in \mathds{C}^{4 \times 4},
  \end{align}
  with 
\begin{align}
  \label{eqn:ABC}
   &A := \begin{pmatrix}
          (X(z)|X(z)) & 
          (X(z)|X(w)) \\ 
          (X(w)|X(z)) & 
          (X(w)|X(w)) \\ 
          \end{pmatrix}, 
           \notag \\
  &B := \begin{pmatrix}
          (X(z)|\partial_z X(z)) & 
          (X(z)|\partial_w X(w)) \\ 
          (X(w)|\partial_z X(z)) & 
          (X(w)|\partial_w X(w)) \\ 
          \end{pmatrix},     
          \notag \\
  &C := \begin{pmatrix}
          (\partial_zX(z)|\partial_z X(z)) & 
          (\partial_zX(z)|\partial_w X(w)) \\ 
          (\partial_wX(w)|\partial_z X(z)) & 
          (\partial_wX(w)|\partial_w X(w)) \\ 
          \end{pmatrix}.
\end{align}
Notice that the matrices $A,B,C$ depend on $h$; see Definition 
\ref{def:X}. Next, we will state a formula for the Lebesgue 
density of the two-point intensity measure $\nu$ in terms of 
the permanent of the Shur complement of $G$, i.e. 
\begin{equation}\label{eq_a2}
\Gamma:=C - B^*A^{-1}B.
\end{equation}
The permanent of a matrix is defined 
as follows (cf. \cite{Matha89}):
\begin{defn}
\label{defn:Permanent}
 Let $(M_{ij})_{ij}=M\in\C^{n\times n}$ be a square matrix 
 and let $S_n$ denote the symmetric group of order $n$. The 
 permanent of $M$ is defined by 
 \begin{equation}\label{def:Per}
  \perm M := \sum_{\sigma\in S_n} \prod_{i=1}^n M_{i \sigma(i)}.
 \end{equation}
\end{defn}
\begin{rem}
 Although the definition of the permanent resembles closely to that 
 of the determinant, the two object are quite different. Many properties 
 known to hold true for determinants, fail to be true for permanents. 
 For our purposes it is enough to note that it is multi-linear and 
 symmetric. For more details concerning permanents and their properties we 
 refer the reader to \cite{Matha89}.
\end{rem}
We will prove the following result:
\begin{prop}
\label{prop:2ptCorrelation}
 Let $\Omega\Subset\Sigma$ be as in Hypothesis \ref{Hyp:H6}. Let $\delta>0$ be 
 as in Hypothesis \ref{Hyp:H7} and let $\Gamma$ be as in \eqref{eq_a2}. 
 Moreover, let $D(\Omega,C_2)$ be as in \eqref{eq_i27}. Then, 
 there exists a smooth function
  \begin{equation*}
   D^{\delta}(z,w;h) 
   = \frac{\perm\Gamma(z,w;h)\
  + \mO\!\left(\e^{-\frac{1}{Ch}}+ 
    \delta h^{-\frac{52}{10}}\right)}{\pi^{2}\left(\sqrt{\det A(z,w;h)} 
    + \mO\left(\delta h^{-\frac{3}{2}}\right)\right)^2}
    + \mO\!\left(\e^{-\frac{D}{h^2}}\right).
\end{equation*}
 and there exists a constant $C_2>0$ such that for all 
 $\varphi\in\mathcal{C}_0(\Omega^2\backslash D_h(\Omega,C_2))$ 
   \begin{equation*}
   \int_{\C^2}\varphi(z,w) d\nu(z,w)
   = 
    \int_{\C^2} \varphi(z,w) D(z,w,h,\delta) L(d(z,w)).
  \end{equation*}
\end{prop}
\begin{rem}
 The proof of Proposition \ref{prop:2ptCorrelation} will take up most 
 of the rest of this paper. Therefore we give a short overview on how we will 
 proceed:
 \par
 In Section \ref{sec:StPh}, we give a formula for the scalar 
 product $(X(z)|X(w))$ by constructing holomorphic quasimodes for the operators 
 $(P_h-z)$ and $(P_h-z)^*$ to approximate the eigenfunction $e_0$ and 
 $f_0$, and by using the method of stationary phase.
 \par
 In Section \ref{susec:Gramian}, we will use this formula to 
 study the invertibility of the matrices $G,A$ and $\Gamma$. 
 Furthermore, we will study the permanent of $\Gamma$. 
 \par
 In Section \ref{sec:CorrelationFormula}, we give a proof of 
 Proposition \ref{prop:2ptCorrelation}.
\end{rem}
\section{Stationary Phase}
\label{sec:StPh}
In this section we are interested in the scalar product 
$(X(z)|X(w))$. Recall from Definition \ref{def:X} that the vector 
$X(z)$, $z\in\Omega$, is given by 
$X_{j,k} = \widehat{e_0}(z;k) \overline{\widehat{f_0}(z;j)}$, 
where $e_0$ and $f_0$ are the eigenfunctions of the operators 
$Q(z)$ and $\widetilde{Q}(z)$, respectively, associated to 
their first eigenvalue $t_0^2$. 
\par
The Fourier coefficients 
$\widehat{e_0}(z;k), \widehat{f_0}(z;j)$ 
and their $z$- and $\overline{z}$-derivatives 
are of order $\mO(|k|^{-\infty})$,  $\mO(|j|^{-\infty})$, 
for $|j|,|k|\geq C/h$ with $C>0$ large enough 
(cf \cite[Propositions 5.3 and 5.4]{Vo14}). The Parseval identity 
implies that for $z,w\in\Omega$
  \begin{equation}
  \label{eqn:SP_XX}
   (X(z)|X(w)) = (e_0(z)|e_0(w))(f_0(w)|f_0(z)) 
		 + \mO_{\mathcal{C}^{\infty}}(h^{\infty}).
  \end{equation}
The aim of this section is to prove the following result: 
\begin{prop}
 \label{prop:SP_XX}
 Let $\Omega\Subset\Sigma$ be as in Hypothesis \ref{Hyp:H6} 
 and let $x_{\pm}(z)$ be as in \eqref{eqn:xpm}. Furthermore, for 
 $z\in\Omega$ let $\sigma(z)$ denote the Lebesgue density of the 
 direct image of the symplectic volume form on $T^*S^1$ under the 
 principal symbol $p$, i.e. $\sigma(z)L(dz) = p_*(d\xi \wedge dx)$.
 \par
 Then, there exists a constant $C>0$ such that for all $(z,w)\in 
 \Delta_{\Omega}(C):=\{(z,w)\in\Omega^2;~ |z-w|<1/C\}$
 \begin{align*}
  (X(z)|X(w)) = \e^{-\frac{1}{h}\Phi(z;h)-\frac{1}{h}\Phi(w;h)}
	        \e^{\frac{2}{h}\Psi(z,w;h)}
      +\mO_{\mathcal{C}^{\infty}}\!\left(h^{\infty}\right)
 \end{align*}
 where: 
 \begin{itemize}
  \item $\Phi(\cdot;h):\Omega \rightarrow \R$ is a family 
        of smooth functions depending only on $i\Ima z$,  
        which satisfy
	\begin{align*}
	  \Phi(z;h) = &
	  \Ima \int_{x_+(z)}^{x_0} (z- g(y))dy 
	  -\Ima \int_{x_-( z)}^{y_0} (z- g(y))dy 
	  \notag \\
	  &+\frac{h}{4}\left[\ln \left(\frac{\pi h}{-\Ima g'(x_+(z))}\right)
	  +\ln \left(\frac{\pi h}{\Ima g'(x_-(z))}\right)\right]
	  +\mO(h^2).
        \end{align*}
        and 
        \begin{equation*}
         \partial_{z\overline{z}}^2
	      \Phi\left(z;h\right) 
	  =\frac{1}{4}\sigma\left(z\right)  +\mO(h).
        \end{equation*}
  \item  $\Psi(\cdot,\cdot;h):\Delta_{\Omega}(C) \rightarrow \C$ is 
   a family of smooth functions which are almost $z$-holo\-morphic 
   and almost $w$-anti-holomorphic extensions from the diagonal 
   $\Delta:=\{(z,z); z\in\Omega\}\subset \Delta_{\Omega}(C)$ 
   of $\Phi(z;h)$, i.e. 
    \begin{equation*}
    \Psi(z,z;h) = \Phi\left(\frac{1}{2}(z-\overline{z});h\right),
    ~
     \partial_{\overline{z}}\Psi,\partial_{w}\Psi
      = \mO(|z-w|^{\infty}). 
    \end{equation*}
    Moreover, we have that $\Psi(z,z)=\Phi(z)$ and 
    for $z,w\in \Delta_{\Omega}(C)$ with $|z-w| \ll 1$, 
	\begin{align*}
	  \Psi(z,w;h) 
	  = \sum_{|\alpha+\beta|\leq 2}&
	      \frac{1}{2^{|\alpha+\beta|}\alpha!\beta!}
	    \partial_z^{\alpha}\partial_{\overline{z}}^{\beta}
	      \Phi\left(\frac{z+w}{2};h\right) 
	    (z-w)^{\alpha}(\overline{w-z})^{\beta}
	    \notag \\
	    &+ 
	    \mO(|z-w|^3+h^{\infty}),
       \end{align*}
	and 
	\begin{align*}
	 2\Rea &\Psi(z,w;h)-\Phi(z;h)- \Phi(w;h) 
	 \notag \\
	 &= -\partial_{z\overline{z}}^2
	      \Phi\left(\frac{z+w}{2};h\right) 
         |z-w|^2(1 + \mO(|z-w|+h^{\infty}));
	\end{align*}
   \item the function $\Psi(z,w;h)$ has the following symmetries: 
      \begin{equation*}
       \Psi(z,w;h) = \overline{\Psi(w,z;h)} 
       \quad
       \text{and}
       \quad
       (\partial_z\Psi)(z,w;h) = \overline{(\partial_{\overline{w}}\Psi)(w,z;h)} .
      \end{equation*}
 \end{itemize}
  \end{prop}
 Let us give some remarks on the above results: Note that the formula 
 for $\Psi$ stated above is simply a special case of the more general 
 Taylor expansion 
  \begin{equation*}
	  \Psi(z_0+\zeta,z_0+\omega;h) 
	  = \sum_{|\alpha+\beta|\leq 2}
	      \frac{1}{\alpha!\beta!}
	    \partial_z^{\alpha}\partial_{\overline{z}}^{\beta}
	      \Phi\left(z_0;h\right) 
	    \zeta^{\alpha}\overline{\omega}^{\beta}
	    + 
	    \mO((\zeta,\omega)^3+h^{\infty}),
 \end{equation*}
 with $z_0\in\Omega$ and $|\zeta|,|\omega|\ll 1$. 
 \\
 \par
To prove Proposition \ref{prop:SP_XX}, we will use \eqref{eqn:SP_XX} and 
study the scalar products $(e_0(z)|e_0(w))$ and $(f_0(w)|f_0(z))$ separately,  
see Sections \ref{susec:SP_ee} and \ref{susec:SP_ff} below. The proof of 
Proposition \ref{prop:SP_XX} will then be stated at the end of this section. 
 \begin{rem}
  Note that the behaviour of $(X(z)|X(w))$ is close to the behaviour of Bergman 
  kernels (see for example \cite[Sec. 13.3]{SemCla}). However, we will not use this 
  notion in the sequel. 
 \end{rem}
 Next, we define for 
 $(z,w)\in \Delta_{\Omega}(C)$, as in Proposition \ref{prop:SP_XX}, 
  \begin{align}
   \label{eqn:K}
   -K(z,w) :&= 2\Rea \Psi(z,w;h)-\Phi(z;h)- \Phi(w;h) 
	  \\
	 &= -\left(\sigma\left(\frac{z+w}{2}\right)  +\mO(h)\right)
         \frac{|z-w|^2}{4}(1 + \mO(|z-w|+h^{\infty})). \notag
  \end{align}
  From the above Proposition we can immediately deduce 
  some growth properties of certain quantities that 
  will be become important in the sequel.
  \begin{cor}
  \label{cor:DetPer_A}
   Under the assumptions of Proposition \ref{prop:SP_XX}, we have that 
   \begin{equation}
   	\begin{split}
	& \bullet \quad |(X(z)|X(w))|=  \e^{-\frac{K(z,w)}{h}}
   			 +\mO_{\mathcal{C}^{\infty}}\!\left(h^{\infty}\right); \\
	& \bullet  \quad \lVert X(z) \rVert^2 \lVert X(w) \rVert^2 \pm
          		|(X(z)|X(w))|^2	= \left( 1\pm 
             			  \e^{-\frac{2K(z,w)}{h}}\right)
             		+\mO_{\mathcal{C}^{\infty}}\!\left(h^{\infty}\right); \\
	& \bullet  \quad \lVert X(z) \rVert^2 \lVert X(w) \rVert^2 |(X(z)|X(w))|^2
			= \e^{-\frac{2K(z,w)}{h}}
             +\mO_{\mathcal{C}^{\infty}}\!\left(h^{\infty}\right).
	\end{split}
   \end{equation}
  \end{cor}
\subsection{The Scalar Product $(e_0(z)|e_0(w))$}
\label{susec:SP_ee}
 We will prove 
 \begin{prop}
 \label{prop:Sp_ee}
 Let $\Omega\Subset\Sigma$ be as in Hypothesis \ref{Hyp:H6} 
 and let $x_+(z)$ be as in \eqref{eqn:xpm}. 
 Then, there exists a constant $C>0$ such that for all $(z,w)\in 
 \Delta_{\Omega}(C):=\{(z,w)\in\Omega^2;~ |z-w|<1/C\}$ 
  \begin{equation}
  \label{eqn:Sp_ee_1}
   (e_0(z)|e_0(w)) = 
      \e^{-\frac{1}{h}\Phi_1(z;h)}
      \e^{-\frac{1}{h}\Phi_1(w;h)}
      \e^{\frac{2}{h}\Psi_1(z,w;h)}
         + \mO\!\left(h^{\infty}\right),
  \end{equation}
 where:
 \begin{itemize}
  \item $\Phi_1(\cdot;h):\Omega \rightarrow \R$ is a family 
        of smooth functions depending only on $i\Ima z$, which satisfy
	\begin{equation*}
	  \Phi_1(z;h) = 
	  \Ima \int_{x_+(\Ima z)}^{x_0} (z- g(y))dy 
	  +\frac{h}{4}\ln \left(\frac{\pi h}{-\Ima g'(x_+)}\right)
	  +\mO(h^2).
        \end{equation*}
  \item   $\Psi_1(\cdot,\cdot;h):\Delta_{\Omega}(C) \rightarrow \C$ is 
   a family of smooth functions which are almost $z$-holo\-morphic 
   and almost $w$-anti-holomorphic extensions from the diagonal 
   $\Delta:=\{(z,z); z\in\Omega\}\subset \Delta_{\Omega}(C)$ 
   of $\Phi_1(z;h)$, i.e. 
    \begin{equation*}
    \Psi_1(z,z;h) = \Phi_1\left(\frac{1}{2}(z-\overline{z});h\right),
    ~
     \partial_{\overline{z}}\Psi_1,\partial_{w}\Psi_1
      = \mO(|z-w|^{\infty}).
    \end{equation*}
    Moreover, for $z,w\in \Delta_{\Omega}(C)$ with $|z-w|\ll 1$, 
    one has that 
	\begin{align*}
	  \Psi_1(z,w;h) 
	  = \sum_{|\alpha+\beta|\leq 2}&
	      \frac{1}{2^{|\alpha+\beta|}\alpha!\beta!}
	    \partial_z^{\alpha}\partial_{\overline{z}}^{\beta}
	      \Phi_1\left(\frac{z+w}{2};h\right) 
	    (z-w)^{\alpha}(\overline{w-z})^{\beta}
	    \notag \\
	    &+ 
	    \mO(|z-w|^3+h^{\infty}),
       \end{align*}
	and that 
	\begin{align*}
	 2\Rea &\Psi_1(z,w;h)-\Phi_1(z;h)- \Phi_1(w;h) 
	 \notag \\
	 &= -
        \partial_z\partial_{\overline{z}}
        \Phi_1\left(\frac{z+w}{2};h\right)  
         |z-w|^2(1 + \mO(|z-w|+h^{\infty}));
	\end{align*}
 \item the function $\Psi_1(z,w;h)$ has the following symmetries: 
      \begin{equation*}
       \Psi_1(z,w;h) = \overline{\Psi_1(w,z;h)} 
       \quad
       \text{and}
       \quad
       (\partial_z\Psi_1)(z,w;h) = 
       \overline{(\partial_{\overline{w}}\Psi_1)(w,z;h)} .
      \end{equation*}
 \end{itemize}
 \end{prop}
 To prove Proposition \ref{prop:Sp_ee}, we begin by constructing an 
 oscillating function to approximate $e_0(z)$. Let us recall from Section 
 \ref{sec:introduction} that the 
 points $a,b\in S^1$ denote the minimum and the maximum of 
 $\Ima g(x)$ and that for $z\in\Omega$ the points 
 $x_{\pm}(z)\in S^1$ are the unique solutions to the 
 equation $\Ima g(x) = \Ima z$. Furthermore, we will identify frequently 
 $S^1$ with the interval $[b-2\pi,b[$. Moreover, let us recall that 
 by the natural projection 
 $\Pi:\mathds{R}\rightarrow S^1=\mathds{R}/2\pi\mathds{Z}$ 
 we identify the points $x_{\pm},a,b\in S^1$ with points 
 $x_{\pm},a,b\in\mathds{R}$ such that 
 $ b-2\pi <  x_+  < a < x_- < b$. 
 \\
 \par
 Let $K_+\subset ]b-2\pi,a[$ be an open interval such that 
 $x_+(z)\in K_+$ for all $z\in\Omega$.
  Let $\chi\in\mathcal{C}^{\infty}_0(]b-2\pi,a[)$ and define
  for $x\in\R$
  \begin{equation}
  \label{eqn:HoloQuasMod}
  \widetilde{e}_0(x,z) : = 
      \chi(x)\exp\left({\frac{i}{h}\psi_+(x,z)}\right).
 \end{equation}
  where, for a fixed $x_0\in K_+$,
  \begin{equation}
  \label{eqn:defn_psi_+}
  \psi_+(x,z):= \int_{x_0}^x \left(z - g(y)\right) dy.
  \end{equation}
  \begin{rem}
   \label{rem:holo}
   Note that the function $u=\exp(i \psi_+(x,z)/h)$ is 
   solution to $(P_h -z) u =0$ on $\supp \chi$, since the phase function 
   $\psi_+$ satisfies the eikonal equation 
  \begin{equation*}
    p(x,\partial_x\psi_+) = z.
  \end{equation*}
  Furthermore, let us remark that $\widetilde{e}_0(x,z)$ depends 
  holomorphically on $z$.
  \end{rem}
  Next, we are interested in the $L^2$-norm 
  of $\widetilde{e}_0$.
 \begin{lem}
  \label{lem:norm_qm_e}
  Let $\Omega\Subset\Sigma$ be as in Hypothesis \ref{Hyp:H6}. Then, there 
  exists a family of smooth functions 
  $\Phi_1(\cdot;h):\Omega \rightarrow \R$, such that 
  \begin{equation*}
   \Phi_1(z;h) = \Phi_1(i\Ima z;h) = 
   \Ima \int_{x_+(\Ima z)}^{x_0} (z- g(y))dy 
   +\frac{h}{4}\ln \left(\frac{\pi h}{-\Ima g'(x_+)}\right)
   +\mO(h^2)
  \end{equation*}
  and  
  \begin{align*}
  \lVert \widetilde{e}_0(z)\rVert^2  
   = \exp\left\{\frac{2}{h}\Phi_1(z;h)\right\}.
 \end{align*}
 \end{lem}
 \begin{proof}
  In view of the definition of $\widetilde{e}_0(z)$, see 
  \eqref{eqn:HoloQuasMod} and \eqref{eqn:defn_psi_+}, one gets that 
  \begin{equation*}
  \lVert \widetilde{e}_0(z)\rVert^2 = 
  \int \chi(x)\e^{\frac{i}{h}(\psi_+(x,z) - \overline{\psi}_+(x,z))}
  dx
  = 
   \int \chi(x)\e^{-\frac{2}{h}\Ima \psi_+(x,z)}
    dx.
 \end{equation*}
 The critical point for $\Ima \psi_+(x,z)$ is given by 
 the equation 
 \begin{equation*}
   \Ima \partial_x \psi_+(x,z) = \Ima z - \Ima g(x) = 0, \quad 
   x\in\supp\chi.
 \end{equation*}
 The critical point, given by $x_+(\Ima z)$, is unique and it 
 satisfies $\Ima g'(x_+(\Ima z)) < 0$, see 
 \eqref{eqn:xpm}. This implies in particular 
 that the critical point is non-degenerate. More precisely,
 \begin{equation}
  \label{eqn:psi+Hess}
    \Ima (\partial_{xx}^2 \psi_+)(x_+,z) = -\Ima g'(x_+) > 0.
 \end{equation}
 The critical value of $\Ima \psi_+$ is given by 
 \begin{equation*}
  \Ima \psi_+(x_+(\Ima z),z) = 
  \Ima \int_{x_0}^{x_+(\Ima z)} (z- g(y))dy \leq 0.
 \end{equation*}
 Using the method of stationary phase, one gets 
 \begin{align*}
  \lVert \widetilde{e}_0(z)\rVert^2 &= 
  \sqrt{\frac{\pi h}{\Ima (\partial_{xx}^2
  \psi_+)(x_+,z)}}(1+ \mO(h))
  \exp\left\{-\frac{2\Ima \psi_+(x_+,z)}{h}\right\}
  \notag\\
  & =: \exp\left\{\frac{2}{h}\Phi_1(z;h)\right\},
 \end{align*}
 where $\Phi_1$ is smooth in $z$. Using \eqref{eqn:psi+Hess}, 
 one gets that 
 \begin{equation*}
  \Phi_1(z;h) = \Ima \int_{x_+(\Ima z)}^{x_0} (z- g(y))dy 
  +\frac{h}{4}\ln \left(\frac{\pi h}{-\Ima g'(x_+)}\right)
  +\mO(h^2).  \qedhere
 \end{equation*}
 \end{proof}
 Recall from \eqref{def:e0,e1,...} that the function $e_0$ is an 
 eigenfunction of the operator $Q(z)$ (cf Section \ref{sec:AuxOpe}) 
 corresponding to its first eigenvalue $t_0^2$. We set 
 \begin{equation*}
  e_0(z) = 
  \frac{\Pi_{t_0^2} \left(\e^{-\frac{1}{h}\Phi_1(z;h)}
	\widetilde{e}_0(z)\right)}
  {\left\lVert \Pi_{t_0^2} 
    \left(\e^{-\frac{1}{h}\Phi_1(z;h)}\widetilde{e}_0(z) \right)\right\rVert},
 \end{equation*}
 where $\Pi_{t_0^2}:L^2(S^1) \rightarrow \C e_0$ denotes the 
 spectral projection for $Q(z)$ onto the eigenspace associated 
 with $t_0^2$. 
 \par
 Next, we prove that up to an exponentially small error in
 $1/h$, $e_0$ is given by the normalization of $\widetilde{e}_0$. 
 \begin{lem} 
  Let $\Omega\Subset\Sigma$ be as in Hypothesis \ref{Hyp:H6}. Then, there exists 
  a constant $C>0$ such that for all $z\in\Omega$ and all 
  $\alpha=(\alpha_1,\alpha_2)\in\mathds{N}^2$ 
  \begin{equation*}
  \left\lVert 
  \partial_{z}^{\alpha_1} \partial_{\overline{z}}^{\alpha_2} 
  \left(
  e_0 (z) - \e^{-\frac{1}{h}\Phi_1(z;h)}\widetilde{e}_0(z)
  \right)
  \right\rVert 
   = \mO\!\left(h^{-|\alpha|}\e^{-\frac{1}{Ch}}\right).
  \end{equation*}

 \end{lem}
 \begin{proof}
 The proof of the lemma is similar to the proof of 
 \cite[Proposition 3.11]{Vo14}. 
 \end{proof}
 This result implies that 
 \begin{equation}
  \label{eqn:sp_ee_dec}
  (e_0(z)|e_0(w)) = 
    \e^{-\frac{1}{h}\Phi_1(z;h)-\frac{1}{h}\Phi_1(w;h)}
		    (\widetilde{e}_0(z)|\widetilde{e}_0(w))
     +\mO_{\mathcal{C}^{\infty}}\!\left(\e^{-\frac{1}{Ch}}\right).
 \end{equation}
 By Remark \ref{rem:holo}, 
 $(\widetilde{e}_0(z)|\widetilde{e}_0(w))$ is holomorphic 
 in $z$ and anti-holomorphic in $w$. We can study this scalar 
 product by the method of stationary phase:
 \begin{proof}[of Proposition \ref{prop:Sp_ee}]
  In view of \eqref{eqn:sp_ee_dec}, it remains to 
  study the oscillatory integral 
  \begin{equation}
   \label{eqn:Int2Treat}
   I(z,w) := (\widetilde{e}_0(z)|\widetilde{e}_0(w)) 
    = \int \chi(x)\exp\left({\frac{i}{h}\Psi_+(x,z,w)}\right) dx,
  \end{equation}
  where $\widetilde{e}_0(x,z)$ is given in \eqref{eqn:HoloQuasMod} 
  and $\Psi_+$ is defined by 
 \begin{equation}
 \label{eqn:defn_Psi_+}
  \Psi_+(x,z,w) : = \psi_+(x,z)- \overline{\psi_+(x,w)}, ~ z,w\in\Omega.
 \end{equation}
 Using \eqref{eqn:defn_psi_+}, 
 \begin{equation}
  \label{eqn:Ps_+_expl}
   \Psi_+(x,z,w) = \int_{x_0}^{x} \Rea(z-w) dy +
    2i\int_{x_0}^{x} 
    \left[\Ima \left(\frac{z+w}{2} \right)- 
	  \Ima  g(y)\right] dy.
  \end{equation}
 Since the imaginary part of $\Psi_+$ can be negative, we 
 shift the phase function by the minimum of $\Ima \Psi_+$.
 \\
 \par
 \textbf{Minimum of $\Ima \Psi_+$.} The critical points of 
 the function $x\mapsto \Ima \Psi(x,z,w)$ are given by 
 the equation $\Ima (\frac{z+w}{2}) = \Ima g(x)$. Since 
 $\Omega$ is convex, this equation has, for $|z-w|$ small 
 enough, on the support of $\chi$ the unique solution 
 $x_+(\frac{z+w}{2})\in\R$ and it satisfies 
 $\Ima g'(x_+(\frac{z+w}{2})) < 0$ (cf. \eqref{eqn:xpm}). 
 Moreover, it depends smoothly on $z$ and $w$ since $g$ is smooth. 
 Therefore,
  \begin{equation*}
   (\partial_{xx}^2\Ima\Psi_+)
   \left(x_+\left(\frac{z+w}{2}\right),z,w\right) 
   = - 2\Ima g_x'\left(x_+\left(\frac{z+w}{2}\right) \right)  >0,
  \end{equation*}
  which implies that $x_+(\frac{z+w}{2})$ is a minimum point, and 
  that 
  \begin{align}
   \label{eqn:Inf_Psi_+}
   2\lambda:= 2\lambda(z,w):&=
   \Ima \Psi_+\left(x_+\left(\frac{z+w}{2}\right),z,w\right)
   \notag \\
   & = 
   2\int_{x_0}^{x_+(\frac{z+w}{2})} 
    \left[\Ima \left(\frac{z+w}{2} \right)- 
	  \Ima  g(y)\right] dy \leq 0.
  \end{align}
 We define $\Theta_+(x,z,w) := \Psi_+(x,z,w)-i\lambda$, 
 and notice that $\Ima \Theta_+(x,z,w) \geq 0$. Hence, 
 we can write \eqref{eqn:Int2Treat} as follows: 
 \begin{equation}
  \label{eqn:I_2}
  I(z,w) = \e^{-\frac{2\lambda}{h}}
  \int \chi(x)\exp\left({\frac{i}{h}\Theta_+(x,z,w)}\right) dx.
 \end{equation}
 To study $I(z,w)$ by the method of stationary 
 phase, we are interested in the critical points 
 of $\Theta_+$.
 \\
 \par
 \textbf{Critical points of $\Theta_+$.} Clearly they 
 are the same as for  $\Psi_+(x,z,w)$. Note that for $z=w$ one has that 
  \begin{equation*}
    \Psi_+(x,z,z) = 2i\Ima \int_{x_0}^{x}(z-g(y))dy
  \end{equation*}
  which has, on the support of $\chi$, the unique 
  critical point $x_+$ and it satisfies 
  $\Ima g'(x_+) < 0$ (cf. \eqref{eqn:xpm}). Therefore,
  \begin{equation*}
   \Ima (\partial_{xx}^2\Psi_+)(x_+(z),z,z) = - 2\Ima g_x'(x_+(z))
   >0
  \end{equation*}
  which implies that $x_+$ is a non-degenerate critical point. 
  \par 
  In the case where $z\neq w$ the situation is more complicated. 
  By  \eqref{eqn:Ps_+_expl} we see that if $\Rea (z-w) =0$, for $|z-w|$ 
  small enough, the critical point is real and given by 
  $x_+(\frac{z+w}{2})$, i.e. the minimum point of $\Ima \Psi_+$. 
  \par
  However, if $\Rea (z-w) \neq 0$, we need to consider an almost $x$-analytic 
  extension of $\Psi_+$, which we shall denote by $\widetilde{\Psi}_+$. 
  As described in \cite{MelSj74}, the ``critical point'' of 
  $\widetilde{\Psi}_+$ is then given by 
  \begin{equation*}
   \partial_x \widetilde{\Psi}_+(x,z,w) = 0,
  \end{equation*}
  and we will see, by the following result, that it ``moves'' to the complex 
  plane. 
 \begin{lem}
 \label{lem:CP_ee}
 Let $\Omega\Subset\Sigma$ be as in \eqref{eq_i35}. Let $\chi$ be as 
 in (\ref{eqn:HoloQuasMod}) and let $p$ be 
 the principal symbol of $P_h$ (cf \eqref{eq_I3}). Let 
 $x_{+}(z)$ be as in \eqref{eqn:xpm}. Furthermore, let $\widetilde{\psi}_+$ 
 denote an almost analytic extension of $\psi_+$ to a small complex 
 neighborhood of the support of $\chi$, and define 
 $\widetilde{\psi}_+^*(x):=\overline{\widetilde{\psi}_+(\overline{x})}$. 
 Then, the there exists a $C>0$ such that for $(z,w)\in \Delta_{\Omega}(C)$ 
 the function
 \begin{equation*}
   \partial_x \widetilde{\Psi}_+(x,z,w)=  \partial_x \widetilde{\psi}_+(x,z) 
	- ( \partial_x \widetilde{\psi}_+)^*(x,w)
 \end{equation*}
 has exactly one zero, $x_{+}^c(z,w)$, and:
 \begin{itemize}
  \item it depends almost holomorphically on $z$ 
	and almost anti-holomorphically $w$ at the diagonal 
        $\Delta$, i.e. 
  \begin{equation*}
    \partial_{w}x_+^c(z,w), \partial_{\overline{z}}x_+^c(z,w) 
     = 
     \mO(|z-w|^{\infty});
   \end{equation*}
 \item it is non-degenerate in the sense that 
    \begin{equation*}
    (\partial_{xx}^2 \widetilde{\Psi}_+)(x_{+}^c(z,w),z,w) \neq 0;
   \end{equation*}
 \item for $z,w\in\Omega$ with $|z-w| < 1/C$, $C>1$ large enough, 
 one has
 \begin{equation*}
   x_+^c(z,w) = x_{+}\left(\frac{z+w}{2}\right)
    - \frac{\Rea (z-w) }{\{p,\overline{p}\}(\rho_{+}\left(\frac{z+w}{2}\right))} 
    + \mO(|z-w|^2).
  \end{equation*}
 \end{itemize}
\end{lem}
 \begin{rem}
  The proof of Lemma \ref{lem:CP_ee} will be given 
  after the proof of Proposition \ref{prop:Sp_ee}. 
 \end{rem}
  Let $\widetilde{\Psi}_+$ denote an almost $x$-analytic 
  extension of $\Psi_+$. Using the method of stationary 
  phase for complex-valued phase functions (cf. Theorem 2.3 in 
  \cite[p.148]{MelSj74}) and Lemma \ref{lem:CP_ee}, one gets 
  that
  \begin{equation}
   \label{eqn:I_1}
   I(z,w)= 
         \exp\left\{\frac{2\Psi_1(z,w;h)}{h}\right\}
         + \mO\!\left(h^{\infty}\right)\e^{-\frac{2\lambda}{h}}.
  \end{equation}
  Using that Lemma \ref{lem:norm_qm_e} and \eqref{eqn:Inf_Psi_+} imply 
  $\lambda(z,w) + \Phi(z;h) + \Phi(w;h) \geq 0$, we obtain 
  \eqref{eqn:Sp_ee_1} from the above and \eqref{eqn:sp_ee_dec}.
  \\
  \par
  In \eqref{eqn:I_1}, $2\Psi_1(z,w)$ is given by the critical value of 
  $i \widetilde{\Psi}_+$ and by the logarithm of the amplitude 
  $c(z,w,h)$, given by the stationary phase method, i.e. 
  \begin{equation*}
   2\Psi_1(z,w;h) = i \widetilde{\Psi}_+(x_+^c(z,w),z,w) + h \ln c(z,w,h)
  \end{equation*}
  and $c(z,w,h) \sim c_0(z,w) + hc_1(z,w) + \dots$ which depends smoothly 
  on $z$ and $w$ in the sense that all $z$-,$\bar{z}$-,$w$- and 
  $\bar{w}$-derivatives remain bounded as $h\rightarrow 0$. 
  $\widetilde{\Psi}_+(x,z,w)$ is by definition $z$-holomorphic, $w$-anti-holomorphic 
  and smooth in $x$. By Lemma \ref{lem:CP_ee}, we know that 
  the critical point $x_+^c(z,w)$ is almost $z$-holomorphic and almost 
  $w$-anti-holomorphic in $\Delta_{\Omega}(C)$, a small neighborhood of the 
  diagonal $z=w$. Hence, $\Psi$ is almost $z$-holomorphic and almost 
  $w$-anti-holomorphic in $\Delta_{\Omega}(C)$. 
  \par
  Equivalently, $\Psi$ is an almost $z$-holomorphic and almost 
  $w$-anti-holomorphic extension from the diagonal of $\Psi_1(z,z;h)$.
  Since $\Psi_1(z,z;h) = \Phi_1(z;h)$, we obtain by Taylor expansion 
  up to order $2$ of $\Psi$ at $(\frac{z+w}{2},\frac{z+w}{2})$, that 
  \begin{align*}
   \Psi_1(z,w;h) 
    = \sum_{|\alpha+\beta|\leq 2}&
	\frac{1}{2^{|\alpha+\beta|}\alpha!\beta!}
        \partial_z^{\alpha}\partial_{\overline{z}}^{\beta}
        \Phi_1\left(\frac{z+w}{2};h\right) 
         (z-w)^{\alpha}(\overline{w-z})^{\beta}
         \notag \\
         &+ 
	\mO(|z-w|^3+h^{\infty}),
  \end{align*}
  for $|z-w|$ small enough. Similarly, 
  \begin{align*}
   \Phi_1(z;h) 
    = \sum_{|\alpha+\beta|\leq 2}&
	\frac{1}{2^{|\alpha+\beta|}\alpha!\beta!}
        \partial_z^{\alpha}\partial_{\overline{z}}^{\beta}
        \Phi_1\left(\frac{z+w}{2};h\right)  
         (z-w)^{\alpha}(\overline{z-w})^{\beta}
         \notag \\
         &+ 
	\mO(|z-w|^3+h^{\infty}), 
  \end{align*}
  which implies that  
  \begin{align*}
   2\Rea \Psi_1(z,w;h)&= 
   \Phi_1(z;h)  + \Phi_1(w;h) 
    -
        \partial_z^{\alpha}\partial_{\overline{z}}^{\beta}
        \Phi_1\left(\frac{z+w}{2};h\right)  
         |z-w|^2
         \notag \\
         &+ 
	\mO(|z-w|^3+h^{\infty}),
  \end{align*}
  concluding the proof of the second point of the proposition. 
  \\
  \par
  Finally, let us give a proof of the stated 
  symmetries. The fact that $\Psi_1(z,w;h) = \overline{\Psi_1(w,z;h)} $ 
  follows directly from the fact that $(e_0(z)|e_0(w)) = 
  \overline{(e_0(w)|e_0(z))}$. One then computes that 
  \begin{equation*}
   (\partial_z\Psi_1)(z,w;h) = 
   \partial_z\Psi_1(z,w;h) = 
   \overline{\partial_{\overline{z}}\Psi_1(w,z;h)}
   =
   \overline{(\partial_{\overline{w}}\Psi_1)(w,z;h)}
  \end{equation*}
 which concludes the proof of the Proposition.
  \end{proof}
 \begin{proof}[of Lemma \ref{lem:CP_ee}]
  We are interested in the solutions of the following equation:
 \begin{equation}
 \label{eqn:CP_psi+}
  0 =(\partial_x\widetilde{\psi}_+)(x,z) 
    - (\partial_x\widetilde{\psi}_+)^*(x,w) 
    = z - \overline{w} 
       -\widetilde{g}(x) 
	    + \widetilde{g}^*(x),
 \end{equation}
 where $\widetilde{g}$ denotes an almost analytic extension 
 of $g$. Since $\dist(\Omega,\partial\Sigma) > 1/C$, it follows 
 from the assumptions on $g$ that $\Ima g'(x) > 0$ for all 
 $x\in \overline{x_+(\Omega)}\subset\R$. Since $g$ depends 
 smoothly on $x$, there exists a small complex open 
 neighborhood $V\subset\C$ of $\overline{x_+(\Omega)}$ such that 
 $\overline{x_+(\Omega)}\subset (V\cap\R)$ and such that 
 for all $x \in V$
  \begin{equation*}
   \widetilde{g}'_x(x) 
	    - \overline{\widetilde{g}'_{x}(\overline{x})} \neq 0, 
	    \quad 
  \widetilde{g}'_{\overline{x}}(x) 
	    - \overline{\widetilde{g}'_{\overline{x}}(\overline{x})}  = \mO(|\Ima x|^{\infty}).
  \end{equation*}
 Thus, it follows by the implicit function theorem, that for 
 $(z,w)\in \Delta_{\Omega}(C)$, with $C>0$ large enough, there exists a 
 unique solution $x_{+}^c(z,w)$ to (\ref{eqn:CP_psi+}) 
 and it depends smoothly on $(z,w)\in \Delta_{\Omega}(C)$. Furthermore,
 we have that $x_+^c(z,z) = x_+(z)\in\R$. Taking the $z$- and $\overline{z}$-
 derivative of (\ref{eqn:CP_psi+}) at the critical point $x_+^c$ yields that 
 \begin{align}
 \label{eqn:cp_zder}
  &\partial_{z}x_{+}^c(z,w)=
   \frac{1 +\mO(|\Ima x_{+}^c(z,w)|^{\infty})}
	{ (\partial_x\widetilde{g})(x_{+}^c(z,w)) 
	    -(\partial_x\widetilde{g})^*(x_{+}^c(z,w))},
  \notag \\
   &\partial_{\overline{z}}x_{+}^c(z,w)=
   \frac{\mO(|\Ima x_{+}^c(z,w)|^{\infty})}
	{ (\partial_x\widetilde{g})(x_{+}^c(z,w)) 
	    -(\partial_x\widetilde{g})^*(x_{+}^c(z,w))}
 \end{align}
 and similarly that 
 \begin{align}
 \label{eqn:cp_wder}
  &\partial_{\overline{w}}x_{+}^c(z,w)=
   \frac{-1 +\mO(|\Ima x_{+}^c(z,w)|^{\infty})}
	{ (\partial_x\widetilde{g})(x_{+}^c(z,w)) 
	    -(\partial_x\widetilde{g})^*(x_{+}^c(z,w))},
  \notag \\
   & \partial_{w}x_{+}^c(z,w)=
   \frac{\mO(|\Ima x_{+}^c(z,w)|^{\infty})}
	{ (\partial_x\widetilde{g})(x_{+}^c(z,w)) 
	    -(\partial_x\widetilde{g})^*(x_{+}^c(z,w))} . 
 \end{align}
 Using that $\Ima x_+^c(z,z) =0$, one calculates 
 that for $z=w$ we have that 
  \begin{align}
   & (\partial_z x_{+}^c)(z,z) = \partial_z x_+(z)
	    = - (\partial_{\overline{w}}x_{+}^c)(z,z)
      ,
      \notag\\ 
    & \text{and} ~ 
   (\partial_{\overline{z}} x_{+}^c)(z,z) = 0
   =  (\partial_{w}x_{+}^c)(z,z),
  \end{align}
 where 
 \begin{equation*}
   \partial_z x_+(z)=
	  \frac{1 }{2 i\Ima g'(x_{+}(z))}.
 \end{equation*}
 Taylor's theorem implies that 
 \begin{equation*}
   x_{+}^c(z+\zeta,z+\omega) = x_{+}(z)
	 + \frac{\zeta - \overline{\omega} }{2 i\Ima g'(x_{+}(z))}
	 + \mO((\zeta,\omega)^2).
 \end{equation*}
 Recall that the principal symbol of the operator $P_h$ is given by 
 $p(\rho) = \xi + g(x)$ (cf \eqref{eq_I3}), 
 which implies that $\{p,\overline{p}\}(\rho_{\pm}(z) = 
 -2i \Ima g'(x_{\pm}(z))$. To conclude the symmetric form of the 
 Taylor expansion stated in the Lemma, we expand around 
 the point $(\frac{z+w}{2},\frac{z+w}{2})$, for $|z-w|$ small enough, with 
 $\zeta = \frac{z-w}{2}$ and $\omega = -\frac{z-w}{2}$, which is 
 possible since $\Omega$ is by \eqref{eq_i35} assumed to be convex. 
 \par
 Finally, by taking the imaginary part of the Taylor expansion 
 of $x_+^c$, we conclude by (\ref{eqn:cp_zder}) and 
 (\ref{eqn:cp_wder}) that 
 \begin{equation*}
  \partial_{w}x_+^c(z,w), \partial_{\overline{z}}x_+^c(z,w) 
   = 
   \mO(|z-w|^{\infty}). \qedhere
 \end{equation*}
 \end{proof}
\subsection{The Scalar Product $(f_0(w)|f_0(z))$}
\label{susec:SP_ff}
 We have, as in Section \ref{susec:SP_ee}, 
 \begin{prop}
 \label{prop:Sp_ff}
 Let $\Omega\Subset\Sigma$ be as in Hypothesis \ref{Hyp:H6} 
 and let $x_-(z)$ be as in \eqref{eqn:xpm}. 
 Then, there exists a constant $C>0$ such that for all $(z,w)\in 
 \Delta_{\Omega}(C):=\{(z,w)\in\Omega^2;~ |z-w|<1/C\}$ 
  \begin{equation*}
   (f_0(w)|f_0(z)) = 
      \e^{-\frac{1}{h}\Phi_2(z;h)}
      \e^{-\frac{1}{h}\Phi_2(w;h)}
      \e^{\frac{2}{h}\Psi_2(z,w;h)}
         + \mO\!\left(h^{\infty}\right),
  \end{equation*}
 where: 
 \begin{itemize}
  \item $\Phi_2(\cdot;h):\Omega \rightarrow \R$ is a family 
        of smooth functions depending only on $\Ima z$, 
        which satisfy
	\begin{equation*}
	  \Phi_2(z;h) = 
	   - \Ima \int_{x_-(z)}^{x_0} (z- g(y))dy 
	    +\frac{h}{4}\ln \left(\frac{\pi h}{\Ima g'(x_-(z))}\right)
	    +\mO(h^2).
        \end{equation*}
  \item $\Psi_2(\cdot,\cdot;h):\Delta_{\Omega}(C) \rightarrow \C$ is 
   a family of smooth functions which are almost $z$-holo\-morphic 
   and almost $w$-anti-holomorphic extensions from the diagonal 
   $\Delta:=\{(z,z); z\in\Omega\}\subset \Delta_{\Omega}(C)$ 
   of $\Phi_2(z;h)$, i.e. 
    \begin{equation*}
     \partial_{\overline{z}}\Psi_2,\partial_{w}\Psi_2
      = \mO(|z-w|^{\infty}), ~ \Psi_2(z,z;h) = \Phi_2\left(\frac{1}{2}(z-\overline{z});h\right)
    \end{equation*}
    Moreover, for $z,w\in \Delta_{\Omega}(C)$ with $|z-w|\ll 1$, 
    one has that 
	\begin{align*}
	  \Psi_2(z,w;h) 
	  = \sum_{|\alpha+\beta|\leq 2}&
	      \frac{1}{2^{|\alpha+\beta|}\alpha!\beta!}
	    \partial_z^{\alpha}\partial_{\overline{z}}^{\beta}
	      \Phi_2\left(\frac{z+w}{2};h\right) 
	    (z-w)^{\alpha}(\overline{w-z})^{\beta}
	    \notag \\
	    &+ 
	    \mO(|z-w|^3+h^{\infty}),
       \end{align*}
	and that 
	\begin{align*}
	 2\Rea &\Psi_2(z,w;h)-\Phi_2(z;h)- \Phi_2(w;h) 
	 \notag \\
	 &= -
        \partial_z\partial_{\overline{z}}
        \Phi_2\left(\frac{z+w}{2};h\right)  
         |z-w|^2(1 + \mO(|z-w|+h^{\infty}));
	\end{align*}
   \item the function $\Psi_2(z,w;h)$ has the following symmetries: 
      \begin{equation*}
       \Psi_2(z,w;h) = \overline{\Psi_2(w,z;h)} 
       \quad
       \text{and}
       \quad
       (\partial_z\Psi_2)(z,w;h) = 
       \overline{(\partial_{\overline{w}}\Psi_2)(w,z;h)} .
      \end{equation*}
 \end{itemize}
 \end{prop}
 \subsection{Link with the symplectic volume}
 Before the proof of Proposition \ref{prop:SP_XX}, 
 let us give a short description of the connection 
 between the functions $\Phi_1(z;h)$, $\Phi_2(z;h)$ 
 in Proposition \ref{prop:Sp_ee}, \ref{prop:Sp_ff}, 
 and the symplectic volume form on the phase space $T^*S^1$. 
 \begin{prop}
 \label{prop:Hess_Phas_SymplVol}
  Let $z\in\Omega\Subset\Sigma$ be as in \eqref{eq_i35} 
  and let $\Phi_1$ and $\Phi_2$ be 
  as in Propositions \ref{prop:Sp_ee} and \ref{prop:Sp_ff}. Furthermore, 
  let $p$ be the principal symbol of $P_h$ (cf \eqref{eq_I3}), let $\rho_{\pm}\in T^*S^1$ 
  be the two solutions to $p(\rho)=z$, see \eqref{eqn:xpm}. 
  Then, 
  \begin{align*}
   \sigma_h(z): &= \left[(\partial^2_{z\overline{z}}\Phi_1)(z;h)
     +(\partial^2_{z\overline{z}}\Phi_2)(z;h)\right]
     \notag \\
     &= \frac{1}{4}\left( 
   \frac{1}{\frac{1}{2i}\{\overline{p},p\}(\rho_{-}(z))}
   +
   \frac{1}{\frac{1}{2i}\{p,\overline{p}\}(\rho_{+}(z))}
   \right)
   +\mO(h)
  \end{align*}
  is, up to an error of order $h$, one-fourth of the Lebesgue density of the 
  direct image, under the principal symbol $p$, of the symplectic 
  volume form $d\xi \wedge dx$ on $T^*S^1$, i.e.
  \begin{equation*} 
   \sigma_h(z)L(dz) = \frac{1}{4} p_*(d\xi\wedge dx) + \mO(h)L(dz)
  \end{equation*}
 \end{prop}
 \begin{proof}
  Using that $x_{\pm}(t)$, with $t=\Ima z$, is the solution 
  to the equation $\Ima g(x_{\pm}(t)) = t$ with 
  \begin{equation*}
   \mp\Ima g'_x(x_{\pm}(t))<0
  \end{equation*}
  (cf \eqref{eqn:xpm}), we get that  
    \begin{equation*}
     x_{\pm}'(t) = \pm \frac{1}{ \Ima g'_x(x_{\pm}(t))} < 0.
    \end{equation*}
  Using Propositions \ref{prop:Sp_ee} and \ref{prop:Sp_ff}, 
  one then computes that 
  \begin{equation*}
   (\partial^2_{z\overline{z}}\Phi_1)(z;h) 
   + (\partial^2_{z\overline{z}}\Phi_2)(z;h)
   =  
   \frac{1}{4} \left( 
   \frac{1}{\Ima g'_x(x_{-}(\Ima z))}
   -
   \frac{1}{\Ima g'_x(x_{+}(\Ima z))}
   \right)  + \mO(h).
  \end{equation*}
  Since $-\frac{1}{2i}\{p,\overline{p}\}(\rho_{\pm}) 
  = \Ima g'_x(x_{\pm})$, we conclude by Proposition 6.2 in 
  \cite{Vo14} that 
  \begin{equation*}
   \left[\partial^2_{z\overline{z}}\Phi_1)(z;h) 
   + (\partial^2_{z\overline{z}}\Phi_2)(z;h)\right] L(dz)
   =  
    \frac{1}{4} p_*(d\xi\wedge dx) + \mO(h)L(dz).  \qedhere
  \end{equation*}
 \end{proof}
 \begin{proof}[of Proposition \ref{prop:SP_XX}]
  The results follow immediately from \eqref{eqn:SP_XX} and 
  the Propositions \ref{prop:Sp_ee}, \ref{prop:Sp_ff} and 
  \ref{prop:Hess_Phas_SymplVol}.
 \end{proof}
\section{Gramian matrix}
\label{susec:Gramian}
 The aim of this section is to study the Gramian matrix 
 $G$ which is defined in \eqref{eqn:G_44} via the blocks $A$, $B$, and $C$,  
 given in \eqref{eqn:ABC}. This will be essential to 
 the proof of Proposition \ref{prop:2ptCorrelation}. Most of 
 the results obtained here follow from involved but 
 straightforward calculations which use strongly 
 Proposition \ref{prop:SP_XX}, the principal 
 result of the previous section. 
 \par
 This section is organized as follows: in Section \ref{suse:GM1} we 
 discuss the invertibility of the matrix $A$ and provide estimates for 
 its determinant. In Section \ref{suse:GM2} we obtain detailed formulas for 
 $\Gamma$, which is given by the Shur complement 
 formula applied to $G$ (cf. \eqref{eqn:G_44}, \eqref{eq_a2}), i.e. 
 \begin{equation}\label{eq:Gam}
   \Gamma=C-B^*A^{-1}B.
 \end{equation}
 In Section \ref{suse:GM3} we will discuss the invertibility of the matrix 
 $G$ and in Section \ref{suse:GM4} we will state a formula for the permanent 
 of $\Gamma$ which is an essential quantity of Proposition 
 \ref{prop:2ptCorrelation}. 
 \subsection{The matrix $A$}\label{suse:GM1}
  We begin by studying the determinant of $A$, cf. \eqref{eqn:ABC}. It 
  is non-zero if and only if the vectors $X(z)$ and 
  $X(w)$ (given in Definition \ref{def:X}) are not co-linear. In particular we are 
  interested in a lower bound of this determinant 
  for $z$ and $w$ close. 
 \begin{prop}
 \label{prop:detA}
  Let $\Omega\Subset\Sigma$ be as in Hypothesis \ref{Hyp:H6} and let 
  $A$ be as in \eqref{eqn:ABC}. 
  For $z,w\in\Omega$ with $|z-w| \leq 1/C$, with $C>1$ large 
  enough (cf. Proposition \ref{prop:SP_XX}), we have
         \begin{align*}
	 \label{eqn:DetA}
	   \det A(z,w) &=  1 - 
	    \e^{-\frac{2K(z,w)}{h}}
	    +\mO_{\mathcal{C}^{\infty}}\!\left(h^{\infty}\right),
	 \end{align*}
  where $K(z,w)$ is as in \eqref{eqn:K}. Moreover,	 
 \begin{itemize}
  \item for $|z-w| \gg \sqrt{h\ln h^{-1}}$
	 \begin{align*}
	   \det A(z,w) =  1 
	    +\mO\!\left(h^{C}\right), ~ C\gg 1;
	 \end{align*}
  \item for $|z-w|\geq \frac{1}{\mO(1)} \sqrt{h}$
	 \begin{align*}
	   \det A \geq \frac{1}{\mO(1)};
	 \end{align*}
   \item let $N>1$ and let $C>1$ be large enough, then for 
   $\frac{1}{C} h^N \leq |z-w| \leq \frac{1}{C}\sqrt{h}$, 
	 \begin{align*}
	  \det A(z,w)  &= \frac{|z-w|^2}{2h} 
	  \left(\sigma\left(\frac{z+w}{2}\right) + 
	  \mO(h) + \mO(|z-w|)
	  + \mO\left(\frac{|z-w|^2}{h}\right)\right)
	  \notag \\
 	 & \phantom{=}+\mO_{\mathcal{C}^{\infty}}\!\left(h^{\infty}\right)
 	 \notag \\
	  &\geq \frac{h^{2N-1}}{\mO(1)}.
 	\end{align*}
 \end{itemize}
 \end{prop}
\begin{proof}
 By Corollary \ref{cor:DetPer_A} and \eqref{eqn:K}, one has that 
 \begin{equation*}
  \det A(z,w) = 1 - 
	    \e^{-\frac{2K(z,w)}{h}}
	    +\mO_{\mathcal{C}^{\infty}}\!\left(h^{\infty}\right), 
 \end{equation*}
 with 
 \begin{align*}
   K(z,w) = \left(\sigma\left(\frac{z+w}{2}\right)  +\mO(h)\right)
         \frac{|z-w|^2}{4}(1 + \mO(|z-w|+h^{\infty})).
  \end{align*}
 The first two estimates are then an immediate consequence of 
 the above formula.
 In the case where $|z-w| \leq \frac{1}{C}\sqrt{h}$, 
 one computes, using Taylor's formula, that
 \begin{equation*}
  \e^{-\frac{2K(z,w)}{h}} = 1 - \frac{|z-w|^2}{2h} 
  \left(\sigma\left(\frac{z+w}{2}\right) + 
  \mO(h) + \mO(|z-w|)
  + \mO\left(\frac{|z-w|^2}{h}\right)\right),
 \end{equation*}
 which implies that 
 \begin{equation*}
 \begin{split}
  \det A(z,w)  &= \frac{|z-w|^2}{2h} 
  \left(\sigma\left(\frac{z+w}{2}\right) + 
  \mO(h) + \mO(|z-w|)
  + \mO\left(\frac{|z-w|^2}{h}\right)\right)
 +\mO_{\mathcal{C}^{\infty}}\!\left(h^{\infty}\right)
  \\
  &\geq \frac{h^{2N-1}}{\mO(1)}. \qedhere
  \end{split}
 \end{equation*}
\end{proof}
 Since the matrix $A$ is self-adjoint, we have a lower bound on the matrix 
 norm of $A$ by its smallest eigenvalue. Using Proposition \ref{prop:SP_XX} we 
 see that $\tr A = 2 + \mO(h^{\infty})$ and one calculates that for a fixed 
 $N>1$ and for $|z-w|\geq \frac{ h^{N}}{\mO(1)}$ the two eigenvalues of $A$ 
 are given by 
 \begin{equation*}
   \lambda_{1,2}(z,w;h) = 1 \pm \e^{-\frac{K(z,w)}{h}} +\mO(h^{\infty}).
 \end{equation*}
 By Taylor expansion we conclude the following result:
\begin{cor}
\label{cor:LB_MN_A}
 Under the assumptions of Proposition \ref{prop:detA}, we have 
 that for $N\geq 1$ and $|z-w|\geq \frac{ h^{N}}{\mO(1)}$ 
 \begin{equation*}
   \min\limits_{\lambda \in\sigma(A)} \lambda
   \geq \frac{h^{2N-1}}{\mO(1)}.
 \end{equation*}
\end{cor}
\subsection{The matrix $\Gamma$}\label{suse:GM2}
The principal aim of this section is to prove a precise formula for the matrix $\Gamma$, 
see Proposition \ref{prop:mat_Gamma} below, and to give formulas for its determinant, 
permanent and trace, see Corollary \ref{cor:det_Gamma} below. 
\par 
We begin by considering a very helpful congruency transformation. 
In view of Proposition \ref{prop:SP_XX}, we prove 
\begin{lem}
 \label{lem:conj}
  Let $\Omega\Subset\Sigma$ be as in \eqref{eq_i35}, and let 
  $\Delta_{\Omega}(C)$, $\Phi(z;h)$ and $\Psi(z,w;h)$ be as in Proposition 
  \ref{prop:SP_XX}, for $(z,w)\in \Delta_{\Omega}(C)$. Let $\Gamma$ be as in 
 \eqref{eq:Gam}. Define the matrices
  \begin{equation*}
   \widetilde{A}:= 
   \begin{pmatrix}
    \e^{\frac{2}{h}\Psi(z,z;h)} 
	        &
    \e^{\frac{2}{h}\Psi(z,w;h)} \\ 
    \e^{\frac{2}{h}\Psi(w,z;h)} 
	        &
    \e^{\frac{2}{h}\Psi(w,w;h)} \\ 	        
  \end{pmatrix}
  \quad 
  \text{and}
  \quad  
   \Lambda : = \begin{pmatrix}
          \e^{-\frac{1}{h}\Phi(z;h)} & 
          0 \\ 
          0 & 
          \e^{-\frac{1}{h}\Phi(w;h)} \\ 
          \end{pmatrix},
  \end{equation*}
  \begin{equation*}
   \widetilde{B}:= 
   2h^{-1}\begin{pmatrix}
    \Psi'_{\overline{w}}(z,z;h)\e^{\frac{2}{h}\Psi(z,z;h)} 
	        &
    \Psi'_{\overline{w}}(z,w;h)\e^{\frac{2}{h}\Psi(z,w;h)} \\ 
    \Psi'_{\overline{w}}(w,z;h)\e^{\frac{2}{h}\Psi(w,z;h)} 
	        &
    \Psi'_{\overline{w}}(w,w;h)\e^{\frac{2}{h}\Psi(w,w;h)} \\ 	        
  \end{pmatrix}
  \end{equation*}
  and 
  \begin{equation*}
   \widetilde{C}:= 
   h^{-2}\begin{pmatrix}
    c(z,z;h)\e^{\frac{2}{h}\Psi(z,z;h)} 
	        &
    c(z,w;h)\e^{\frac{2}{h}\Psi(z,w;h)} \\ 
    c(w,z;h)\e^{\frac{2}{h}\Psi(w,z;h)} 
	        &
    c(w,w;h)\e^{\frac{2}{h}\Psi(w,w;h)} \\ 	        
  \end{pmatrix}
  \end{equation*}
  with $c(z,w;h) := 4\Psi'_{z}(z,w;h)\Psi'_{\overline{w}}(z,w;h) 
  + 2 h \Psi''_{z\overline{w}}(z,w;h)$. Then, we have for 
  $|z-w| \geq h^{N}/\mO(1)$ that 
  \begin{equation*}
   \Gamma = \Lambda (\widetilde{C} -
    \widetilde{B}^*\widetilde{A}^{-1} \widetilde{B})\Lambda
   +\mO_{\mathcal{C}^{\infty}}\!\left(h^{\infty}\right).
  \end{equation*} 
 \end{lem}
\begin{proof}
To abbreviate the notation, we define for $(z,w)\in D_{\Omega}(C)$ 
the following function
 \begin{equation*}
  F(z,w) :=  \e^{-\frac{1}{h}\Phi(z;h)}
 \e^{-\frac{1}{h}\Phi(w;h)}
 \e^{\frac{2}{h}\Psi(z,w;h)}.
 \end{equation*}
By Proposition \ref{prop:SP_XX}, we see that $F$ is bounded by $1$ and 
that all its derivatives are bounded polynomially in $h^{-1}$. Furthermore, 
the matrices $A,B$ and $C$ are given by 
\begin{align*}
  &A(z,w) = A_0(z,w) + 
        \mO_{\mathcal{C}^{\infty}}\!\left(h^{\infty}\right), 
        \notag\\
  &B(z,w) = B_0(z,w) + 
        \mO_{\mathcal{C}^{\infty}}\!\left(h^{\infty}\right), 
        \notag\\
   &C(z,w) = C_0(z,w) + 
        \mO_{\mathcal{C}^{\infty}}\!\left(h^{\infty}\right), 
\end{align*}
where $(z,w)\in D_{\Omega}(C)$ and 
\begin{align*}
 A_0(z,w) = 
 \begin{pmatrix}
  F(z,z)
 &
 F(z,w)
 \\
 F(w,z)
 &
 F(w,w)
 \\
 \end{pmatrix},
\end{align*}
and 
\begin{align*}
 B_0(z,w) = 
 \begin{pmatrix}
  (\partial_{\overline{w}}F)(z,z)
  &
 (\partial_{\overline{w}}F)(z,w)
 \\
 (\partial_{\overline{w}}F)(w,z)
 &
 (\partial_{\overline{w}}F)(w,w)
 \\
 \end{pmatrix},
\end{align*}
and 
\begin{align*}
 C_0(z,w) = 
 \begin{pmatrix}
  (\partial_{z\overline{w}}^2F)(z,z)
  &
 (\partial_{z\overline{w}}^2F)(z,w)
 \\
 (\partial_{z\overline{w}}^2F)(w,z)
 &
 (\partial_{z\overline{w}}^2F)(w,w)
 \\
 \end{pmatrix}.
\end{align*}
One computes that 
\begin{align*}
  (\partial_{\overline{w}}F)(z,w)
  &= \frac{1}{h}\left[2(\partial_{\overline{w}}\Psi)(z,w;h)
  -(\partial_{\overline{w}})\Phi(w;h)\right]\e^{-\frac{1}{h}\Phi(z;h)-\frac{1}{h}\Phi(w;h)}
	        \e^{\frac{2}{h}\Psi(z,w)}
    \notag \\
    &\phantom{=} 
      +\mO_{\mathcal{C}^{\infty}}\!\left(h^{\infty}\right),  
 \end{align*}
and that 
\begin{align*}
 (\partial_{z\overline{w}}^2 &F)(z,w) \notag \\
  &= \frac{1}{h^2}\Big[
   \left[2(\partial_{z}\Psi)(z,w;h)
  -(\partial_{z}\Phi)(z;h)\right]
  \left[2(\partial_{\overline{w}}\Psi)(z,w;h)
  -(\partial_{\overline{w}}\Phi)(w;h)\right]
  +
  \notag \\
  &\phantom{= \frac{1}{h^2}-}
  2h(\partial_{z\overline{w}}^2\Psi)(z,w;h)
    \Big]\e^{-\frac{1}{h}\Phi(z_1;h)-\frac{1}{h}\Phi(z_2;h)}
    \e^{\frac{2}{h}\Psi(z_1,z_2)}    
      +\mO_{\mathcal{C}^{\infty}}\!\left(h^{\infty}\right).  
 \end{align*}
Using that $ \det A_0 =  \det A + \mO(h^{\infty})$ and that 
$\det A \geq h^{2N-1}/\mO(1)$ for $|z-w| \geq h^{N}/\mO(1)$ 
(cf. Proposition \ref{prop:detA}), we see that 
\begin{equation*}
 \Gamma = C_0 - B_0^* A_0^{-1} B_0
 +\mO\!\left(h^{\infty}\right).
\end{equation*}
Defining, 
\begin{equation*}
\Lambda' : = \begin{pmatrix}
          \partial_z\e^{-\frac{1}{h}\Phi(z;h)} & 
          0 \\ 
          0 & 
          \partial_w\e^{-\frac{1}{h}\Phi(w;h)} \\ 
          \end{pmatrix}
  \end{equation*}
we see that 
\begin{align*}
 &A_0 = \Lambda \widetilde{A}\Lambda ,
 \notag \\
 &B_0 = \Lambda (\widetilde{B})\Lambda + \Lambda \widetilde{A}(\Lambda')
  +\mO_{\mathcal{C}^{\infty}}\!\left(h^{\infty}\right),
 \notag \\
 & C_0 = \Lambda (\widetilde{C})\Lambda + \Lambda (\widetilde{B}^*) (\Lambda')
  + \Lambda' (\widetilde{B}) \Lambda+
  \Lambda' \widetilde{A}(\Lambda')
  +\mO_{\mathcal{C}^{\infty}}\!\left(h^{\infty}\right).
\end{align*}
A direct computation then yields that 
\begin{equation*}
 \Gamma = \Lambda (\widetilde{C} 
    - \widetilde{B}^*\widetilde{A}^{-1} \widetilde{B})\Lambda
 +\mO_{\mathcal{C}^{\infty}}\!\left((\det A)^{-1}h^{\infty}\right). \qedhere
\end{equation*}
\end{proof}
\begin{prop}
\label{prop:mat_Gamma}
 Let $\Omega\Subset\Sigma$ be as in \eqref{eq_i35}, and let 
 $\Delta_{\Omega}(C)$ and $\Psi(z,w;h)$, for $(z,w)\in \Delta_{\Omega}(C)$,  
 be as in Proposition \ref{prop:SP_XX}. Let $\Gamma$ be as in 
 \eqref{eq:Gam}. For $(z,w)\in D_{\Omega}(C)$ 
 let $K(z,w)$ be as in \eqref{eqn:K} and define
 \begin{align*}
  &a_1:=a_1(z,w;h):= (\partial_{z}\Psi)(z,z;h) - (\partial_{z}\Psi)(z,w;h),
  \\
  &a_2:=a_2(z,w;h):= -a_1(w,z;h).
 \end{align*}
 Then, for $N>1$ and $ \frac{1}{C} h^N \leq |z-w|$, with $C>1$ large enough, 
 we have that 
\begin{align*}
 \Gamma =& 
 \frac{-4}{h^2\left(1 - \e^{-\frac{2}{h}K(z,w)}\right)}
 \begin{pmatrix}
 a_1\overline{a}_1 
 \e^{- \frac{2}{h}K(z,w)}
 &
 a_1\overline{a}_2 
 \e^{\frac{1}{h}( 2i\Ima\Psi(z,w)- K(z,w))}
 \\
 a_2\overline{a}_1 
 \e^{\frac{1}{h}( -2i\Ima\Psi(z,w)- K(z,w))}
 &
 a_2\overline{a}_2 
  \e^{-\frac{2}{h}K(z,w)}
  \\
 \end{pmatrix}
 \notag \\
 &+
 \frac{2}{h}
 \begin{pmatrix}
    \Psi''_{z\overline{w}}(z,z;h)
	        &
    \Psi''_{z\overline{w}}(z,w;h)\e^{\frac{1}{h}(2i\Ima\Psi(z,w)- K(z,w))} \\ 
    \Psi''_{z\overline{w}}(w,z;h)\e^{\frac{1}{h}(-2i\Ima\Psi(z,w)- K(z,w))} 
	        &
    \Psi''_{z\overline{w}}(w,w;h) \\ 	        
  \end{pmatrix}
  \notag \\
 & + \mO(h^{\infty}).
\end{align*}
\end{prop}
Before we give the proof of this result, we state formulae for the trace, 
the determinant and the permanent of $\Gamma$.
\begin{cor}
\label{cor:det_Gamma}
 Under the assumptions of Proposition \ref{prop:mat_Gamma}, we have that 
 \begin{align*}
 \mathrm{tr}\,\Gamma = 
 \frac{2}{h\left(\e^{\frac{2}{h}K(z,w)}-1\right)}
 \Big[&
 \Big(\Psi''_{z\overline{w}}(z,z;h) + \Psi''_{z\overline{w}}(w,w;h)
 +\mO(h^{\infty})\Big)
 \left(\e^{\frac{2}{h}K(z,w)}-1\right)
 \notag \\
 &-
 2h^{-1}(|a_1|^2 + |a_2|^2)
 \Big],
\end{align*}
\begin{align*}
 \det&\Gamma = 
 -\frac{16}{h^4\left(1 - \e^{-\frac{2}{h}K(z,w)}\right)}\e^{-\frac{2}{h} K(z,w)}
 \Big[|a_{1}a_{2}|^2
 +\frac{h}{2}
 \big(
 |a_{1}|^2
 (\partial_{z\overline{w}}^2\Psi)(w,w;h)
 \notag \\
 & \phantom{---------}
 -
 2\Rea\left\{
 (\partial_{z\overline{w}}^2\Psi)(w,z;h)
 a_1\overline{a}_2
 \right\}
 +|a_{2}|^2
 (\partial_{z\overline{w}}^2\Psi)(z,z;h)
 \big)\Big]
 \notag \\
  &+
 \frac{4}{h^2}
 \left(
 (\partial_{z\overline{w}}^2\Psi)(z,z;h)
 (\partial_{z\overline{w}}^2\Psi)(w,w;h)
 -
 (\partial_{z\overline{w}}^2\Psi)(z,w;h)
 (\partial_{z\overline{w}}^2\Psi)(w,z;h)
 \e^{-\frac{2}{h}K(z,w)}
 \right)
 \notag \\
 & + \mO(h^{\infty})
\end{align*}
and that 
\begin{align*}
 &\perm\Gamma = 
 \frac{16}{h^4\left(1 - \e^{-\frac{2}{h}K(z,w)}\right)^2}\e^{-\frac{2}{h} K(z,w)}
  |a_{1}a_{2}|^2\left(1+\e^{-\frac{2}{h} K(z,w)}\right)
  \notag \\
 &\phantom{per}-
 \frac{8}{h^3\left(1 - \e^{-\frac{2}{h}K(z,w)}\right)}\e^{-\frac{2}{h} K(z,w)}
 \big(
 |a_{1}|^2
 (\partial_{z\overline{w}}^2\Psi)(w,w;h)
 \notag \\
 & \phantom{-----------------}
 +
 2\Rea\left\{
 (\partial_{z\overline{w}}^2\Psi)(w,z;h)
 a_1\overline{a}_2
 \right\}
 +|a_{2}|^2
 (\partial_{z\overline{w}}^2\Psi)(z,z;h)
 \big)
 \notag \\
  &\phantom{per}+
 \frac{4}{h^2}
 \left(
 (\partial_{z\overline{w}}^2\Psi)(z,z;h)
 (\partial_{z\overline{w}}^2\Psi)(w,w;h)
 +
 (\partial_{z\overline{w}}^2\Psi)(z,w;h)
 (\partial_{z\overline{w}}^2\Psi)(w,z;h)
 \e^{-\frac{2}{h}K(z,w)}
 \right)
 \notag \\
 &\phantom{per} + \mO(h^{\infty}).
\end{align*}
\end{cor}
\begin{proof}
 The result follows from a direct computation using 
 Proposition \ref{prop:mat_Gamma}; for the definition of the permanent 
 of a matrix see \eqref{def:Per}.
\end{proof}
\begin{proof}[of Proposition \ref{prop:mat_Gamma}]
In view of Lemma \ref{lem:conj}, it remains to consider 
the matrix
\begin{equation*}
 \widetilde{\Gamma}:= 
\widetilde{C} - \widetilde{B}^* \widetilde{A}^{-1} \widetilde{B}.
\end{equation*}
In the sequel we will suppress the $h$-dependency of 
the function $\Psi$ to abbreviate our notation. Recall the definition 
of $\widetilde{A}$ from Lemma \ref{lem:conj} and note that   
\begin{align}
 \label{eqn:DetAtilde}
 \det \widetilde{A} &=\e^{\frac{2}{h}\Psi(z,z)} \e^{\frac{2}{h}\Psi(w,w)} 
 -\e^{\frac{4}{h}\Rea\Psi(z,w)}
 \notag \\
 &=\e^{\frac{2}{h}\Psi(z,z)} \e^{\frac{2}{h}\Psi(w,w)} 
 \left(1 - \e^{-\frac{2}{h}K(z,w)}\right).
\end{align}
For $\frac{1}{C} h^N \leq |z-w|$, Proposition \ref{prop:SP_XX} implies 
that $\det \widetilde{A}$ is positive. Hence, the inverse of $\widetilde{A}$ 
exists and is given by 
\begin{equation*}
 \widetilde{A}^{-1}:= \frac{1}{\det \widetilde{A}}
   \begin{pmatrix}
    \e^{\frac{2}{h}\Psi(w,w)} 
	        &
    -\e^{\frac{2}{h}\Psi(z,w)} \\ 
    -\e^{\frac{2}{h}\Psi(w,z)} 
	        &
    \e^{\frac{2}{h}\Psi(z,z)} \\ 	        
  \end{pmatrix}.
\end{equation*}
To calculate $\widetilde{B}^*$, we use Lemma 
\ref{lem:conj} and the symmetries of the function $\Psi(z,w)$ 
given in Proposition \ref{prop:SP_XX}. Indeed, one gets that 
\begin{equation*}
   \widetilde{B}^*:= 
   2h^{-1}\begin{pmatrix}
    \Psi'_{z}(z,z)\e^{\frac{2}{h}\Psi(z,z)} 
	        &
    \Psi'_{z}(z,w)\e^{\frac{2}{h}\Psi(z,w)} \\ 
    \Psi'_{z}(w,z)\e^{\frac{2}{h}\Psi(w,z)} 
	        &
    \Psi'_{z}(w,w)\e^{\frac{2}{h}\Psi(w,w)} \\ 	        
  \end{pmatrix}
  \end{equation*}
and one computes that $M:=
h\widetilde{B}^* \widetilde{A}^{-1} h\widetilde{B}$ 
is given by 
\begin{equation*}
  M = \frac{4}{\det \widetilde{A}} 
 \begin{pmatrix}
    M_{11} & M_{12}   \\ 
    M_{21} & M_{22}   \\ 	        
  \end{pmatrix}
\end{equation*}
with 
\begin{align*}
 M_{11} =& \Psi'_{z}(z,z)\Psi'_{\overline{w}}(z,z)
 \e^{\frac{1}{h}(4\Psi(z,z) + 2\Psi(w,w))} +
 \big[\Psi'_{z}(z,w)\Psi'_{\overline{w}}(w,z)
  \notag \\
 &-\Psi'_{z}(z,w)\Psi'_{\overline{w}}(z,z)
 -\Psi'_{z}(z,z)\Psi'_{\overline{w}}(w,z)\big]
 \e^{\frac{1}{h}( 2\Psi(z,z) + 4 \Rea \Psi(z,w))},
\end{align*}
\begin{align*}
 M_{12} =& -\Psi'_{z}(z,w)\Psi'_{\overline{w}}(z,w)
 \e^{\frac{1}{h}(4\Psi(z,w) + 2\Psi(w,z))} +
 \big[\Psi'_{z}(z,z)\Psi'_{\overline{w}}(z,w)
  \notag \\
 &+\Psi'_{z}(z,w)\Psi'_{\overline{w}}(w,w)
 -\Psi'_{z}(z,z)\Psi'_{\overline{w}}(w,w)\big]
 \e^{\frac{2}{h}( \Psi(z,z) + \Psi(z,w)+ \Psi(w,w))},
\end{align*}
and
\begin{align*}
 M_{22} =& \Psi'_{z}(w,w)\Psi'_{\overline{w}}(w,w)
 \e^{\frac{1}{h}(2\Psi(z,z) + 4 \Psi(w,w))} +
 \big[\Psi'_{z}(w,z)\Psi'_{\overline{w}}(z,w)
  \notag \\
 &-\Psi'_{z}(w,w)\Psi'_{\overline{w}}(z,w)
 -\Psi'_{z}(w,z)\Psi'_{\overline{w}}(w,w)\big]
 \e^{\frac{1}{h}( 2\Psi(w,w) + 4 \Rea \Psi(z,w))}.
\end{align*}
Since the matrix $M$ is clearly self-adjoint, one has 
that $M_{21} = \overline{M}_{12}$. Comparing the coefficients 
of $M$ with with those of $ h^2 (\det\widetilde{A}/4) \widetilde{C}$ 
(cf. Lemma \ref{lem:conj}) and using the symmetries of $\Psi$ (cf. 
Proposition \ref{prop:SP_XX}), we see that
\begin{align}
\label{eqn:Gam_01}
 h^2 \widetilde{\Gamma} =& 
 \frac{-4}{\det \widetilde{A}}
 \begin{pmatrix}
 a_1\overline{a}_1 
 \e^{\frac{1}{h}( 2\Psi(z,z) + 4 \Rea \Psi(z,w))}
 &
 a_1\overline{a}_2 
 \e^{\frac{2}{h}( \Psi(z,z) + \Psi(z,w)+ \Psi(w,w))}
 \\
 a_2\overline{a}_1 
 \e^{\frac{2}{h}( \Psi(z,z) + \Psi(w,z)+ \Psi(w,w))}
 &
 a_2\overline{a}_2 
  \e^{\frac{1}{h}( 2\Psi(w,w) + 4 \Rea \Psi(z,w))}
  \\
 \end{pmatrix}
 \notag \\
 &+
 2h
 \begin{pmatrix}
    \Psi''_{z\overline{w}}(z,z;h)\e^{\frac{2}{h}\Psi(z,z)} 
	        &
    \Psi''_{z\overline{w}}(z,w;h)\e^{\frac{2}{h}\Psi(z,w)} \\ 
    \Psi''_{z\overline{w}}(w,z;h)\e^{\frac{2}{h}\Psi(w,z)} 
	        &
    \Psi''_{z\overline{w}}(w,w;h)\e^{\frac{2}{h}\Psi(w,w)} \\ 	        
  \end{pmatrix}
\end{align}
with $a_{i}$ as in the hypothesis of Proposition \ref{prop:mat_Gamma}. 
Recall from \eqref{eqn:K} that the function $K(z,w)$ is defined by
\begin{align*}
   -K(z,w) = 2\Rea \Psi(z,w)-\Phi(z)- \Phi(w) 
\end{align*}
where $\Phi(z) = \Psi(z,z)$. Using \eqref{eqn:DetAtilde}, we find that the 
first matrix in \eqref{eqn:Gam_01} is equal to 
\begin{align*}
 \frac{-4}{1 - \e^{-\frac{2}{h}K(z,w)}}
 \begin{pmatrix}
 a_1\overline{a}_1 
 \e^{\frac{1}{h}(2\Psi(z,z) - 2K(z,w))}
 &
 a_1\overline{a}_2 
 \e^{\frac{2}{h}\Psi(z,w)}
 \\
 a_2\overline{a}_1 
 \e^{\frac{2}{h}\Psi(w,z)}
 &
 a_2\overline{a}_2 
  \e^{\frac{1}{h}( 2\Psi(w,w) -2K(z,w))}
  \\
 \end{pmatrix}.
\end{align*}
It follows by Lemma \ref{lem:conj} that 
\begin{align*}
 \Gamma &= \Lambda \widetilde{\Gamma}\Lambda^*
 +\mO_{\mathcal{C}^{\infty}}\!\left(h^{\infty}\right).
\end{align*}
In the last equality we used that $\det A$ is bounded from 
below by a power of $h$; see Lemma \ref{lem:conj}. Carrying out 
the matrix multiplication $\Lambda \widetilde{\Gamma}\Lambda^*$ 
implies the statement of the proposition.
\end{proof}
\subsection{The determinant of $G$}\label{suse:GM3} We show that 
the matrix $G(z,w)$ (cf. \eqref{eqn:G_44}) is invertible 
if $z$ and $w$ are outside a neighborhood of size of order $h^{3/5}$ 
of the diagonal $\{z=w\}$. More precisely, we prove the following result:
 \begin{prop}
 \label{prop:GramInvert}
  Let $\Omega\Subset\Sigma$ be as in \eqref{eq_i35} and 
  let $z,w\in\Omega$. Then, 
  \begin{equation*}
   \det G(z,w) > 0 
   \quad 
   \text{for}
   \quad 
    h^{\frac{3}{5}} \ll |z-w| \ll 1. 
  \end{equation*}
 \end{prop}
 \begin{proof}
The Shur complement formula yields that the determinant of the 
Gra\-mian matrix $G$ is given by 
$\det G = \det A \det \Gamma $. Hence, using Proposition 
\ref{prop:detA} and Corollary \ref{cor:det_Gamma}, we see that
\begin{align}
 \label{eqn:detG}
 \det G &= 
 -\frac{16\left(1 + \mO(h^{\infty})\right)}{h^4}
 \e^{-\frac{2}{h} K(z,w)}
 \Big[|a_{1}a_{2}|^2
 +\frac{h}{2}
 \big(
 |a_{1}|^2
 (\partial_{z\overline{w}}^2\Psi)(w,w;h)
 \notag \\
 & \phantom{---------}
 -
 2\Rea\left\{
 (\partial_{z\overline{w}}^2\Psi)(w,z;h)
 a_1\overline{a}_2
 \right\}
 +|a_{2}|^2
 (\partial_{z\overline{w}}^2\Psi)(z,z;h)
 \big)\Big]
 \notag \\
  &+
 \frac{4}{h^2}
 \left(
 (\partial_{z\overline{w}}^2\Psi)(z,z;h)
 (\partial_{z\overline{w}}^2\Psi)(w,w;h)
 -
 (\partial_{z\overline{w}}^2\Psi)(z,w;h)
 (\partial_{z\overline{w}}^2\Psi)(w,z;h)
 \e^{-\frac{2}{h}K(z,w)}
 \right)
 \notag \\
 & ~~\cdot\left(1 - \e^{-\frac{2}{h} K(z,w)} + 
 \mO(h^{\infty})\right) + \mO(h^{\infty}).
\end{align}
 Next, we consider the Taylor expansion of the terms $a_1$ 
 and $a_2$ up to first order. Similarly as in Proposition 
 \ref{prop:SP_XX}, we develop around the point $(\frac{z+w}{2},
 \frac{z+w}{2})$ and get that 
 \begin{align}
  \label{eqn:Tay1}
  a_1 &= (\partial_z\Psi)(z,z) - (\partial_z\Psi)(z,w) 
  \notag \\
  &=(\partial_{z\bar{w}}^2\Psi)\left(\frac{z+w}{2},\frac{z+w}{2}\right)
  \left(z-w\right)
  +\mO(|z-w|^2 + h^{\infty})
 \end{align}
and 
\begin{align}
\label{eqn:Tay2}
  a_2 &= (\partial_z\Psi)(w,z) - (\partial_z\Psi)(w,w) 
  \notag \\
  &=(\partial_{z\bar{w}}^2\Psi)\left(\frac{z+w}{2},\frac{z+w}{2}\right)
  \left(z-w\right)
  +\mO(|z-w|^2 + h^{\infty}).
 \end{align}
Moreover, one has that for $\zeta,\omega\in\{z,w\}$
\begin{equation}
\label{eqn:Tay3}
 (\partial_{z\bar{w}}^2\Psi)\left(\zeta,\omega\right)
 =
 (\partial_{z\bar{w}}^2\Psi)\left(\frac{z+w}{2},\frac{z+w}{2}\right)
 +\mO(|z-w| + h^{\infty}).
\end{equation}
Since we suppose that $|z-w|\gg h^{3/5}$, the above error term is 
equal to $\mO(|z-w|)$. Since $\partial_{z\bar{w}}^2\Psi$ 
is evaluated at a point on the diagonal, it follows from Proposition \ref{prop:SP_XX}, 
that
\begin{align}
 \label{eqn:AbrSigh}
 (\partial_{z\bar{w}}^2\Psi)\left(\frac{z+w}{2},\frac{z+w}{2}\right)
 &=
 (\partial_{z\bar{z}}^2\Phi)\left(\frac{z+w}{2},\frac{z+w}{2}\right)
 \notag \\
 &=\frac{1}{4}\sigma\left(\frac{z+w}{2}\right) + \mO(h)
 =:\frac{1}{4}\sigma_{h}(z,w).
\end{align}
Plugging the above Taylor expansion into \eqref{eqn:detG}, one gets 
that $\det G$ is equal to 
\begin{align*}
 &
 \frac{ \sigma_{h}(z,w)^2}{4h^2}
 \bigg\{
 \Big[1+\mO(|z-w|) - (1+\mO(|z-w|))\e^{-\frac{2}{h} K(z,w)}\Big]
 \left(1 - \e^{-\frac{2}{h} K(z,w)} + 
 \mO(h^{\infty})\right)
 \notag \\
 &\phantom{=}\left.
 -4 \e^{-\frac{2}{h} K(z,w)}
 \left(
 \left(
 \frac{\sigma_{h}(z,w)|z-w|^2}{4h}
 \right)^2(1+\mO(|z-w|) )
 +
 \frac{\sigma_{h}(z,w)|z-w|^2}{4h}
 \mO(|z-w|)
 \right)
 \right\}
 \notag \\
 &\phantom{=}
 + \mO(h^{\infty})
 \notag \\
 &=
 \frac{ \sigma_{h}(z,w)^2}{4h^2}
 \bigg\{
 \left(
  1 - \e^{-\frac{2}{h} K(z,w)}
 \right)^2
 +
 \mO(|z-w|) \left(1 - \e^{-\frac{2}{h} K(z,w)}\right) + 
 \mO(h^{\infty})
 \notag \\
 &\phantom{=}\left.
 -4 \e^{-\frac{2}{h} K(z,w)}
 \left[
 \left(
 \frac{\sigma_{h}(z,w)|z-w|^2}{4h}
 \right)^2
 +
 \mO\left(\frac{|z-w|^5}{h^2}\right)
 +
 \mO\left(\frac{|z-w|^3}{h}\right)
 \right]
 \right\}.
\end{align*}
Recall from \eqref{eqn:K} that $K(z,w) \asymp |z-w|^2$, wherefore we see that 
$\det G$ is positive for $|z-w|\gg \sqrt{h}$. Next, we suppose that 
$ |z-w| \asymp\sqrt{h}$. Hence, one gets that 
\begin{align}
\label{eqn:detG2}
 \det G 
 &=
 \frac{\sigma_{h}(z,w)^2\e^{-\frac{2}{h} K(z,w)} }{h^2}
 \bigg\{
  \sinh^2 \frac{K(z,w)}{h}
 +
 \mO(|z-w|) \left(\e^{\frac{2}{h} K(z,w)}-1\right) + 
 \mO(h^{\infty})
 \notag \\
 &\phantom{=}\left.
 -
 \left[
 \left(
 \frac{\sigma_{h}(z,w)|z-w|^2}{4h}
 \right)^2
 +
 \mO\left(\frac{|z-w|^5}{h^2}\right)
 +
 \mO\left(\frac{|z-w|^3}{h}\right)
 \right]
 \right\}.
\end{align}
Using the Taylor expansion of the $\sinh x$ and \eqref{eqn:K}, one gets that 
\begin{align}
 \label{eqn:sinh_db}
 \sinh^2& \frac{K(z,w)}{h} - \left(
 \frac{\sigma_{h}(z,w)|z-w|^2}{4h}
 \right)^2 
 \notag \\
 &\geq 
 \left(
 \frac{1}{3}\frac{\sigma_{h}(z,w)|z-w|^2}{4h}
 \right)^4(1 + \mO(|z-w|))
 +
 \mO\!\left(
 \frac{\sigma_{h}(z,w)|z-w|^5}{h^2}
 \right).
\end{align}
Note that the principal term on the right
hand side of the inequality dominates the error terms. The same 
holds true for the other error terms in \eqref{eqn:detG2}. 
\par
Next, let us suppose that $h^{3/5} \ll|z-w| \ll \sqrt{h}$. Since 
\begin{equation*}
 \mO(|z-w|) \left(\e^{\frac{2}{h} K(z,w)}-1\right) 
 = \mO\left(\frac{|z-w|^3}{h}\right),
\end{equation*}
it follows by \eqref{eqn:detG2} and \eqref{eqn:sinh_db} that $\det G $ 
is positive for $|z-w|\gg h^{3/5}$.
\end{proof}
\subsection{The permanent of $\Gamma$}\label{suse:GM4}
The permanent of the matrix $\Gamma$ (cf. \eqref{eq:Gam}) is vital to the $2$-point density 
of eigenvalues and therefore, we shall give a more detailed description 
of it than the one given in Corollary \ref{cor:det_Gamma}.
\par 
We begin by proving the following bound on the trace of $\Gamma$: 
\begin{prop}
 \label{prop:Trace_est_Gam}
 Under the assumptions of Proposition \ref{prop:mat_Gamma}, 
 we have that for $|z-w| \gg h$ 
 \begin{equation*}
   0  < \tr\Gamma \leq \mO(h^{-1}).
 \end{equation*}
\end{prop}
\begin{proof}
 Using \eqref{eqn:Tay1}, \eqref{eqn:Tay2} and \eqref{eqn:Tay3}, 
 one gets that 
 \begin{align}
  \label{eqn:tay_tr}
  \tr\Gamma = 
  \frac{\sigma_{h}(z,w)}{2h\left(\e^{\frac{2}{h} K(z,w)} - 1\right) }
  \Big[&
  \left(\e^{\frac{2}{h} K(z,w)} - 1\right) 
  (1 + \mO(|z-w|))
  \notag \\
  &-
  \frac{\sigma_{h}(z,w)|z-w|^2}{2h}(1 + \mO(|z-w|))
  \Big].
 \end{align}
Since 
\begin{equation*}
 \e^{\frac{2}{h} K(z,w)} - 1 
 \geq 
 \frac{\sigma_{h}(z,w)|z-w|^2}{2h}(1 + \mO(|z-w|))
 +
 \frac{\sigma_{h}(z,w)|z-w|^4}{8h^2}(1 + \mO(|z-w|)),
\end{equation*}
it follows that for $|z-w| \gg h$ the trace of $\Gamma$ 
is positive. Furthermore, the above inequality applied to 
\eqref{eqn:tay_tr}, implies the upper bound stated in 
the Proposition.
\end{proof}
\begin{prop}
\label{prop:perm_Gamma}
 Let $\sigma_h(z,w)$ be as in Theorem \ref{thm_H2} and let $K(z,w)$ be 
 as in \eqref{eqn:K}. Under the assumptions 
 of Proposition \ref{prop:mat_Gamma}, we have that for $N>1$ and 
 $ \frac{1}{C} h^N \leq |z-w|$, 
 \begin{align*}
  &\perm\Gamma(z,w;h) \notag \\
   &=  \frac{1}{4h^{2}}
	  \Bigg[
	  \sigma_h(z,z)\sigma_h(w,w) + \sigma_h(z,w)^2(1 + \mO(|z-w|))
	      \e^{-\frac{2K(z,w)}{h}}+ \mO(h^{\infty})
	  \notag \\
	  &+ 
	  \frac{\sigma_h(z,w)^2(1 + \mO(|z-w|)) }{\e^{\frac{K(z,w)}{h}}\sinh \frac{K(z,w)}{h}}
	  \left(
	  \left(\frac{\sigma_h(z,w)|z-w|^2}{4h}\right)^2 2\coth \frac{K(z,w)}{h} 
	  - \frac{\sigma_h(z,w)|z-w|^2}{h} 
	  \right)
	  \Bigg].
 \end{align*}
\end{prop}
\begin{proof}
 Applying \eqref{eqn:Tay1}, \eqref{eqn:Tay2} and \eqref{eqn:Tay3} 
 to the formula for $\perm \Gamma$ given in Proposition 
 \ref{prop:Trace_est_Gam} and using the notation introduced in 
 \eqref{eqn:AbrSigh}, one gets that 
 \begin{align*}
 \perm\Gamma =& 
 \frac{8\coth\frac{K}{h}}{h^4\sinh \frac{K}{h}}\e^{-\frac{1}{h} K(z,w)}
  |4^{-2}\sigma_h(z,w)^2(z-w)^2(1 + \mO(|z-w|)|^2
  \notag \\
  &
  -
 \frac{\e^{-\frac{1}{h} K(z,w)}}{4 h^3\sinh \frac{K}{h}}
 \sigma_h(z,w)^3|z-w|^2(1 + \mO(|z-w|)
 \notag \\
  &+
 \frac{1}{4h^2}
 \left(
 \sigma_h(z,z)
 \sigma_h(w,w)
 +
 \sigma_h(z,w;h)^2(1 + \mO(|z-w|)
 \e^{-\frac{2}{h}K(z,w)}
 \right)
 \notag \\
 & + \mO(h^{\infty}).
\end{align*}
 Thus, one computes that 
 \begin{align*}
 \perm\Gamma =& 
 \frac{\sigma_h(z,w)^2(1 + \mO(|z-w|)}{4h^2\e^{\frac{1}{h} K(z,w)}\sinh \frac{K}{h}}
 \left[
  \left(\frac{\sigma_h(z,w)|z-w|^2}{4h}\right)^2 2\coth\frac{K}{h} 
  - \frac{\sigma_h(z,w)|z-w|^2}{h}
 \right]
 \notag \\
  &+
 \frac{1}{4h^2}
 \left(
 \sigma_h(z,z)
 \sigma_h(w,w)
 +
 \sigma_h(z,w;h)^2(1 + \mO(|z-w|)
 \e^{-\frac{2}{h}K(z,w)}
 \right)
 \notag \\
 & + \mO(h^{\infty})
\end{align*}
and we conclude the statement of the proposition. 
\end{proof}
%
\section{Proof of Proposition \ref{prop:2ptCorrelation}}
\label{sec:CorrelationFormula}
The first ingredient of the proof of Proposition \ref{prop:2ptCorrelation}, 
is the following global version of the implicit function theorem.
\begin{lem}
\label{lem:ImplFunThm}
 Let  $0<R_0<R$, let $n,m\in\mathds{N}$, with $n > m$, and let 
 $B(0,R)\subset\C^n=\C_z^{n-m}\times\C_w^{m}$ denote the complex open ball of radius $R>0$ 
 centered at $0$. For $z\in B_{\C^{n-m}}(0,R_0)$, define 
 $R(z) := (R^2 - \lVert z\rVert^2_{\C^{n-m}})^{1/2}$. We consider a 
 holomorphic function
 \begin{equation*}
  F: B(0,R) \longrightarrow \mathds{C}^m
 \end{equation*}
 such that   
	 \begin{itemize}
	 \item for all $(z,w)\in B(0,R)$ the Jacobian of $F$ with respect to $w$ is given 
	       by 
            \begin{equation*}
              \frac{\partial F(z,w)}{\partial w}= A + G(z,w), 
            \end{equation*}  
            where ${G:B(0,R) \longrightarrow \mathds{C}^{m\times m}}$ is a matrix-valued 
            holomorphic function and 
	  \item $A\in\mathrm{GL}_m({\mathds{C}})$ such that 
	     \begin{equation*}
	      \lVert A^{-1}\rVert\cdot\lVert G(z,w)\rVert\leq \theta < 1
	     \end{equation*}
             for all $(z,w)\in B(0,R)$.
	 \end{itemize}
  Then, for all $z\in B_{\C^{n-m}}(0,R_0)$ and for all 
  $y\in B_{\C^{m}}(F(z,0),\frac{1-\theta}{\lVert A^{-1}\rVert}r)$, with $0< r < R(z)$, 
  the equation 
  \begin{equation}
   \label{eqn:eqn2Inv}
   F(z,w) = y
  \end{equation}
  has exactly one solution $w(z,y)\in B_{\C^{m}}(0,R(z))$, it satisfies 
  $w(z,y) \in B_{\C^{m}}(0,r)$ and it depends holomorphically on $z$ and 
  on $y$.
\end{lem}
\begin{rem}
 Observe that the choice of $R_0 < R$ yields a uniform lower bound on $R(z)$ 
 and so we can choose the radius of the ball 
 $B_{\C^{m}}(F(z,0),\frac{1-\theta}{\lVert A^{-1}\rVert}r)$ uniformly in $z$. 
 This will become important in the proof of Proposition \ref{prop:2ptCorrelation}.
\end{rem}
\begin{proof} 
 Let $z\in B_{\C^{n-m}}(0,R_0)$ and set
 \begin{equation*}
  B_{\C^{m}}(0,R(z)) \ni w \longmapsto \widetilde{F}(w):=F(z,w).
 \end{equation*}
 We begin by observing that $ d\widetilde{F}(w)$ is invertible for all 
 $w\in B_{\C^m}(0,R(z))$ and the norm of the inverse is bounded 
 (uniformly in $z$). Indeed, for one has that 
 \begin{equation*}
  \left\lVert\left(d\widetilde{F}(w)\right)^{-1}
  \right\rVert 
   \leq \lVert A^{-1}\rVert \cdot \lVert(1 + A^{-1} G(z,w))^{-1} \rVert 
  \leq \frac{\lVert A^{-1}\rVert}{1-\theta}.
 \end{equation*}
 \\
 \textit{Claim $\# 1$}: $\widetilde{F}$ is injective. 
 \par 
 Let $w_0,w_1\in B_{\C^{m}}(0,R(z))$ and define $y_i:=\widetilde{F}(w_i)$. 
 Hence, with $w_t:=(1-t)w_0+tw_1$, we have that
 \begin{equation*}
  \frac{d}{dt}\widetilde{F}(w_t) 
  = d\widetilde{F}(w_t)\cdot(w_1-w_0)
  =(A + G(z,w_t))\cdot(w_1-w_0).
 \end{equation*}
 Thus, 
 \begin{equation*}
 y_1-y_0 = (A + H(z,w_1,w_0))\cdot(w_1-w_0),
 \quad 
 H(z,w_1,w_0))=
 \int_0^1G(z,w_t) dt,
 \end{equation*}
 where $\lVert H(z,w_1,w_0)\rVert \leq \sup_{B(0,R)} \lVert G(z,w)\rVert $. 
 Therefore, $\lVert A^{-1}\rVert\cdot\lVert H(z,w_1,w_0)\rVert 
 \leq \theta < 1$, and we see that $(A+H(z,w_1,w_0))$ is invertible and the 
 norm of its inverse is $\leq \frac{\lVert A^{-1}\rVert}{1-\theta}$ (uniformly 
 in $z$). Hence, 
 \begin{equation}
 \label{eqn:FInj}
  \lVert w_1 - w_0\rVert \leq 
  \frac{\lVert A^{-1}\rVert}{1-\theta} \lVert y_1 - y_0 \rVert,
 \end{equation}
 and we conclude that $\widetilde{F}$ is injective. In particular, 
 we have proven the uniqueness of the solution to the equation 
 \eqref{eqn:eqn2Inv}. 
 \\
 \par
 \textit{Claim $\# 2$}: Let $0< r < R(z)$. Then, for all 
 $y\in B_{\C^{m}}(\widetilde{F}(0),\frac{1-\theta}{\lVert A^{-1}\rVert}r)$ there exists a 
 $w\in B_{\C^{m}}(0,r)$ such that 
 \begin{equation*}
  \widetilde{F}(w) = y .
 \end{equation*}
  \par
 For $y = \widetilde{F}(0)$, we take $w=0$. Using the fact that $d\widetilde{F}$ is invertible 
 everywhere, the implicit function theorem implies that for all $y\in B(\widetilde{F}(0),\rho)$ 
 there exists a solution $w\in B_{\C^{m}}(0,r)$, if $\rho>0$ is small enough (cf. \eqref{eqn:FInj}). 
 Let $y\in B_{\C^{m}}(\widetilde{F}(0),\frac{1-\theta}{\lVert A^{-1}\rVert}r)$, and define 
 $y_t:=(1-t)\widetilde{F}(0) + ty$. Let $t_0\in [0,1]$ be the supremum of 
 $\widetilde{t}\in [0,1]$ such that there exists a solution to $\widetilde{F}(w_t)=y_t$ 
 for all $0 \leq t \leq \widetilde{t}$. 
 \par
 We have already proven that $t_0 >0$. As $t\nearrow t_0$ we have that 
 $w_t\in B_{\C^{m}}(0,r)$. Since $B_{\C^{m}}(0,r)$ is relatively compact in 
 $B_{\C^{m}}(0,R(z))$, there exists a sequence $t_j \nearrow t_0$ such 
 that $w_{t_j}\rightarrow \widetilde{w}$ with 
 $\widetilde{w}\in \overline{B_{\C^{m}}(0,r)}$. Thus, 
 \begin{equation*}
  \widetilde{F}(\widetilde{w}) = y_{t_0},
 \end{equation*}
 and we see by \eqref{eqn:FInj} that $\widetilde{w}\in B_{\C^{m}}(0,r)$. 
 \par
 If $t_0<1$, we get by the implicit function theorem, that for all 
 $y \in B(y_{t_0},\delta)$, with $\delta>0$ small enough, there exists 
 a solution $w\in B_{\C^{m}}(0,r)$. Therefore, we can solve 
 $\widetilde{F}(w_t) = y_t$ for all $0< t < t_0 + \delta$, which is a 
 contradiction. Hence, $t_0=1$, which concludes the proof of the existence 
 of a solution. 
 \\
 \par 
 Finally, note that for all $(z,w)\in B(0,R)$ the Jacobian 
 $\partial F(z,w)/ \partial w $ is invertible and the norm of its 
 inverse is uniformly bounded, indeed 
\begin{equation*}
  \left\lVert\left(\frac{\partial F(z,w)}{\partial w}\right)^{-1}
  \right\rVert 
   \leq \lVert A^{-1}\rVert \cdot \lVert(1 + A^{-1} G(z,w))^{-1} \rVert 
  \leq \frac{\lVert A^{-1}\rVert}{1-\theta}.
 \end{equation*}
 In particular, we have that the determinant of the Jacobian is 
 never equal to $0$, and we conclude by the holomorphic implicit 
 function theorem that the solution $w(z,y)$ to the equation 
 \eqref{eqn:FInj} depends holomorphically on $z$ and $y$. 
\end{proof}
\begin{proof}[of Proposition \ref{prop:2ptCorrelation}] 
 In view of \eqref{eqn:kthMoment}, it remains to study the integral 
 \begin{align}
   \label{eqn:Ist}
    I(z_1,z_2,h)= \lim\limits_{\varepsilon\rightarrow 0^+}
    \pi^{-N}\int_{B(0,R)}
    H^{\delta}_{\varepsilon}(z_1,z_2,\alpha;h)
    \e^{-\alpha\overline{\alpha}}L(d\alpha).
 \end{align}
 with
 \begin{equation*}
  H^{\delta}_{\varepsilon}(z_1,z_2,\alpha;h):=
  \prod_{k=1}^2\varepsilon^{-2}\chi
    \left(\frac{E_{-+}^{\delta}(z_k,\alpha)}{\varepsilon}\right)
    |\partial_{z_k}E_{-+}^{\delta}(z_k,\alpha)|^2 
 \end{equation*}
 for $1/C \geq |z_1-z_2| \gg h^{3/5}$. We begin by performing a 
 change of variables in the $\alpha$-space. 
 \\
 \\
 \textbf{Change of variables:}  
  For $X(z)\in\C^N$ as in Definition \ref{def:X}, define 
  the matrix
  \begin{equation*}
   {^tV}:=\left(X(z_1),X(z_2),\partial_{z_1}X(z_1),
   \partial_{z_2}X(z_2)\right) \in \C^{N\times 4}
  \end{equation*}
 and note that the Gramian matrix $G$ (cf. \eqref{eqn:G_44}) 
 satisfies
 \begin{equation*}
  G = 
  \begin{pmatrix}
   A & B \\
   B^* & C \\
  \end{pmatrix}
  =
  V\cdot V^*.
 \end{equation*}
 Moreover, $G$ is invertible by virtue of 
 Proposition \ref{prop:GramInvert}, since $|z_1-z_2| \gg h^{3/5}$. 
 Next, we define the matrix $U\in\C^{4\times 4}$ by 
 \begin{equation*}
   U:=\begin{pmatrix}
          1& 
          0 \\ 
          B^*A^{-1}& 
          1\\ 
          \end{pmatrix}.
 \end{equation*}
 $U$ is invertible and thus satisfies that $(U^{-1})^* = (U^*)^{-1}$. 
 Define the matrix  
 \begin{equation*}
  \widetilde{G} := 
  \begin{pmatrix}
   A &0  \\
   0 & \Gamma \\
  \end{pmatrix}\in\C^{4\times 4}, 
 \end{equation*}
 and notice that 
 \begin{align*}
   U
   \begin{pmatrix}
          A& 
          0 \\ 
          0 & 
          \Gamma \\ 
   \end{pmatrix}
    U^*
    =
    \begin{pmatrix}
          1& 
          0 \\ 
          B^*A^{-1}& 
          1\\ 
    \end{pmatrix}
   \widetilde{G}
   \begin{pmatrix}
          1& 
          A^{-1}B \\ 
          0& 
          1\\ 
    \end{pmatrix}
    =G.
  \end{align*}
 We see that $\widetilde{G} = U^{-1}G(U^*)^{-1}$. Next, we define 
 the matrix 
 \begin{equation}
  \label{eqn:Vtilde}
  \widetilde{V}^* := (U^{-1}V)^* \widetilde{G}^{-\frac{1}{2}} 
  \in \C^{N\times 4}.
 \end{equation}
 $\widetilde{V}^*$ is an isometry since $\widetilde{V} \widetilde{V}^* = 1_{\C^4}$. 
 Thus, its columns form an orthonormal 
 family in $\C^N$. It follows from \eqref{eqn:Vtilde} that the kernel of $V$ 
 and of $\widetilde{V}$ are equal, i.e. 
 $\mathcal{N}(V)=\mathcal{N}(\widetilde{V})$. The same holds true 
 for the range of $\widetilde{V}$ and of $V$, i.e. 
 $\mathcal{R}(V)=\mathcal{R}(\widetilde{V})$.
 \par
 Next, we choose an orthonormal basis, $e_1,\dots, e_N \in \mathds{C}^N$, of 
 the space of random variables $\alpha$ such that $\widetilde{V}_1^*,\dots,
 \widetilde{V}_4^*$, the column vectors of the matrix $\widetilde{V}^*$, are 
 among them. In particular, let $e_i=\widetilde{V}_i^*$ for $i=1,\dots,4$, 
 and let $e_5,\dots,e_N$ be in the orthogonal complement of the space 
 spanned by $e_1,\dots,e_4$. Hence, we write for $\alpha\in\C^N$
 \begin{equation*}
  \alpha = \sum_{i=1}^{N} \widetilde{\alpha}_ie_i ,
 \end{equation*}
 where
 $\widetilde{\alpha}=(\widetilde{\alpha}_1,\dots,\widetilde{\alpha}_N) \in\C^N$. 
 Moreover, note that 
 \begin{equation}
  \label{eqn:alpahIso}
  \alpha^* \cdot \alpha = \widetilde{\alpha}^* \cdot\widetilde{\alpha}.
 \end{equation}
 \begin{rem}
 Recall from Proposition \ref{prop:GramInvert} that we can only guarantee the 
 invertibility of $G$ for $ h^{\frac{3}{5}} \ll |z-w| \ll 1$. This makes the assumption 
 in Proposition \ref{prop:2ptCorrelation} (and Theorem \ref{thm_H2}) that the support 
 of the test function $\varphi$ avoids $D(\Omega,c)$, see \eqref{eq_i27}, necessary. 
 This might be avoided by choosing another set of basis vectors.
 \end{rem}
 Next, we apply this change of variables to the vector $F$ given in 
 \eqref{eqn:F} and we get
 \begin{align*}
    &F(z,\alpha(\widetilde{\alpha});\delta,h) 
    \notag \\
    &=
    \begin{pmatrix}
            E_{-+}(z_1) \\ 
            E_{-+}(z_2) \\ 
            (\partial_z E_{-+})(z_1) \\
            (\partial_z E_{-+})(z_2) \\
   \end{pmatrix}
   -
   \delta
    \begin{pmatrix}
        {^tX}(z_1) \\ 
        {^tX}(z_2)  \\ 
        {^t(\partial_z X)}(z_1) \\
        {^t(\partial_z X)(z_2)} \\
   \end{pmatrix}
   \cdot
   \alpha(\widetilde{\alpha})
   +
        \begin{pmatrix}
            T(z_1,\alpha(\widetilde{\alpha})) \\ 
            T(z_2,\alpha(\widetilde{\alpha})) \\ 
            (\partial_z T)(z_1,\alpha(\widetilde{\alpha})) \\ 
            (\partial_z T)(z_2,\alpha(\widetilde{\alpha})) \\ 
     \end{pmatrix}
     \notag \\
     &=
    \begin{pmatrix}
            E_{-+}(z_1) \\ 
            E_{-+}(z_2) \\ 
            (\partial_z E_{-+})(z_1) \\
            (\partial_z E_{-+})(z_2) \\
   \end{pmatrix}
   -
   \delta
    (V\cdot \widetilde{V}) \cdot
    \begin{pmatrix}
            \widetilde{\alpha}_1 \\ 
            \vdots \\ 
            \widetilde{\alpha}_{4} \\
   \end{pmatrix}
   +
        \begin{pmatrix}
            T(z_1,\alpha(\widetilde{\alpha})) \\ 
            T(z_2,\alpha(\widetilde{\alpha})) \\ 
            (\partial_z T)(z_1,\alpha(\widetilde{\alpha})) \\ 
            (\partial_z T)(z_2,\alpha(\widetilde{\alpha})) \\ 
     \end{pmatrix}.
  \end{align*}
 Furthermore, one computes that 
  \begin{equation}
  \label{eqn:UsqrtG}
   V\widetilde{V} = 
   U\widetilde{G}^{\frac{1}{2}} = 
   \begin{pmatrix}
          A^{\frac{1}{2}}& 
          0 \\ 
          B^*A^{-\frac{1}{2}}& 
          \Gamma^{\frac{1}{2}}\\ 
          \end{pmatrix},
  \end{equation}
 and we get that 
\begin{align*}
    F(z,\alpha(\widetilde{\alpha});\delta,h) 
    =
    \begin{pmatrix}
            E_{-+}(z_1) \\ 
            E_{-+}(z_2) \\ 
            (\partial_z E_{-+})(z_1) \\
            (\partial_z E_{-+})(z_2) \\
   \end{pmatrix}
   -
   \delta
   U\widetilde{G}^{\frac{1}{2}}\cdot
   \begin{pmatrix}
            \widetilde{\alpha}_1 \\ 
            \vdots \\ 
            \widetilde{\alpha}_{4} \\
   \end{pmatrix}
   +
        \begin{pmatrix}
            T(z_1,\alpha(\widetilde{\alpha})) \\ 
            T(z_2,\alpha(\widetilde{\alpha})) \\ 
            (\partial_z T)(z_1,\alpha(\widetilde{\alpha})) \\ 
            (\partial_z T)(z_2,\alpha(\widetilde{\alpha})) \\ 
     \end{pmatrix}.
  \end{align*}
 Next, to simplify our notation, we call the $\widetilde{\alpha}$ variables 
 again $\alpha$. Also, to abbreviate our notation, define 
 \begin{equation*}
   \mu(z,w;h):= 
      \begin{pmatrix}
            E_{-+}(z_1) \\ 
            E_{-+}(z_2) \\ 
           \end{pmatrix}
    ~ \text{and} ~
    \tau(z,\alpha;h,\delta) := 
        \begin{pmatrix}
            T(z_1,\alpha) \\ 
            T(z_2,\alpha) \\ 
           \end{pmatrix}.
  \end{equation*}
  and 
 \begin{equation*}
   \partial_z\mu(z,w;h):= 
      \begin{pmatrix}
           (\partial_z E_{-+})(z_1) \\ 
           (\partial_z E_{-+})(z_2) \\ 
           \end{pmatrix}
    ~ \text{and} ~
    \partial_z\tau(z,\alpha;h,\delta) := 
        \begin{pmatrix}
            (\partial_zT)(z_1,\alpha) \\ 
            (\partial_zT)(z_2,\alpha) \\ 
           \end{pmatrix}.
  \end{equation*}
 \begin{rem}
  Recall that $T$ (cf. \eqref{eqn:PowerSeriesT_b}) depends on 
  $h$ and on $\delta$, though not explicit in the above notation.
  \par
  When we write $\partial_z \mu$ and $\partial_z \tau$ the 
  derivatives are to be understood component wise, each of which 
  only depends either on $z_1$ or $z_2$. 
 \end{rem}
 Hence, 
  \begin{align}
  \label{eqn:kthRandomVector}
    F^{\delta}(z,\alpha):=F(z,\alpha;\delta,h)
    &= 
     \begin{pmatrix}
           \mu(z,h,\delta)  \\ 
           \partial_{z}\mu(z,h,\delta) \\
    \end{pmatrix}
    -\delta U\widetilde{G}^{\frac{1}{2}} 
           \begin{pmatrix}
            \alpha_1 \\ 
            \vdots \\ 
            \alpha_{4} \\
           \end{pmatrix}
     +
     \begin{pmatrix}
            \tau(z,\alpha,h,\delta)  \\ 
           \partial_{z}\tau(z,\alpha,h,\delta)  \\
    \end{pmatrix}.  
  \end{align}
As noted in Remark \ref{rem:ET_estim}, $\mu$ and $\tau$ are smooth in $z$, 
and $\tau$ is holomorphic in $\alpha$. Moreover, $\tau$ satisfies the estimates
  \begin{equation}\label{eqn:TauEst}
   \tau_i = \mO\!\left(h^{-5/2}\delta^2\right), ~i =1,2 
   ~\text{and}~ 
   \partial_{z_i}\tau_i = \mO\!\left(h^{-7/2}\delta^2\right),
   ~i =1,2;
  \end{equation}
 and $\mu$ satisfies the estimates
 \begin{equation}\label{eqn:muEst}
   \mu_i= \mO\!\left(h^{1/2}\e^{-\frac{S}{h}}\right),
   ~~\partial_{z_i}\mu_i = \mO\!\left(h^{-1/2}\e^{-\frac{S}{h}}\right), ~i =1,2
  \end{equation}
 with $S$ as in \eqref{def_act}. 
 Finally, we perform the above described change of variables in 
 the integral \eqref{eqn:Ist}, and, using the fact that we chose 
 an orthonormal basis of the $\alpha$-space, we get that
 \begin{align*}
  H^{\delta}_{\varepsilon}(z_1,z_2,\alpha;h)= 
  \prod_{k=1}^2\varepsilon^{-2}\chi
    \left(\frac{F^{\delta}_k(z_k,\alpha)}{\varepsilon}\right)
    |F^{\delta}_{k+2}(z_k,\alpha)|^2.
 \end{align*}
 Next, let $\alpha=(\alpha_1,\alpha_2,\alpha') = (\widetilde{\alpha},\alpha')$ and 
 split the ball $B(0,R)$, $R=Ch^{-1}$, into two pieces: 
 pick $C_0>0$ such that $0 < C_1 < C_0 < C < 2C_0$, and define $R_0=C_0h^{-1}$. 
 Then, we perform the splitting: $I(z,h) = I_1(z,h) + I_2(z,h)$ with 
 \begin{align}\label{eqn:I1}
    I_1(z,h):= \lim\limits_{\varepsilon \rightarrow 0^+}
    \pi^{-N}\int\limits_{ \substack{B(0,R) \\ \lVert\alpha'
		    \rVert_{\C^{N-2}} \leq R_0} }
     H^{\delta}_{\varepsilon}(z_1,z_2,\alpha;h) \e^{-\alpha^*\alpha}L(d\alpha).
 \end{align}
 and 
 \begin{align}
 \label{eqn:I2}
    I_2(z,h):= \lim\limits_{\varepsilon \rightarrow 0^+}
    \pi^{-N}\int\limits_{ \substack{B(0,R) \\ R_0 <\lVert \alpha' 
		    \rVert_{\C^{N-2}} < R} }
    H^{\delta}_{\varepsilon}(z_1,z_2,\alpha;h) \e^{-\alpha^*\alpha}L(d\alpha).
 \end{align}
 \textbf{The integral $I_1$ } First, we perform a new change of variables 
 in the $\alpha$-space. Let $\beta_1, \dots ,\beta_N \in\C$ such that 
 \begin{equation*}
  \beta_1 =  F_1^{\delta}(z_1,\alpha),~ \beta_2 =  F_2^{\delta}(z_2,\alpha)
  ~\text{and}~ \beta_i = \alpha_i, ~\text{for } i =3,\dots,N.
 \end{equation*}
 We use the following notation: 
 $\beta=(\beta_1,\beta_2,\beta') = (\widetilde{\beta},\alpha') $. 
 It is sufficient to check that we can express 
 $\widetilde{\alpha}=(\alpha_1,\alpha_2)$ as 
 a function of $(\widetilde{\beta},\alpha') $. Therefore, we 
 apply Lemma \ref{lem:ImplFunThm} to the function 
 \begin{equation*}
  \mathcal{F}^{\delta}(z,\alpha) = 
  \begin{pmatrix}
   F_1^{\delta}(z_1,\alpha) \\ 
   F_2^{\delta}(z_2,\alpha)\\
  \end{pmatrix}.
 \end{equation*}
 where $\alpha$ plays the role of $(z,w)$ in the Lemma. In particular, 
 $\widetilde{\alpha}$ plays the role of 
 $w$. Let us check that the assumptions of Lemma \ref{lem:ImplFunThm} 
 are satisfied: $\mathcal{F}^{\delta}(z,\alpha)$ is by definition holomorphic 
 in $\alpha$. Using \eqref{eqn:kthRandomVector} and \eqref{eqn:UsqrtG}, we see that 
 its Jacobian, with respect to the variables $\widetilde{\alpha}$, is given by 
 \begin{equation}
  \label{eqn:der_beta}
  \frac{\partial \mathcal{F}(z,\alpha)}{\partial \widetilde{\alpha}} = 
      \frac{\partial \tau}{\partial\widetilde{\alpha}} 
		 - \delta A^{\frac{1}{2}}
 \end{equation}
 The Cauchy inequalities and (\ref{eqn:TauEst}) imply that 
 \begin{equation*}
   \frac{\partial \tau_i}{\partial\widetilde{\alpha}_j} 
      = \mO\left(\delta^2 h^{-\frac{3}{2}}\right), 
      \quad i,j =1,2.
 \end{equation*}
 This estimate is uniform in $\alpha\in B(0,R)$ and $(z_1,z_2)\in\supp\varphi$. 
 Expansion of the determinant yields that 
 \begin{equation}\label{eqn:Jacobian}
  \det\left( \frac{\partial \tau}{\partial\widetilde{\alpha}} 
		 - \delta A^{\frac{1}{2}}\right)
  = 
   \delta^2 \left(\sqrt{\det A} 
    + \mO\left(\delta h^{-\frac{3}{2}}\right)\right).
 \end{equation}
 Using that $A$ is self-adjoint, we see by Corollary \ref{cor:LB_MN_A} that 
 for $(z_1,z_2)\in\supp \varphi$
 \begin{equation}
  \label{eqn:lb_A}
  \lVert A^{-\frac{1}{2}}\rVert \leq 
  \frac{1}{\min\limits_{\lambda \in\sigma(A)}\sqrt{\lambda } }
  \leq \mO\!\left( h^{-\frac{1}{10}}\right).
 \end{equation}
 By the hypothesis \eqref{Hyp:Delta}, we have that $\delta \ll h^{7/2}$. Hence, 
 one gets that for all $\alpha\in B(0,R)$
 \begin{equation*}
  \delta^{-1} \lVert A^{-\frac{1}{2}} \rVert \cdot 
  \rVert \partial_{\widetilde{\alpha}}\tau \lVert 
  \leq \mO\!\left(\delta h^{-\frac{3}{2}-\frac{1}{10}}\right)
  \ll 1.
 \end{equation*}
 Hence $\mathcal{F}^{\delta}(z,\alpha)$ satisfies 
 the assumptions of Lemma \ref{lem:ImplFunThm}. 
 In the integral $I_1$ we restricted $\alpha'$ to the open ball 
 $\lVert\alpha'\rVert_{\C^{N-2}} < R_0$. It follows by Lemma 
 \ref{lem:ImplFunThm} that for all
 \begin{equation}
  \label{beta_sol_dom}
  \widetilde{\beta} \in B_{\C^2}\left(\mathcal{F}^{\delta}(z;0,\alpha'),r\right)
 \end{equation}
with 
\begin{align*}
 r:&= \left(
      \delta\lVert A^{-\frac{1}{2}}\rVert^{-1}
        (1 -\max\limits_{\alpha\in B(0,R)}\delta^{-1} 
        \lVert A^{-\frac{1}{2}} \rVert \cdot 
        \rVert \partial_{\widetilde{\alpha}}\tau \lVert  ) 
	\right)
	\sqrt{R^2-R_0^2}\notag \\
	&\geq
	 \frac{\delta h^{\frac{1}{10}-1}}{\mO(1)} > 0,
\end{align*}
the equation $\widetilde{\beta} =\mathcal{F}^{\delta}(z,\widetilde{\alpha},\alpha')$ 
has exactly one solution $\widetilde{\alpha}(\widetilde{\beta},\alpha';z)$ in the 
ball 
\begin{equation*}
 B\left(0,\sqrt{R^2-\lVert\alpha'\rVert_{\C^{N-2}}^2}\right)).
\end{equation*}
Moreover, the solution satisfies 
$\widetilde{\alpha}(\widetilde{\beta},\alpha';z)\in B(0,\sqrt{R^2-R_0^2})$, and it 
depends holomorphically on $\widetilde{\beta}$ and $\alpha'$ and is smooth in $z$. 
Using \eqref{eqn:kthRandomVector}, we see that the solution is implicitly given by 
\begin{align}
 \label{eqn:dens_11}
 \widetilde{\alpha}(\widetilde{\beta},\alpha') 
      &= 
      -\delta^{-1}A^{-\frac{1}{2}}\left(\widetilde{\beta} - 
      \nu(z,\widetilde{\alpha}(\widetilde{\beta},\alpha'),\alpha',h,\delta)\right).
\end{align}
with 
  \begin{equation*}
   \nu :=(\nu_1,\nu_2)^t := \mu(z,h)+
     \tau(z,\widetilde{\alpha}(\widetilde{\beta},\alpha'),\alpha',h,\delta)
  \end{equation*}
where $\tau$ satisfies the estimate (\ref{eqn:TauEst}).
Since the support of $\chi$ is compact (cf. Section \ref{sec:CountZero2}), we can restrict 
our attention to $\widetilde{\beta}$ and $\mathcal{F}^{\delta}(z;0,\alpha')$ 
in a small poly-disc of radius $K \varepsilon>0$ centered at $0$, with $K>0$ large 
enough such that $\supp\chi \subset D(0,K)$. By choosing 
$\varepsilon < \delta h/C$, $C>0$ large enough, we see that 
$\widetilde{\beta},\mathcal{F}^{\delta}(z;0,\alpha')\in D(0,K\varepsilon)\times D(0,K\varepsilon)$ 
implies \eqref{beta_sol_dom}.
\par
From \eqref{eqn:UsqrtG}, \eqref{eqn:kthRandomVector} and \eqref{eqn:dens_11}, 
it follows that  
   \begin{align}
    \label{eqn:dens_21}
   \begin{pmatrix}
    F^{\delta}_{3}(z,\widetilde{\alpha}(\widetilde{\beta},\alpha'),\alpha')  \\
    F^{\delta}_{4}(z,\widetilde{\alpha}(\widetilde{\beta},\alpha'),\alpha')  \\
   \end{pmatrix}
     = 
     \partial_z\nu  +  B^*A^{-1}(\widetilde{\beta} - \nu)
      - \delta\Gamma^{\frac{1}{2}}
      \begin{pmatrix}
       \alpha_3 \\
       \alpha_4  \\
   \end{pmatrix},
   \end{align}
  with 
  \begin{equation*}
   \partial_z\nu =(\partial_z\nu_1,\partial_z\nu_2)^t = 
   (\partial_z\mu)(z,h)+
     (\partial_z\tau)(z,\widetilde{\alpha}(\widetilde{\beta},\alpha'),\alpha',h,\delta)
  \end{equation*}
where $ \partial_z\tau$ satisfies the estimate given in (\ref{eqn:TauEst}). 
Furthermore, \eqref{eqn:der_beta} and \eqref{eqn:Jacobian} imply that 
\begin{equation}
 \label{eqn:Jacobian2}
 L(d{\widetilde{\alpha}}) =  
 \delta^{-4} \left(\sqrt{\det A} 
    + \mO\left(\delta h^{-\frac{3}{2}}\right)\right)^{-2}
    L(d{\widetilde{\beta}}) =: J(\widetilde{\beta},\alpha') L(d{\widetilde{\beta}}) 
\end{equation}
By performing this change of variables in the integral $I_1$ and by picking $\varepsilon>0$ 
small enough as above, we get that $I_1$ is equal to 
 \begin{align*}
   \lim\limits_{\varepsilon \searrow 0} 
    \pi^{-N} 
    \iint\limits_{
    \substack{
     \widetilde{\beta} \in D(0,K\varepsilon)\times D(0,K\varepsilon) \\
    (\tilde{\alpha}(\widetilde{\beta},\alpha'),\alpha')\in B(0,R) \\ 
	\lVert \alpha' \rVert_{\C^{N-2}} \leq R_0} }
    H^{\delta}_{\varepsilon}(z_1,z_2,\widetilde{\alpha}(\widetilde{\beta},\alpha'),\alpha';h)
    \e^{-\Phi(\widetilde{\beta},\alpha')}J(\widetilde{\beta},\alpha')
    L(d\alpha') L(d\widetilde{\beta}),
  \end{align*}
where 
\begin{equation*}
 \Phi(\widetilde{\beta},\alpha'):= \widetilde{\alpha}(\widetilde{\beta},\alpha')^*\cdot
 \widetilde{\alpha}(\widetilde{\beta},\alpha')+(\alpha')^*\cdot\alpha'.
\end{equation*}
The integrand of $I_1$ depends continuously on 
$\widetilde{\beta}$. Hence, by performing the limit $\varepsilon \rightarrow 0^+$, 
we get
 \begin{align}\label{eqn:Int1_1}
   I_1(z,h) =\pi^{-N}
   \int\limits_{ 
	 \substack{(\tilde{\alpha}(0,\alpha'),\alpha')\in B(0,R) \\ 
	\lVert \alpha' \rVert_{\C^{N-2}} \leq R_0} }
   H^{\delta}_{0}(z_1,z_2,\widetilde{\alpha}(0,\alpha'),\alpha';h)
    \e^{-\Phi(0,\alpha')}J(0,\alpha')
    L(d\alpha')
  \end{align}
with 
\begin{equation*}
 H^{\delta}_{0}(z_1,z_2,\widetilde{\alpha}(0,\alpha'),\alpha';h)
 = |F_{3}(z,0,\alpha') F_{4}(z,0,\alpha')|^2.
\end{equation*}
Using \eqref{eqn:dens_11}, one computes that
  \begin{equation*}
    \Phi(0,\alpha') = 
    \frac{1}{\delta^2}\nu^* A^{-1}\nu 
    +(\alpha')^*\cdot\alpha'
  \end{equation*}
and, using \eqref{eqn:dens_21}, we get
\begin{align}
 \label{eqn:rest_dens}
 \begin{pmatrix}
    F^{\delta}_{3}(z,\widetilde{\alpha}(0,\alpha'),\alpha')  \\
    F^{\delta}_{4}(z,\widetilde{\alpha}(0,\alpha'),\alpha')  \\
   \end{pmatrix}
     = 
     \partial_z\nu  - B^*A^{-1}\nu
      - \delta\Gamma^{\frac{1}{2}}
      \begin{pmatrix}
       \alpha_3 \\
       \alpha_4  \\
   \end{pmatrix},
   \end{align}
where $\nu = \nu(z,\widetilde{\alpha}(0,\alpha'),\alpha',h,\delta)$. 
Using \eqref{eqn:TauEst}, \eqref{eqn:muEst} and \eqref{eqn:lb_A} one 
computes that 
\begin{align}
 \label{eqn:Expo1}
 \lVert \widetilde{\alpha}(0,\alpha') \rVert^2 
 = \frac{1}{\delta^2}\nu^* A^{-1}\nu 
 \leq C h^{-\frac{1}{5}}
 \left[\mO\!\left(\delta^{-2} \e^{-\frac{2S}{h}}\right) 
	+ \mO\!\left(\delta^2h^{-5}\right)\right],
\end{align}
where the constant $C>0$ comes from the upper bound of 
$\lVert A^{-1/2} \rVert$ given in \eqref{eqn:lb_A}. By 
the Hypothesis \eqref{Hyp:Delta}, we conclude that 
\begin{align*}
 \lVert \widetilde{\alpha}(0,\alpha') \rVert^2 
 \ll h^{-\frac{1}{5}}.
\end{align*}
which implies that $(\tilde{\alpha}(0,0,\alpha'),\alpha')\in B(0,R)$ for 
all $\alpha'$ with $\lVert \alpha' \rVert_{\C^{N-2}} \leq R_0$. Hence, 
\begin{align}\label{eqn:Int1_2}
   I_1(z,h) =\pi^{-N}
   \int\limits_{\lVert \alpha' \rVert_{\C^{N-2}} \leq R_0}
    |F_{3}(z,0,\alpha') F_{4}(z,0,\alpha')|^2 
    \e^{-\Phi(0,\alpha')}J(0,\alpha')
    L(d\alpha').
  \end{align}
Next, we want to apply a multi-dimensional version of the mean value 
theorem for integrals to (\ref{eqn:Int1_2}). Indeed, let $U\subset\R^n$ 
be open, relatively compact and path-\-connected, it then holds true that 
for a continuous function $f:\overline{U}\rightarrow \R$ and a 
positive integrable function $g:\overline{U}\rightarrow \R$, 
there exists a $y\in\overline{U}$ such that 
\begin{equation*}
 f(y)\int_{U} g(x) dx =\int_{U}  f(x)g(x) dx.
\end{equation*}
Hence, the mean value theorem applied to (\ref{eqn:Int1_2}) yields that 
 \begin{equation*}
   I_1(z,h) = 
    \pi^{-N}J\e^{-\frac{\widetilde{\nu}^* A^{-1}\widetilde{\nu}}{\delta^2}}
      \int\limits_{ \lVert \alpha' \rVert_{\C^{N-2}} \leq R_0 }
     |F_{3}(z,0,\alpha') F_{4}(z,0,\alpha')|^2 
    \e^{-\alpha'\overline{\alpha'}}L(d\alpha').
  \end{equation*}
Here, $J$ denotes the evaluation of the Jacobian 
$J(0,\alpha')$ (cf. \eqref{eqn:Jacobian2}) at the intermediate 
point for $\alpha'$ given by mean value theorem. 
Note that $J$ depends smoothly on $z_1$ and $z_2$ because 
$\tau$ and $A$ do. 
\par
Similarly, $\widetilde{\nu}$ above denotes the evaluation of 
the function $\nu(z,\widetilde{\alpha}(0,\alpha'),\alpha',h,\delta)$ 
at the intermediate point for $\alpha'$ given by mean value theorem. It depends 
smoothly on $z_1$ and $z_2$ because $\mu$ and $\tau$ do. Moreover, using 
(\ref{eqn:TauEst}), we see that it satisfies 
  \begin{equation*}
   \widetilde{\nu} = \begin{pmatrix}
            E_{-+}(z_1) \\ 
            E_{-+}(z_2) \\
           \end{pmatrix}
           +\mO\left(\delta^2h^{-\frac{5}{2}}\right).
  \end{equation*}
In remains to study the integral	
\begin{equation}\label{eq_ad7}
 \widetilde{I}_1(z,h):= \pi^{-N}\int\limits_{ \lVert \alpha' \rVert_{\C^{N-2}} \leq R_0 }
     |F_{3}(z,0,\alpha') F_{4}(z,0,\alpha')|^2
    \e^{-\alpha'\overline{\alpha'}}L(d\alpha').
\end{equation}
Define the linear forms
\begin{equation*}
 l_1(\alpha') = [\Gamma^{\frac{1}{2}}]_{11}\alpha_3 +
	       [\Gamma^{\frac{1}{2}}]_{12}\alpha_4
	       , \quad 
 l_2(\alpha') = [\Gamma^{\frac{1}{2}}]_{21}\alpha_3 +
	       [\Gamma^{\frac{1}{2}}]_{22}\alpha_4.	       
\end{equation*}
Using \eqref{eqn:rest_dens}, we get that 
\begin{equation}\label{eq_ad4}
\begin{split}
 &F_{3}(z,0,\alpha') = 
 ( \partial_z\nu  - B^*A^{-1}\nu)_1 
  - \delta l_1(\alpha')
  = 
   \mO\!\left(
    h^{-\frac{3}{5}}\e^{-\frac{S}{h}} + \delta^2h^{-\frac{37}{10}}\right) 
    - \delta l_1(\alpha'), \\
  &F_{4}(z,0,\alpha') = 
 ( \partial_z\nu  - B^*A^{-1}\nu)_2 
  - \delta l_2(\alpha')
  = 
   \mO\!\left(
    h^{-\frac{3}{5}}\e^{-\frac{S}{h}} + \delta^2 h^{-\frac{37}{10}}\right) 
    - \delta l_2(\alpha'),
    \end{split}
\end{equation}
where the error estimate is uniform in $\alpha'$, for $\lVert \alpha' \rVert_{\C^{N-2}} \leq R_0$. 
In the last equation we used \eqref{eqn:TauEst}, \eqref{eqn:muEst}, \eqref{eqn:lb_A} 
and the fact that the Hilbert-Schmidt norm of $B^*$ is $\leq \frac{1}{h\mO(1)}$ which 
follows from the fact that elements of the matrix $B^*$ are bounded by a term of order 
$h^{-1}$. 
\par
By Proposition \ref{prop:Trace_est_Gam}, one gets that the Hilbert-Schmidt norm of 
$\Gamma^{\frac{1}{2}}$ is bounded, indeed one has that 
\begin{equation*}
 \lVert \Gamma^{\frac{1}{2}} \rVert_{\mathrm{HS} } 
 = \sqrt{\tr\Gamma} \leq \mO(h^{-\frac{1}{2}}).
\end{equation*}
Hence, the linear forms $l_i(\alpha')$, $i=1,2$, satisfy 
\begin{equation*}
 |l_i(\alpha')| \leq \mO(h^{-\frac{1}{2}}) \|(\alpha_3,\alpha_4)\|.
\end{equation*}
Using \eqref{eq_ad4}, we compute that
\begin{equation}\label{eq_ad5}
  |F_{3}(z,0,\alpha') F_{4}(z,0,\alpha')|^2  
  =\delta^{4}\left(|l_1(\alpha')l_2(\alpha')|^2
  + \mO\!\left(\e^{-\frac{1}{Ch}}+ 
    \delta h^{-\frac{52}{10}}\right)\sum_{j=0}^3\|(\alpha_3,\alpha_4)\|^j
    \right).
\end{equation}
Here we used as 
well that by Hypothesis \ref{Hyp:H7},  
we have that $\mO(\delta^{-1}\e^{-\frac{S}{h}}) = \mO(\e^{-\frac{1}{Ch}})$. 
Observe that since $C>C_0>C_1>0$, see the discussion before \eqref{eqn:I1}, 
we have that for $k=0,\dots,4$
\begin{equation}\label{eq_ad6}
   \pi^{-N}
      \int\limits_{ \lVert \alpha' \rVert_{\C^{N-2}} \geq R_0} 
    \|(\alpha_3,\alpha_4)\|^k
    \e^{-\alpha'\overline{\alpha'}}L(d\alpha')
    \leq  \mO\!\left(\e^{-\frac{D}{h^2}}\right).
  \end{equation}
Technically this holds true if the difference $C_0-C_1>0$ is assumed to be 
sufficiently large. Notice that we have room for that if we take $C>0$ 
in \eqref{eq_ad5.1} large enough to begin with and choose $C_0$ in 
the discussion before \eqref{eqn:I1} sufficiently large.
\par
Extend the function $|F_{3}(z,0,\alpha') F_{4}(z,0,\alpha')|^2$ in the variables $\alpha'$ 
to the whole of $\C^{N-2}$ by a function such that \eqref{eq_ad5} holds 
for all $\alpha'\in\C^{N-2}$. Hence, by \eqref{eq_ad6}, \eqref{eq_ad7} 
\begin{equation*}
 \widetilde{I}_1(z,h)= \delta^4 
    \pi^{-N} \int_{ \C^{N-2} }
    |l_1(\alpha')l_2(\alpha')|^2 
    \e^{-\alpha'\overline{\alpha'}}L(d\alpha') 
    + \mO\!\left(\delta^4 \e^{-\frac{1}{Ch}}+ \delta^5 h^{-\frac{52}{10}}\right) .
\end{equation*}
Integration by parts yields that 
\begin{equation*}
   \pi^{-N}
      \int_{\C^{N-2}} 
    |l_1(\alpha')l_2(\alpha')|^2 
    \e^{-\alpha'\overline{\alpha'}}L(d\alpha')
    = 
    \pi^{-4}\int_{\C^{2}}\e^{-\widetilde{\alpha}\overline{\widetilde{\alpha}}}
    \prod_{k=1}^2
    l_k(\overline{\partial_{\widetilde{\alpha}}})
    \left(\prod_{n=1}^2
    \overline{l_n}(\overline{\widetilde{\alpha}})\right)
    L(d\widetilde{\alpha}).
  \end{equation*}
Note that for any permutation $\sigma\in S_n$, where 
$S_n$ is the symmetric group, we have that 
$(l_i|l_{\sigma(i)}) = \Gamma_{i\sigma(i)}$. Thus, in 
view of \eqref{def:Per}, we have that 
\begin{equation*}
\prod_{k=1}^2
    l_k(\overline{\partial_{\widetilde{\alpha}}})
     \left(\prod_{n=1}^2
    \overline{l_n}(\overline{\widetilde{\alpha}})\right)
    =
    \sum_{\sigma\in S_2} (l_1|l_{\sigma(1)}) (l_2|l_{\sigma(2)})
    = 
    \perm \Gamma.
\end{equation*}
We conclude that 
\begin{align*}
   I_1(z,h) &=
    \frac{\perm\Gamma\
  + \mO\!\left(\e^{-\frac{1}{Ch}}+ 
    \delta h^{-\frac{52}{10}}\right)}{\pi^{2}\left(\sqrt{\det A} 
    + \mO\left(\delta h^{-\frac{3}{2}}\right)\right)^2},
  \end{align*}
 where we used the fact that $\det A \geq \frac{h^{\frac{1}{5}}}{\mO(1)}$ 
 for $1/C \geq |z-w| \gg h^{3/5}$, see Proposition \ref{prop:mat_Gamma}, 
 to obtain the last equality.
 \\
 \\
\textbf{The integral $I_2$ }  
In this step we will estimate the second integral of equation (\ref{eqn:I2}). 
Therefore, we will increase the space of integration 
  \begin{align*}
  &\pi^{-N}\int\limits_{ \substack{B(0,R) \\ R_0 <\lVert \alpha' 
		    \rVert_{\C^{N-2}} < R} }
    \prod_{k=1}^2\varepsilon^{-2}\chi
    \left(\frac{F_k(z,\alpha)}{\varepsilon}\right)
    |\partial_{z_k}F_{k}(z,\alpha)|^2 \e^{-\alpha\overline{\alpha}}L(d\alpha)
    \notag \\
    &\leq
    \pi^{-N}\int\limits_{ \substack{B(0,2R) \\ R_0 <\lVert \alpha' 
		    \rVert_{\C^{N-2}} < 2R_0} }
    \prod_{k=1}^2\varepsilon^{-2}\chi
    \left(\frac{F_k(z,\alpha)}{\varepsilon}\right)
    |\partial_{z_k}F_{k}(z,\alpha)|^2 \e^{-\alpha\overline{\alpha}}L(d\alpha)
    =: W_{\varepsilon}.
  \end{align*}
It is easy to see that Lemma \ref{lem:ImplFunThm} holds true for the set 
$B(0,2R) \cap \{ R_0 < \lVert\alpha'\rVert_{\mathds{C}^{N-2}} < 2R_0\}$. 
Therefore, we can proceed as for the integral $I_1$: 
perform the same change of variables and perform the limit of 
$\varepsilon\rightarrow 0$. As for $I_1$, the integrand remains 
bounded by at most a finite power of $h^{-1}$ which then yields that 
  \begin{align*}
   \lim\limits_{\varepsilon\rightarrow 0} \,W_{\varepsilon}   
       = \mathcal{O}\!\left(\e^{-\frac{D}{h^2}}\right),
\end{align*}
where the exponential decay comes from the fact that 
$R_0 <\lVert \alpha'\rVert_{\C^{N-2}}$. Therefore, 
\begin{equation*}
 \int_{\C^2}\varphi_1(z_1)\varphi_2(z_2) d\nu(z_1,z_2)
 = \int_{\C^2}\varphi_1(z_1)\varphi_2(z_2) D(z,h)L(dz_1dz_2) 
\end{equation*}
with 
\begin{equation*}
 D(z,h,\delta) 
 = \frac{\perm\Gamma\
  + \mO\!\left(\e^{-\frac{1}{Ch}}+ 
    \delta h^{-\frac{52}{10}}\right)}{\pi^{2}\left(\sqrt{\det A} 
    + \mO\left(\delta h^{-\frac{3}{2}}\right)\right)^2}
    + \mO\!\left(\e^{-\frac{D}{h^2}}\right). \qedhere
\end{equation*}
\end{proof}
\section{Proof of the main results}\label{sec:PMR} 
Using the above results, in particular Propositions \ref{prop:2ptCorrelation} and 
\ref{prop:perm_Gamma}, we can now prove Theorem \ref{thm_H2}, 
Theorem \ref{prop:H10}, Theorem \ref{prop:H11} and Corollary \ref{cor1}.
\begin{proof}[of Theorem \ref{thm_H2}]
 The result follows directly from Proposition \ref{prop:2ptCorrelation} 
 with the density $D$ given by Proposition \ref{prop:perm_Gamma} and 
 by Proposition \ref{prop:detA}.
\end{proof}
\begin{proof}[of Theorem \ref{prop:H10}]
 First, let us treat the case of the long range interaction: 
 we suppose that $|z-w| \gg (h\ln h^{-1})^{\frac{1}{2}}$. Here, 
 we have that for any power $N>1$ the term 
 \begin{equation*}
  \left(\frac{\sigma_h(z,w)|z-w|^2}{4h}\right)^N \e^{-K(z,w)}
 \end{equation*}
 remains bounded. Using that $\sinh K(z,w)\geq \mO(h^{-C}) >0 $ 
 with $C\gg 1$ and using that $\sigma_h(z,z) = \sigma(z) + \mO(h)$, it follows that 
 \begin{equation*}
  D^{\delta}(z,w;h) = \frac{\sigma(z)\sigma(w) + \mO(h)}{(2h\pi)^{2}}
    \left(1 + \mO\!\left( \delta h^{-\frac{8}{5}}\right)\right).
 \end{equation*}
 Next, we consider the case where $h^{\frac{4}{7}}\ll |z-w| \ll h^{\frac{1}{2}}$. 
 Recall from  Theorem \ref{thm_H2} that 
  \begin{align}\label{de.1}
   D^{\delta}(z,w;h) = \frac{\Lambda(z,w)}{(2\pi h)^{2}\left(1- \e^{-2K(z,w)}\right)}
    \left(1 
      + \mO\!\left(\delta h^{-\frac{8}{5}}\right)\right)
    + \mO\!\left(\e^{-\frac{D}{h^2}}\right)
  \end{align}
  with $\Lambda(z,w;h)$ equal to 
   \begin{align*}
   &\sigma_h(z,z)\sigma_h(w,w) + \sigma_h(z,w)^2(1 + \mO(|z-w|))
	      \e^{-2K(z,w)}+\mO\!\left(h^{\infty} + \delta h^{-\frac{32}{10}}\right)
	  \notag \\
	  &+ 
	  \frac{\sigma_h(z,w)^2(1 + \mO(|z-w|)) }{\e^{K(z,w)}\sinh K(z,w)}
	  \left(
	  \left(\frac{\sigma_h(z,w)|z-w|^2}{4h}\right)^2 2\coth K(z,w)
	  - \frac{\sigma_h(z,w)|z-w|^2}{h} 
	  \right).
 \end{align*}
 Similarly to \eqref{eqn:Tay3}, we have that $\sigma_h(z,z) =  \sigma_h(z,w)(1 + \mO(|z-w|)$. 
 We start by considering the first term in \eqref{de.1}:
 \begin{equation}\label{de.2}
  \frac{\Lambda(z,w)}{(2\pi h)^{2}\left(1- \e^{-2K(z,w)}\right)}.
 \end{equation}
 Set $\sigma_h = \sigma_h(z,w)$. Using the Taylor expansions of the functions $\sinh x$, $\coth x$ 
 and $\e^{-x}$, one computes, that \eqref{de.2} is equal to
 \begin{align*}
  &
  \frac{1}{h\pi^2\sigma_h|z-w|^2\left(1+ \mO\!\left(\frac{|z-w|^2}{h}\right)\right)}
  \Bigg[ \sigma_h^2\left(1+ \mO\!\left(|z-w|\right)\right) 
  -\frac{ \sigma_h^3|z-w|^2}{4h}\left(1+ \mO\!\left(|z-w|\right)\right)
  \notag \\
  &+
  \frac{ \sigma_h^4|z-w|^4}{4^2h^2}\left(1+ \mO\!\left(\frac{|z-w|^2}{h}\right)\right)
  + \left\{
  \frac{\sigma_h^4|z-w|^4}{3\cdot4^4h^2}\left(1+ \mO\!\left(\frac{|z-w|^4}{h^2}\right)\right)
  -1\right\}\cdot
  \notag \\
  &
  \cdot
  \left.
  \frac{\sigma_h^2\left(1 -\frac{\sigma_h|z-w|^2}{4h}\left(1+ \mO\!\left(|z-w|\right)\right)
  +\frac{\sigma_h^2|z-w|^4}{2\cdot4^2h}\left(1+ \mO\!\left(\frac{|z-w|^2}{h}\right)\right)
  \right)}
  {1+ \mO\!\left(|z-w|\right)+
  \frac{\sigma_h^2|z-w|^4}{4^2\cdot6h}\left(1+ \mO\!\left(\frac{|z-w|^2}{h}\right)\right)}
  +\mO\!\left(h^{\infty} + \delta h^{-\frac{32}{10}}\right)
  \right]
 \end{align*}
 which simplifies to 
 \begin{align*}
 \Lambda(z,w;h)=
  \frac{\sigma_h^3|z-w|^2}{(4\pi)^2h^3}\left(1+ \mO\!\left(\frac{|z-w|^2}{h}\right)\right).
 \end{align*}
 Hence, 
 \begin{equation*}
  D^{\delta}(z,w;h) 
  =  \frac{\sigma_h^3|z-w|^2}{(4\pi)^2h^3}\left(1+ 
  \mO\!\left(\frac{|z-w|^2}{h}+\delta h^{-\frac{8}{5}}\right)\right)
 \end{equation*}
 which concludes the proof.
\end{proof}.
\begin{proof}[of Theorem \ref{prop:H11}]
 Using that $\sigma_h(z,w_0) =  \sigma_h(z,z)(1 + \mO(|z-w_0|)$ (cf. \eqref{eqn:Tay3} 
 and \eqref{eqn:AbrSigh}), the result of Theorem \ref{prop:H11} follows 
 from Theorems \ref{thm_H2} and \ref{prop:H10}.
\end{proof}
\begin{proof}[of Corollary \ref{cor1}]
Let $ W\Subset\{(z,w)\in\C^2; z\neq w\}$ be compact. 
Recall from the discussion at the beginning of Section \ref{se_SL} that, 
for $h>0$ small enough, 
\begin{equation*}
	\widetilde{\kappa}_h(z,w) : = \kappa^{\delta}(z_0+d_0^{-1/2}z,z_0+d_0^{-1/2}w;h),
\end{equation*}
is well defined for all $(z,w)\in W$, where 
$d_0:=d(z_0;h)\asymp h^{-1}$, see \eqref{eq_i30}, \eqref{eq_b1}. Using 
Theorems \ref{thm_H2} and \ref{prop:H11} we see that 
  \begin{equation*}
  \begin{split}
 \kappa^{\delta}(z_0+&d_0^{-1/2}z,z_0+d_0^{-1/2}w;h) \\
   & =
   \frac{(1+\mO(h))( (\sinh^2 K + (1 + \mO(h^{1/2}))K^2)\cosh K -(1 + \mO(h^{1/2}))2K\sinh K)
   }{\sinh^3 K}
	  \\
	  &\phantom{++}
	  +\frac{\mO\!\left(h^{\infty} + \delta h^{-\frac{32}{10}}\right)}{\left(1- \e^{-2K}\right)}
    + \mO\!\left(\e^{-\frac{D}{h^2}}\right),
 \end{split}
 \end{equation*}
 with 
 \begin{equation*}
 \begin{split}
   K&=K(z_0+d_0^{-1/2}z,z_0+d_0^{-1/2}w;h)\\ 
   &=
   \sigma_h(z_0+d_0^{-1/2}z,z_0+d_0^{-1/2}w)\frac{|z-w|^2}{4hd_0}
   (1+\mO(h^{1/2}))
   \\
   &= \frac{\pi}{2}|z-w|^2(1+\mO(h^{1/2})),
    \end{split}
 \end{equation*}
 where the error estimates are uniform in $W$. Here, we used as well that $d_0=(2\pi h)^{-1}\sigma(z_0)(1+\mO(h))$, cf. \eqref{eq_i30}, and that by Taylor expansion 
 $\sigma_h(z_0+d_0^{-1/2}z,z_0+d_0^{-1/2}w)=\sigma(z_0)(1+\mO(h^{1/2}))$. Taking 
 the limit $h\to0^+$ we conclude the statement of the Corollary.
\end{proof}
\providecommand{\bysame}{\leavevmode\hbox to3em{\hrulefill}\thinspace}
\providecommand{\MR}{\relax\ifhmode\unskip\space\fi MR }
\providecommand{\MRhref}[2]{%
  \href{http://www.ams.org/mathscinet-getitem?mr=#1}{#2}
}
\providecommand{\href}[2]{#2}

\end{document}